\documentclass[12pt]{amsart}

\usepackage{pstricks,graphicx,pst-node,amsmath,amssymb,url}
\newtheorem{theorem}{Theorem}[section]
\newtheorem{lemma}[theorem]{Lemma}
\newtheorem{proposition}[theorem]{Proposition}
\newtheorem{corollary}[theorem]{Corollary}
\newtheorem{conjecture}[theorem]{Conjecture}

\theoremstyle{definition}
\newtheorem{definition}[theorem]{Definition}

\theoremstyle{remark}
\newtheorem{remark}[theorem]{Remark}

\numberwithin{equation}{section}

\newcommand{\D}{D\ }   
\newcommand{\lsp}{\mathrm{LSP}}
\newcommand{\aSn}{\widetilde{\mathcal{S}}_{n}}
\newcommand{\Sn}{\mathcal{S}_{n}}

\newcommand{\outside}{outside}

\newcommand{\SYT}{\mathrm{SYT}}
\newcommand{\SST}{\mathrm{SST^{*}}}

\newcommand{\LLT}{\mathrm{LLT}}

\newcommand{\fl}{\mathrm{fl}}

\newcommand{\cl}{\mathrm{cl}}
\newcommand{\spin}{\mathrm{spin}}
\newcommand{\cospin}{\mathrm{cospin}}
\newcommand{\inv}{\mathrm{inv}}

\newcommand{\Inv}{\mathrm{Inv}}
\newcommand{\Des}{\mathrm{Des}}
\newcommand{\maj}{\mathrm{maj}}

\newcommand{\vp}{\ensuremath \varphi}
\newcommand{\swap}{\mathrm{swap}}
\newcommand{\basic}{\mathrm{bswap}}
\newcommand{\snake}{\mathrm{snake}}
\newcommand{\starswap}{\mathrm{star}}
\newcommand{\double}{\mathrm{double}}
\newcommand{\core}{\mathcal{C}}
\newcommand{\aff}{\mathcal{A}}
\newcommand{\Trans}[1]{T(#1)}

\newcommand{\G}{\mathcal{G}}

\newcommand{\chs}[2]{
\left(\begin{smallmatrix} #1 \\#2 
        \end{smallmatrix} \right)}
\newcommand{\stab}[3]{\begin{array}{c}\rnode{#1}{\tableau{#2}}\\\rnode{#1#1}{_{#3}}\end{array}}

\newlength\cellsize 
\savebox2{%
\begin{picture}(12,12)
\put(0,0){\line(1,0){12}}
\put(0,0){\line(0,1){12}}
\put(12,0){\line(0,1){12}}
\put(0,12){\line(1,0){12}}
\end{picture}}

\newcommand\cellify[1]{\def\thearg{#1}\def\nothing{}%
\ifx\thearg\nothing
\vrule width0pt height\cellsize depth0pt\else
\hbox to 0pt{\usebox2\hss}\fi%
\vbox to 12\unitlength{
\vss
\hbox to 12\unitlength{\hss$#1$\hss}
\vss}}

\newcommand\tableau[1]{\vtop{\let\\=\cr
\setlength\baselineskip{-12000pt}
\setlength\lineskiplimit{12000pt}
\setlength\lineskip{0pt}
\halign{&\cellify{##}\cr#1\crcr}}}

\newlength\smcellsize 
\savebox3{%
\begin{picture}(8,8)
\put(0,0){\line(1,0){8}}
\put(0,0){\line(0,1){8}}
\put(8,0){\line(0,1){8}}
\put(0,8){\line(1,0){8}}
\end{picture}}
\newcommand\smcellify[1]{\def\thearg{#1}\def\nothing{}%
\ifx\thearg\nothing
\vrule width0pt height\smcellsize depth0pt\else
\hbox to 0pt{\usebox3\hss}\fi%
\vbox to 8\unitlength{
\vss
\hbox to 8\unitlength{\hss$#1$\hss}
\vss}}
\newcommand\smtableau[1]{\vtop{\let\\=\cr
\setlength\baselineskip{-10000pt}
\setlength\lineskiplimit{10000pt}
\setlength\lineskip{0pt}
\halign{&\smcellify{##}\cr#1\crcr}}}

\newcommand{\e}{\mbox{}}

\begin{document}


\title[Affine dual equivalence and $k$-Schur functions]{Affine dual
  equivalence and $k$-Schur functions}

\author[S. Assaf]{Sami Assaf} %
\address{Department of Mathematics, Massachusetts Institute of
  Technology, 77 Massachusetts Avenue, Cambridge, MA 02139-4307}
\email{sassaf@math.mit.edu} 

\thanks{Sami Assaf acknowledges support from an NSF postdoctoral
fellowship.  Sara Billey acknowledges support from grants DMS-0800978
and DMS-1101017 from the National Science Foundation.}

\author[S. Billey]{Sara Billey} %
\address{Department of Mathematics, Padelford C-445, University of
  Washington, Box 354350, Seattle, WA 98195-4350 }
\email{billey@math.washington.edu}





\begin{abstract}
  The $k$-Schur functions were first introduced by Lapointe, Lascoux
and Morse \cite{LLM2003} in the hopes of refining the expansion of
Macdonald polynomials into Schur functions.  Recently, an alternative
definition for $k$-Schur functions was given by Lam, Lapointe, Morse,
and Shimozono \cite{LLMS} as the weighted generating function of
starred strong tableaux which correspond with labeled saturated chains
in the Bruhat order on the affine symmetric group modulo the symmetric
group.  This definition has been shown to correspond to the Schubert
basis for the affine Grassmannian of type $A$ \cite{Lam2008} and at
$t=1$ it is equivalent to the $k$-tableaux characterization of
Lapointe and Morse \cite{LaMo2007}.  In this paper, we extend Haiman's
dual equivalence relation on standard Young tableaux \cite{Haiman1992}
to all starred strong tableaux.  The elementary equivalence relations
can be interpreted as labeled edges in a graph which share many of the
properties of Assaf's dual equivalence graphs.  These graphs display
much of the complexity of working with $k$-Schur functions and the
interval structure on $\widetilde{S}_{n}/S_{n}$.  We introduce the
notions of flattening and squashing skew starred strong tableaux in
analogy with jeu da taquin slides in order to give a method to find
all isomorphism types for affine dual equivalence graphs of rank 4.
Finally, we make connections between $k$-Schur functions and both LLT
and Macdonald polynomials by comparing the graphs for these functions.
\end{abstract}

\maketitle


\section{Introduction}
\label{sec:intro}

Classically, the Schur functions have played a central role in the
theory of symmetric functions \cite{Macdonald1995}.  They also appear
in geometry as representatives for Schubert classes in the cohomology
rings of Grassmannian manifolds, and they appear in representation
theory as the Frobenius characteristics of irreducible $S_{n}$
representations and as the trace for certain irreducible $GL_{n}$
representations.   

In \cite{LLM2003}, Lapointe, Lascoux and Morse introduced a new larger
family of symmetric functions which includes the Schur functions,
namely the $k$-Schur functions, with similar connections both to
geometry and to representation theory. The $k$-Schur functions were
defined in hopes of refining and ultimately proving the Macdonald
Positivity Conjecture \cite{Macdonald1988}. Precisely, Lapointe
Lascoux and Morse conjectured that the Macdonald polynomials expand
into $k$-Schur functions with polynomial coefficients in two
parameters $q,t$ with nonnegative integer coefficients, and that the
$k$-Schur functions expand into Schur functions with polynomial
coefficients with parameter $t$ and nonnegative integer coefficients.
Haiman \cite{Haiman2001} has since shown that the Macdonald
polynomials are the Frobenius characteristic of a bigraded
$S_{n}$-module defined by Garsia and Haiman \cite{GaHa1993} using the
geometry of the Hilbert scheme of points in the plane. This resolved
the $n!$ Conjecture and provided the first proof of Macdonald
positivity. 

At this time, a number of conjecturally equivalent definitions for
$k$-Schur functions exist
\cite{LLMS,LLM2003,LaMo2003-2,LaMo2003-3,LaMo2005,LaMo2007}, making
the term ``$k$-Schur function'' rather ambiguous. In this paper, we
advocate for the geometrically inspired definition as the weighted
generating function of starred strong tableaux presented by Lam,
Lapointe, Morse and Shimozono \cite{LLMS}. This definition at $t=1$ is
equivalent to the $k$-tableaux characterization in \cite{LaMo2007}
which has been shown to represent the Schubert basis in the homology
of the affine Grassmannian of type $A$ \cite{Lam2008}.
Furthermore, the starred strong tableaux are a natural generalization
of standard tableaux which appear throughout combinatorics.

Recently, Lam, Lapoint, Morse and Shimozono proved that the $k$-Schur
functions as defined below except with $t=1$ are Schur positive
\cite{LLMS2010}.  Their approach shows how $k$-Schur functions relate
to $k+1$-Schur functions when the $t$ is not included.

It is an open problem to show that the $k$-Schur functions including
the $t$ statistic are Schur positive.  Toward proving this conjecture,
we define a family of involutions on starred strong tableaux which
generalize Haiman's elementary dual equivalence moves
 on standard Young tableaux \cite{Haiman1992}.  Using these involutions, one can put a graph
structure on starred strong tableaux which satisfies many of the same
axioms as the dual equivalence graphs defined by the first author in
\cite{Assaf2007-3}.  As our model for dual equivalence is based on the
poset on $n$-cores induced from Young's lattice, our results extend to
$k$-Schur functions indexed by skew shapes. Our main result is that
these graphs, which we call \textit{affine dual equivalence graphs},
are locally Schur positive when restricted to edges of 2 adjacent
colors and the spin is constant on connected components, see
Definition~\ref{defn:lsp} and Theorem~\ref{thm:Dgraph}.
\footnote{Earlier, we announced the stronger result that $k$-Schur
functions as defined here are Schur positive.  However, we have since
realized that the proof is incomplete for two reasons.  First, the
proof outline requires one to identify all isomorphism types for
3-colored components in affine dual equivalence graphs of the form.
Our computer verification relies on a halting problem which has not
terminated.  Second, the axiom (4') required in \cite{Assaf2007-3} is
not known to hold for affine dual equivalence graphs.}

Jeu da taquin is an important algorithm in the theory of symmetric
functions related to Littlewood-Richardson coefficients.  One of the
properties of jeu da taquin slides is that they commutes with
elementary dual equivalence moves on tableaux \cite[Lemma
2.3]{Haiman1992}.  There is no known analog of jeu da taquin for
$k$-Schur functions at this time.  Such an analog would in principle
be useful for multiplying $k$-Schur functions and expanding again into $k$-Schurs.  One approach to finding such a jeu da taquin algorithm
is to look for sliding moves which commute with affine dual
equivalence moves.  In Sections~\ref{sec:flattening.map}
and~\ref{sec:graph-cloning}, we describe two types of collapsing moves
which commute with affine dual equivalence in specified cases.  These
collapsing moves are the analogs of removing empty rows and columns in
a skew tableau via jeu da taquin.

One of the main consequences of our results is a connection between
$k$-Schur functions and LLT polynomials which is realized by an
isomorphism of graphs for the two functions in certain cases. More
generally, we expect that a better understanding of the connections
between the graph we construct for $k$-Schur functions and that for
LLT polynomials will ultimately show that an LLT polynomial expands
into $k$-Schur functions with coefficients that are polynomials in $t$
with nonnegative integer coefficients for an appropriate value of
$k$. Given Haglund's formula expanding Macdonald polynomials
positively into certain LLT polynomials \cite{Haglund2004,HHL2005},
this would also establish the missing connection between Macdonald
polynomials and $k$-Schur functions.

The outline of the paper goes as follows.  In Section~\ref{sec:defs},
we review the basic vocabulary on partitions, the affine symmetric
group, symmetric functions and quasisymmetric functions.  In
particular, we review an interesting order preserving bijection
between a quotient of the affine symmetric group with the $n$-core
partitions relating Bruhat order to a subposet of Young's lattice. In
Section~\ref{sec:kschur}, one definition of $k$-Schur functions
expanded into fundamental quasisymmetric functions is given following
\cite[Conjecture 9.11]{LLMS}.  These functions can be indexed by
$n$-cores, minimal length coset representatives for $\aSn/\Sn$, or
$k=n-1$ bounded partitions since all three sets are in bijection.  In
Section~\ref{sec:degs}, we review dual equivalence on standard Young
tableaux along with the associated graph structures and axioms.  In Section~\ref{sec:poset}, we carefully study the covering
relations and the rank two intervals in the poset on $n$-core
partitions.  In Section~\ref{sec:equivalence}, we define the affine
analog of dual equivalence operations and prove these maps are
involutions.  The main theorem is proved at the end of
Section~\ref{sec:graph}.  Here we also give our definition of the
affine dual equivalence graph on starred strong tableaux of a given
shape.  In Section~\ref{sec:LLT}, we describe the connections between
$k$-Schur functions and both the LLT polynomials and Macdonald
polynomials.  We encourage the reader to look ahead to this section
after seeing the definition of $k$-Schur functions in
Section~\ref{sec:kschur} in order to see the similarities.  Finally,
in the Appendix, we have included some examples of $k$-Schur functions
expanded both in quasisymmetric functions and Schur functions along
with their affine dual equivalence graphs.

%
%

\begin{center}
{\sc Acknowledgments}
\end{center}

We would like to thank Nantel Bergeron, Andrew Crites, Adriano Garsia, Jim Haglund,
Mark Haiman, Steve Mitchell, Jennifer Morse, Austin Roberts and Mike Zabrocki for
inspiring conversations on this topic.

\section{Basic definitions and notations}
\label{sec:defs}

\subsection{Partitions}
\label{sec:defs-parts}

A {\em partition} $\lambda$ is a weakly decreasing sequence of
non-negative integers
$$
\lambda = (\lambda_1,\lambda_2, \ldots, \lambda_l), \;\;\;\;\;
\lambda_1 \geq \lambda_2 \geq \cdots \geq \lambda_l > 0 .
$$
The {\em Young diagram} of a partition $\lambda$ is the set of points
$(i,j)$ in $\mathbb{N} \times \mathbb{N}$ such that $1 \leq i \leq
\lambda_j$. We draw the diagram so that each point $(i,j)$ is
represented by the unit cell southwest of the point.  Abusing
notation, we will write $\lambda$ for both the partition and its
diagram. For example, the diagram of $(4,3,1)$ is 
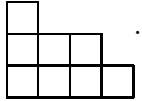
\begin{displaymath}
  \tableau{\e \\ \e & \e & \e \\ \e & \e & \e & \e}.
\end{displaymath}
We may also represent $\lambda$ by an infinite binary string as
follows. Consider the diagram of $\lambda$ lying in the $\mathbb{N}
\times \mathbb{N}$ plane with infinite positive axes. Walk in unit
steps along the boundary of $\lambda$, writing $1$ for each vertical
step and $0$ for each horizontal step.  For example, $(4,3,1)$ becomes
\begin{displaymath}
  \cdots \ 1 \ 1 \ 1 \ 0 \ 1 \ 0 \ 0 \ 1 \ 0 \ 1 \ 0 \ 0 \ 0 \cdots .
\end{displaymath}
Note that this establishes a bijective correspondence between
partitions and doubly infinite binary strings $s$ such that $s_i = 1$
for all $i<l$ and $s_i=0$ for all $i>r$ for some $l,r \in \mathbb{Z}$.

For partitions $\lambda,\mu$, we write $\mu \subset \lambda$ whenever
the diagram of $\mu$ is contained within the diagram of $\lambda$;
equivalently $\mu_i \leq \lambda_i$ for all $i$. \emph{Young's
  lattice} is defined by the partial ordering on partitions given by
containment.

A \emph{standard Young tableau of shape $\lambda$} is a saturated
chain in Young's lattice from the empty partition to $\lambda$. As
moving from rank $i-1$ to rank $i$ adds a single box, filling this
added box with the letter $i$ uniquely records the chosen
chain. Therefore standard Young tableaux are also characterized as
bijective fillings of the cells of $\lambda$ with the letters $1$ to
$m$ so that entries increase along rows and up columns. 
Let $\SYT(\lambda)$ denote the
set of all standard Young tableaux of shape $\lambda$, and let $\SYT$
denote the union of all $\SYT(\lambda)$.  
For example, a
standard tableau of shape $(4,3,1)$ is
\begin{equation}\label{ex:std.tab}
  \tableau{6 \\ 2 & 5 & 8 \\ 1 & 3 & 4 & 7}. 
\end{equation}

When $\mu \subset \lambda$, we may define the {\em skew diagram}
$\lambda / \mu$ to be the set theoretic difference $\lambda - \mu$. A
standard tableau of skew shape $\lambda/\mu$ is a saturated chain in
Young's lattice from $\mu$ to $\lambda$, or, equivalently, a bijective
filling of the cells of $\lambda / \mu$ with entries $1$ to $m$ so
that entries increase along rows and up columns.

An \emph{addable} cell for a partition $\lambda$ is any cell $c$
such that $c \cup \lambda $ is again a Young diagram of a partition.
Similarly, a \emph{removable} cell for a partition $\lambda$ is any
cell $c$ such that $\lambda - c $ is again a Young diagram of a
partition.

A {\em connected skew diagram} is one where exactly one cell has no
cell immediately north or west of it, and exactly one cell has no cell
immediately south or east of it.  Two distinct connected components
can meet at one point but not along an edge of a cell.  A connected
skew diagram is necessarily nonempty. A \emph{ribbon} is a connected
skew diagram containing no $2 \times 2$ subdiagram. We may define
\emph{addable} and \emph{removable ribbons} of $\lambda$ just as
with cells; namely, a ribbon $R$ is an addable (resp. removable)
ribbon for a partition $\lambda$ if $\lambda \cup R$ (resp. $\lambda -
R$) is again a partition.

To each cell $x$ of a diagram $\lambda$ associate the \emph{content of
$x$} defined by $c(x) = i-j$ where the cell $x$ lies in row $j$ and
column $i$. We also consider the \emph{residue of $x$}, defined as the
content of $x$ modulo $n$. The content and residue of ribbons are
defined with respect to the southeasternmost cell.  The \textit{head}
of a ribbon is its southeasternmost cell, and the \textit{tail} of a
ribbon is its \textit{northwesternmost} cell.   

The \emph{hook length} of $x$ is the number of squares above and to
the right of $x$ in $\lambda$ including $x$ itself.  Define the
\emph{bandwidth} of a partition to be the number of distinct
contents occupied by its cells. Equivalently, the bandwidth of a
non-skew partition is its maximum hook length.

An \emph{$n$-core} is a partition having no removable ribbon of
length $n$. Equivalently, no hook length of $\lambda$ is divisible by
$n$.  Young's lattice restricted to $n$-cores gives another ranked
partial order, but it is not a lattice.  This partial order on
$n$-cores is central to the definition of $k$-Schur functions and
strong tableaux given in Section~\ref{sec:kschur}.

\subsection{Affine permutations}
\label{sec:defs-perms}

Here we briefly recall the necessary vocabulary on affine
permutations.  For a more thorough treatment of the combinatorial
aspects of Coxeter groups we recommend \cite{BjBr1996}, specifically
see Section 8.3 for details on the affine symmetric group.  Recent
developments on core partitions and connections to affine Weyl groups
can be found in \cite{BJV,Hanusa-Jones}.

Given $n$, consider the set $\aSn$ of all bijections $w:\mathbb{Z}
\longrightarrow \mathbb{Z}$ such that 
\begin{displaymath}
  w(i+n)=w(i)+n \ \forall i \in \mathbb{Z} \hspace{1em} \mbox{and} \hspace{1em} 
  w(1)+w(2)+\dotsb +w(n)=\chs{n+1}{2}.
\end{displaymath}
For example, given $i,j \in \mathbb{Z}$ such that $i \not \equiv j$
(all congruences should be taken modulo $n$ throughout the paper), the
\emph{affine transposition} $t_{i,j} \in \aSn$ is the periodic
bijection such that $t_{i,j}(i+p\cdot n)=j+p \cdot n$,
$t_{i,j}(j+p\cdot n)=i+p\cdot n$, and $t_{i,j}(k)=k$ for all $k\not
\equiv i$ and $k\not \equiv j$ and all $p \in \mathbb{Z}$.  $\aSn$ is known as the \emph{affine
  symmetric group}.  It is the affine Weyl group of type $A_{n-1}$.
As a Coxeter group, $\aSn$ is generated by the adjacent transpositions
$s_{i}=t_{i,i+1}$ for $0\leq i<n$.  If $w=s_{i_{1}}s_{i_{2}}\cdots
s_{i_{p}} \in \aSn$ and $p$ is minimal among all such expressions for
$w$, then $s_{i_{1}}s_{i_{2}}\cdots s_{i_{p}}$ is a \emph{reduced
  expression} for $w$ and the \emph{length} of $w$ is $p$, denoted
$\ell(w)=p$.  The length function is the rank function for the
\emph{Bruhat order} on $\aSn$.  As a partial order, Bruhat order can
be described as the transitive closure of the relation $w < t_{i,j} w$
if $\ell(w) < \ell(t_{i,j}w)$.  The symmetric group $\Sn$ can be
viewed as the parabolic subgroup of $\aSn$ generated by $s_{1},\dotsc,
s_{n-1}$.

Let $Q_{n}$ be the minimal length coset representatives for the
quotient $\aSn/\Sn$.  Bruhat order restricted to $Q_{n}$ is again a
partial order ranked by the length function. There is a rank
preserving bijection from $n$-core partitions to $Q_{n}$ which
respects the Bruhat order.  This correspondence leads to  useful
criteria for Bruhat order on $Q_{n}$ in Theorem~\ref{t:lascoux} and
the covering relation in Proposition~\ref{p:rod.cover} and
Corollary~\ref{cor:iribbons}.  We follow \cite{ec1} for terminology on
partial orders.

\begin{definition}\label{defn:aff.to.n-cores}
\cite{JaKe1981,LaMo2005,Misra-Miwa} 
Define the function 
\begin{equation}\label{e:aff.to.n-cores}
\core: Q_{n} \longrightarrow n\text{-core partitions}
\end{equation}
recursively as follows.  Associate the empty partition with the
identity in $Q_{n}$; namely, $\core(\mathrm{id})=\emptyset$.  Say
$\core(w)=\lambda$ and $\ell(s_{i}w)> \ell(w)$, then $\core(s_{i}w)$
is obtained from $\lambda $ by adding every addable cell with residue
$i$ to $\lambda$. 
\end{definition}
In \cite{JaKe1981,LaMo2005}, $\core$ is shown to be a bijection.  
Denote $\core^{-1}$ by \begin{equation}\label{e:core.to.aff}
  \aff: n\text{-core partitions} \longrightarrow Q_{n}.
\end{equation}

\begin{remark}\label{r:algorithm}
Definition~\ref{defn:aff.to.n-cores} can be used as an algorithm for
generating $n$-core partitions.  The reader is encouraged to look
ahead to Figure~\ref{fig:poset} to see how the $3$-core partitions up
to length 4 are generated.
\end{remark}

Note, if $\lambda$ is an $n$-core with an addable cell of residue $i$,
then $\lambda$ has no removable cells of residue $i$.
Similarly, if $\lambda$ is an $n$-core with a removable cell of
residue $i$, then $\lambda$ has no addable cells of
residue $i$ \cite[\S 5]{LaMo2005}.

The following beautiful theorem of Lascoux shows the power of the
$n$-core model for $Q_{n}$.

\begin{theorem}\label{t:lascoux}\cite{Lascoux1999}
Given $v,w \in Q_{n}$, let $\mu = \core(v)$ and $\lambda = \core(w)$
be the corresponding $n$-core partitions.  Then $\mu \subset \lambda $
in Young's lattice if and only if $v<w$ in Bruhat order restricted to
$Q_{n}$.  
\end{theorem}

\subsection{Symmetric and Quasisymmetric functions}
\label{sec:defs-syms}

We adopt notations for the standard bases for $\Lambda$, the ring of
symmetric functions, from \cite{Macdonald1995}. For this paper, we are
primarily interested in the Schur functions $s_{\lambda}$, indexed by
partitions. The Schur functions form an orthonormal basis for
$\Lambda$ with the Hall scalar product. The Schur functions also give
the irreducible characters for representations of the general linear
group as well as the Schubert basis for the cohomology of the
Grassmannian \cite{Fulton-book}. The $k$-Schur functions have
analogous interpretations for each of these viewpoints.

We will use the expansion for Schur functions in terms of Gessel's
fundamental quasisymmetric functions \cite{Gessel1984} rather than in
terms of monomials on an alphabet $X=\{x_{1},x_{2},\ldots\}$.  The
$k$-Schur functions will have a similar expansion, presented in
Section~\ref{sec:kschur-quasi}.

\begin{definition}
  For $\sigma \in \{\pm 1\}^{m-1}$, the \emph{fundamental
    quasisymmetric function associated to $\sigma$}, denoted
  $Q_{\sigma}$, is given by
  \begin{equation}
    Q_{\sigma}(X) = \sum_{\substack{i_1 \leq \cdots \leq i_m \\ i_j =
        i_{j+1} \Rightarrow \sigma_j = +1}} x_{i_1} \cdots x_{i_m} .
    \label{eqn:quasisym}
  \end{equation}
  \label{defn:quasisym}
\end{definition}

To connect quasisymmetric functions with Schur functions, for $T$ a
standard tableau on $1,\ldots,m$, define the {\em descent signature}
$\sigma(T) \in \{\pm 1\}^{m-1}$ by
\begin{equation}
  \sigma_{i}(T) \; = \; \left\{ 
    \begin{array}{ll}
      +1 & \; \mbox{if the content of $i$ is less than the content of $i+1$} \\
      -1 & \; \mbox{if the content of $i+1$ is less than the content of $i$.}
    \end{array} \right\}
\label{eqn:sigma}
\end{equation}
Note that in a standard tableau, consecutive entries may never appear
along the same diagonal so the content of the cells containing $i$ and $i+1$ are never equal. In particular, $\sigma$ is well-defined on
$\SYT$. 

\begin{theorem}\cite{Gessel1984}
  The Schur function $s_{\lambda}$ can be expressed in terms of
  quasisymmetric functions by
  \begin{equation}\label{e:quasi-s}
    s_{\lambda}(X) = \sum_{T \in \SYT(\lambda)} Q_{\sigma(T)}(X) .
  \end{equation}
\label{thm:quasisym}
\end{theorem}

By Theorem~\ref{thm:quasisym}, working with quasisymmetric functions
instead of monomials affords us the benefit of working with standard
objects instead of semistandard objects.  Furthermore, the expansion
in \eqref{e:quasi-s} is independent of the size of the alphabet $X$
which could be finite or infinite.

\section{$k$-Schur functions}
\label{sec:kschur}

In this section, we recall two analogs of standard Young tableaux for
the $n$-core poset called strong tableaux and starred strong tableaux
from \cite{LLMS}.  The spin statistic is defined on starred strong
tableaux.  These ingredients are combined to give the definition of
$k$-Schur functions in terms of their expansion into fundamental
quasisymmetric functions.

\subsection{Strong tableaux}
\label{sec:kschur-sst}

Consider the poset on $n$-core partitions induced from Young's
lattice. A \emph{strong tableau of shape $\lambda$} is a saturated
chain
\begin{displaymath}
  \emptyset \subset \lambda^{(1)} \subset \lambda^{(2)} \subset \cdots
  \subset \lambda^{(m)} = \lambda  
\end{displaymath}
in the $n$-core poset from the empty tableau to $\lambda$. We denote
this chain by the filling $S$ of $\lambda$ where all cells of
$\lambda^{(i)} / \lambda^{(i-1)}$ contain the letter $i$.

\begin{figure}[ht]
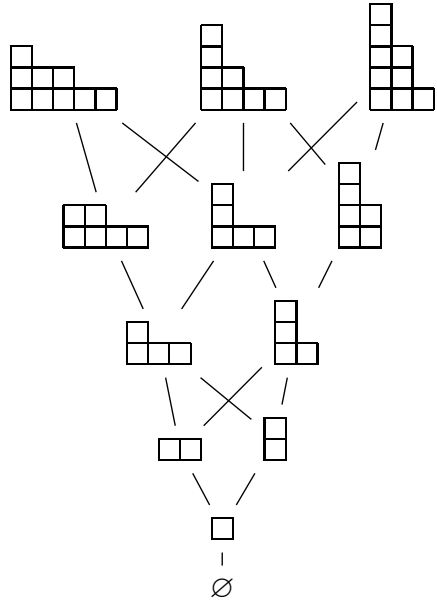

  \begin{center}
    \begin{displaymath}
      \begin{array}{c}
        \rnode{531}{\smtableau{\e \\ \e & \e & \e \\ \e & \e & \e &\e & \e}} 
        \hspace{4\smcellsize}
        \rnode{4211}{\raisebox{\smcellsize}{\smtableau{\e \\ \e \\ \e & \e \\
              \e & \e & \e & \e}}} 
        \hspace{4\smcellsize}
        \rnode{32211}{\raisebox{2\smcellsize}{\smtableau{\e \\ \e \\ \e &
              \e \\ \e & \e \\ \e & \e & \e}}} 
        \\[4\smcellsize]
        \rnode{42}{\smtableau{\e & \e \\ \e & \e &\e & \e}} 
        \hspace{3\smcellsize}
        \rnode{311}{\raisebox{\smcellsize}{\smtableau{\e \\ \e \\ \e & \e
              & \e}}} 
        \hspace{3\smcellsize}
        \rnode{2211}{\raisebox{2\smcellsize}{\smtableau{\e \\ \e \\ \e &
              \e \\ \e & \e}}} 
        \\[3\smcellsize]
        \rnode{31}{\smtableau{\e \\ \e & \e & \e}} 
        \hspace{4\smcellsize}
        \rnode{211}{\raisebox{\smcellsize}{\smtableau{\e \\ \e \\ \e &
              \e}}} 
        \\[3\smcellsize] 
        \rnode{2}{\smtableau{\e & \e}}
        \hspace{3\smcellsize}
        \rnode{11}{\raisebox{\smcellsize}{\smtableau{\e \\ \e}}} 
        \\[2\smcellsize]
        \rnode{1}{\smtableau{\e}} 
        \\[\smcellsize]
        \rnode{0}{\varnothing}
      \end{array}
      \psset{nodesep=5pt,linewidth=.1ex}
      \ncline{531}{42}
      \ncline{531}{311}
      \ncline{4211}{42}
      \ncline{4211}{311}
      \ncline{4211}{2211}
      \ncline{32211}{311}
      \ncline{32211}{2211}
      \ncline{42}{31}
      \ncline{311}{31}
      \ncline{311}{211}
      \ncline{2211}{211}
      \ncline{31}{2}
      \ncline{31}{11}
      \ncline{211}{2}
      \ncline{211}{11}
      \ncline{2}{1}
      \ncline{11}{1}
      \ncline{1}{0}
    \end{displaymath}
    \caption{\label{fig:poset} Poset of $3$-cores up to rank $5$.}
  \end{center}
\end{figure}

For example, from Figure~\ref{fig:poset}, the strong tableaux for
$n=3$ of size $m=4$ are
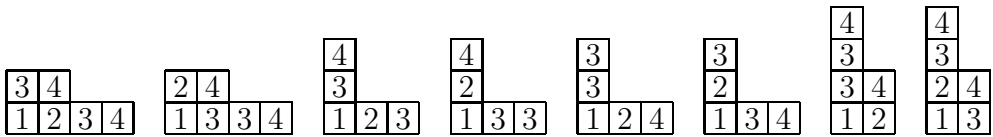
\begin{displaymath}
  \tableau{3 & 4 \\ 1 & 2 & 3 & 4}  \hspace{\cellsize}
  \tableau{2 & 4 \\ 1 & 3 & 3 & 4}  \hspace{\cellsize}
  \raisebox{\cellsize}{\tableau{4 \\ 3 \\ 1 & 2 & 3}}     \hspace{\cellsize}
  \raisebox{\cellsize}{\tableau{4 \\ 2 \\ 1 & 3 & 3}}     \hspace{\cellsize}
  \raisebox{\cellsize}{\tableau{3 \\ 3 \\ 1 & 2 & 4}}     \hspace{\cellsize}
  \raisebox{\cellsize}{\tableau{3 \\ 2 \\ 1 & 3 & 4}}     \hspace{\cellsize}
  \raisebox{2\cellsize}{\tableau{4 \\ 3 \\ 3 & 4\\ 1 & 2}} \hspace{\cellsize}
  \raisebox{2\cellsize}{\tableau{4 \\ 3 \\ 2 & 4\\ 1 & 3}}
\end{displaymath}

\subsection{Starred strong tableaux}
\label{sub:starred.strong.tableaux}

A \emph{starred strong tableau}, $S^*$, is a strong tableau $S$
where one connected component of the cells containing $i$ is chosen
for each $i$, and the southeasternmost cell of the chosen components
are adorned with a $*$.  Therefore, the information contained in
$S^{*}$ is equivalent to the pair $(S,c^{*})$ where
$c^{*}=(c_{1},c_{2},\dotsc, c_{m})$ is the \emph{content vector},
namely $c_{i}$ is the content of the cell containing $i^{*}$.  

Let $\SST(\lambda,n)$ be the set of all starred strong tableaux of
shape $\lambda$ regarded as an $n$-core.  For example, the 6 starred
strong tableaux of shape $\lambda = (2,2,1,1)$ are
\begin{equation}
  \tableau{4^* \\ 3^* \\ 2^* & 4 \\ 1^* & 3} \hspace{2\cellsize}  
  \tableau{4^* \\ 3 \\ 2^* & 4 \\ 1^* & 3^*} \hspace{2\cellsize}
  \tableau{4 \\ 3^* \\ 2^* & 4^* \\ 1^* & 3} \hspace{2\cellsize}
  \tableau{4 \\ 3 \\ 2^* & 4^* \\ 1^* & 3^*} \hspace{2\cellsize} 
  \tableau{4^* \\ 3 \\ 3^* & 4 \\ 1^* & 2^*} \hspace{2\cellsize} 
  \tableau{4 \\ 3 \\ 3^* & 4^* \\ 1^* & 2^*}
\label{eqn:2211}
\end{equation}

The following statistics on a starred strong tableau $S^*$ were first
introduced in \cite{LLMS}. Let $n(i)$ denote the number of connected
components of the cells containing $i$ of the underlying tableau
$S$. Among such connected components, let $h(i)$ be the height,
i.e. number of rows, of the starred connected component. Finally, let
$d(i^*)$ denote the \emph{depth of $i^*$ in $S^*$}, defined to be the
number of components northwest of the component containing
$i^*$. Define the statistic $\spin$ on starred strong tableaux as
follows,
\begin{equation}
  \spin(S^*) = \sum_{i} n(i) \cdot (h(i) - 1) + d(i^*).
\label{eqn:spin}
\end{equation}
For example, the spins of the starred strong tableaux in equation
\eqref{eqn:2211}, from left to right, are $0,1,1,2,1,2$.

 
This spin statistic was dubbed ``spin'' based on similarities with the
spin statistic on ribbon tableaux that gives LLT polynomials
\cite{LLT1997}. We explore deeper connections between LLT polynomials
and $k$-Schur functions in Section~\ref{sec:LLT}.

\subsection{Quasisymmetric expansion}
\label{sec:kschur-quasi}

The $k$-Schur function $s_{\lambda}^{(k)}(X;t)$ is the weighted
generating function of starred strong tableaux of shape
$\rho(\lambda)$, where $\rho$ is the bijection between $k$-bounded
partitions and $k+1$-cores introduced in \cite{LaMo2007}.  In was also
shown that the rank of $\rho(\lambda)$ in the $n$-core poset equals
$|\lambda|$ and it is conjectured that the leading term of
$s_{\lambda}^{(k)}(X;t)$ in the Schur function expansion is
$s_{\lambda}(X)$.

To define $\rho$ on a $k$-bounded partition $\lambda$, from north to
south slide each row of $\lambda$ east as far as necessary so that no
cell has hook length greater than $k$. Filling in the resulting skew
diagram gives $\rho(\lambda)$. To go back, remove all cells of
$\rho(\lambda)$ with hook length greater than $k$ and re-align the
rows with the western boundary. For example, we compute
$\rho(3,3,2,1,1) = (5,4,2,1,1)$ when $k=4$ as follows.
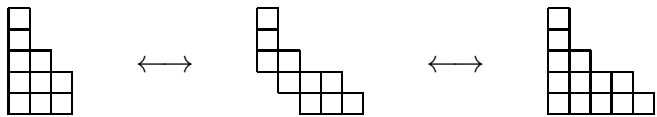
\begin{displaymath}
  \smtableau{\e \\ \e \\ \e & \e \\ \e & \e & \e \\ \e & \e & \e}
  \hspace{2\cellsize} \raisebox{-2\smcellsize}{$\longleftrightarrow$} 
  \hspace{2\cellsize}
  \smtableau{\e \\ \e \\ \e & \e \\ & \e & \e & \e \\ & & \e & \e & \e}
  \hspace{2\cellsize} \raisebox{-2\smcellsize}{$\longleftrightarrow$} 
  \hspace{2\cellsize}
  \smtableau{\e \\ \e \\ \e & \e \\ \e & \e & \e & \e \\ \e & \e & \e & \e & \e}
\end{displaymath}
Throughout this paper, we fix $n = k+1$ so that
we relate $n$-cores with $k$-Schur functions.

Rather than defining a semi-standard analog of strong tableaux to
define the expansion in terms of monomials as was given in
\cite{LLMS}, we formulate the definition in terms of (standard)
starred strong tableaux using quasisymmetric functions.  The two
versions of the definition are easily seen to be equivalent. We begin
by defining the \emph{descent signature}, $\sigma \in \{\pm
1\}^{m-1}$, of a starred strong tableau $S^*$ of rank $m$ as follows.

\begin{equation}
  \sigma_{i}(S^*) = \left\{ \begin{array}{rl}
      +1 & \mbox{if the content of $i^*$ is less than the content of
        $(i+1)^*$} \\
      -1 & \mbox{if the content of $i^*$ is greater than the content
        of $(i+1)^*$} 
    \end{array} \right.
    \label{eqn:sigmaS}
\end{equation}

\begin{remark}
  Since the union of cells containing $i$ and those containing $i+1$
  must be a valid skew shape, the southeasternmost cells containing
  $i$ and $i+1$ may not lie on the same diagonal. Therefore $\sigma$
  is well-defined for all starred strong tableaux.
\label{rmk:content}
\end{remark}

\begin{definition}
  Let $\nu$ be a $k$-bounded partition. The \emph{$k$-Schur
    function} indexed by $\nu$ is given by
  \begin{equation}
    s_{\nu}^{(k)}(X;t) = \sum_{S^* \in \SST(\rho(\nu),n)} t^{\spin(S^*)}
    Q_{\sigma(S^*)}(X) ,
    \label{eqn:kschur}
  \end{equation}
  where the sum is over all standard starred strong tableaux of shape
  $\rho(\nu)$ in the $n=k+1$-core poset.
  \label{defn:kschur}
\end{definition}

\begin{remark}
  We may extend Definition~\ref{defn:kschur} to \emph{skew} strong
  tableaux in the obvious way by considering all saturated chains from
  an $n$-core $\mu$ to an $n$-core $\nu$. The definitions for starred
  strong tableaux and spin extend trivially to this
  setting. Consequently, all of our results for $k$-Schur functions
  also extend to this skew setting.
\end{remark}

\section{Dual equivalence}
\label{sec:degs}

The main idea behind a dual equivalence graph, introduced in
\cite{Assaf2007-2}, is to provide a structure whereby the
quasisymmetric functions contributing to a single Schur function are
grouped together into equivalence classes, thereby demonstrating the
Schur positivity of the given quasisymmetric expansion. For standard
Young tableaux, the desired classes are precisely the \emph{dual
equivalence classes} defined by Haiman \cite{Haiman1992}. An abstract
\emph{dual equivalence graph} is defined by modeling the internal
structure of these classes using Haiman's \emph{elementary dual
equivalence relations}. The connected components of a dual equivalence
graph are exactly the desired equivalence classes, namely the sum over
the quasisymmetric functions in a given connected component is equal
to a single Schur function. Dual equivalence graphs, and more
generally \D graphs, provide a structure whereby we may extend the
notion of dual equivalence to more general objects, in our case,
starred strong tableaux.

\subsection{Dual equivalence on standard Young tableaux}
\label{sec:degs-tableaux}

We begin by constructing a graph on standard tableaux using dual
equivalence. Originally, Haiman defined an {\em elementary dual
 equivalence} on three consecutive letters $i-1,i,i+1$ of a
permutation by switching the outer two letters whenever the middle
letter is not $i$:
\begin{equation}\label{e:dual.equivalence}
 \cdots\; i \;\cdots\;i\pm 1\;\cdots\;i\mp 1\;\cdots\; \cong
 \; \cdots\;i\mp 1\;\cdots\;i\pm 1\;\cdots\; i \;\cdots .
\end{equation}
In Equation~\eqref{e:dual.equivalence}, $i\pm 1$ acts as a
\emph{witness} to the $i,i\mp 1$ exchange ensuring they are not
adjacent letters in the permutation.

The definition of dual equivalence extends naturally to standard Young
tableaux by applying the action to the permutation obtained by reading
the entries along content lines.  For example, the content reading
word of the standard tableau in \eqref{ex:std.tab} is 62153847.  Note
that in a standard tableau, $i$ and $j$ may lie on the same content
line only if $|i-j| \geq 3$. In particular, each of $i-1, i$ and $i+1$
must lie on distinct content lines, making
equation~\eqref{e:dual.equivalence} well-defined on standard tableaux.

It will also be helpful to think of dual equivalence on standard
tableaux in terms of Young's lattice. Recall, that a standard tableau
is equivalent to a saturated chain in Young's lattice with the empty
partition as its minimal element.  If two standard tableaux $S$ and
$T$ are dual equivalent via an elementary dual equivalence on
$i-1,i,i+1$, then the length two interval corresponding to the
addition of $i$ and the further away of $i-1$ and $i+1$ will be the
Boolean poset on subsets of $\{1,2 \}$ ordered by containment, denoted
$B_{2}$. Indeed, any length two interval in Young's lattice is either
isomorphic to $B_{2}$ or a chain. In this paradigm, exchanging $i,i\mp
1$ is equivalent to traversing the length two interval where these
cells are added using the other saturated chain in the interval.

We say that two standard tableaux are \emph{dual equivalent} if one
can be obtained from the other by a sequence of elementary dual
equivalences.  The following theorem of Haiman \cite{Haiman1992}
together with Theorem~\ref{thm:quasisym} show that the sum over the
quasisymmetric functions in a dual equivalence class of standard
tableaux is precisely a Schur function.

\begin{theorem}\cite{Haiman1992}
  Two standard tableaux of partition shape are dual equivalent if and
  only if they have the same shape.
\label{prop:shape}
\end{theorem}

Enrich the structure of these equivalence classes by tracking the
sequence of elementary dual equivalences taking one tableau to
another. Whenever $T$ and $U$ differ by an elementary dual equivalence
for $i-1,i,i+1$, connect $T$ and $U$ with an edge colored by
$i$. Additionally, we track the quasisymmetric function corresponding
to the given tableau by writing the descent signature $\sigma(T)$,
defined in Equation \eqref{eqn:sigma}, below each tableaux. Let
$\G_{\lambda}$ denote the graph on all standard tableaux of shape
$\lambda$.  See Figure~\ref{fig:G5} for examples of $\G_{\lambda}$.  

\begin{figure}[ht]
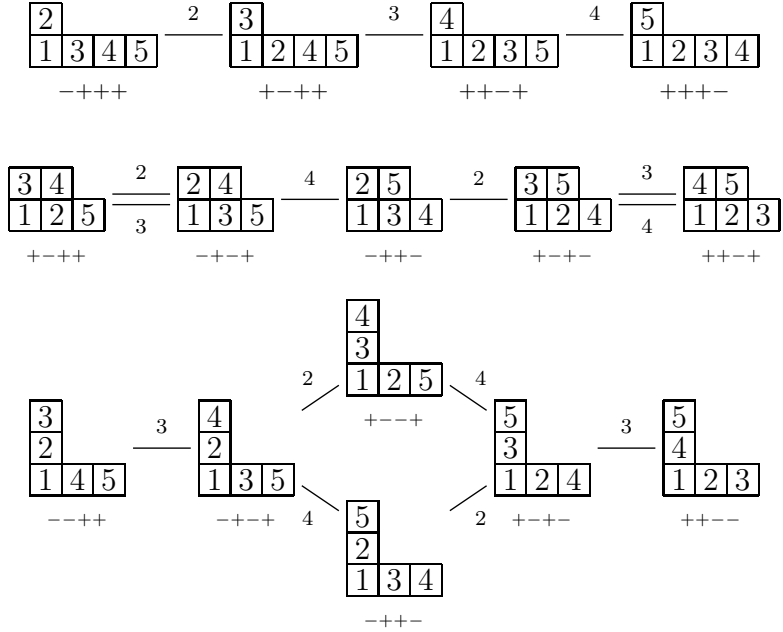

 \begin{center}
   \begin{displaymath}
     \begin{array}{c}
       \begin{array}{cccc}
         \stab{h}{2 \\ 1 & 3 & 4 & 5}{-+++} \ & \
         \stab{i}{3 \\ 1 & 2 & 4 & 5}{+-++} \ & \
         \stab{j}{4 \\ 1 & 2 & 3 & 5}{++-+} \ & \
         \stab{k}{5 \\ 1 & 2 & 3 & 4}{+++-}
       \end{array} \\[3\cellsize]
       \psset{nodesep=3pt,linewidth=.1ex}
       \everypsbox{\scriptstyle}
       \ncline  {h}{i} \naput{2}
       \ncline  {i}{j} \naput{3}
       \ncline  {j}{k} \naput{4}
       \begin{array}{ccccc}
         \stab{a}{3 & 4 \\ 1 & 2 & 5}{+-++} \ & \
         \stab{b}{2 & 4 \\ 1 & 3 & 5}{-+-+} \ & \
         \stab{c}{2 & 5 \\ 1 & 3 & 4}{-++-} \ & \
         \stab{d}{3 & 5 \\ 1 & 2 & 4}{+-+-} \ & \
         \stab{e}{4 & 5 \\ 1 & 2 & 3}{++-+}
       \end{array} \\[2\cellsize]
       \psset{nodesep=3pt,linewidth=.1ex}
       \everypsbox{\scriptstyle}
       \ncline[offset=2pt] {a}{b} \naput{2}
       \ncline[offset=2pt] {b}{a} \naput{3}
       \ncline             {b}{c} \naput{4}
       \ncline             {c}{d} \naput{2}
       \ncline[offset=2pt] {d}{e} \naput{3}
       \ncline[offset=2pt] {e}{d} \naput{4}
       \begin{array}{ccccc}
         & & \stab{w}{4 \\ 3 \\ 1 & 2 & 5}{+--+} & & \\[-1\cellsize]
             \stab{u}{3 \\ 2 \\ 1 & 4 & 5}{--++} \ & \
             \stab{v}{4 \\ 2 \\ 1 & 3 & 5}{-+-+} & &
             \stab{y}{5 \\ 3 \\ 1 & 2 & 4}{+-+-} \ & \
             \stab{z}{5 \\ 4 \\ 1 & 2 & 3}{++--} \\[-1\cellsize]
         & & \stab{x}{5 \\ 2 \\ 1 & 3 & 4}{-++-} & &
       \end{array}
       \psset{nodesep=3pt,linewidth=.1ex}
       \everypsbox{\scriptstyle}
       \ncline {u}{v}  \naput{3}
       \ncline {v}{w}  \naput{2}
       \ncline {v}{x}  \nbput{4}
       \ncline {w}{y}  \naput{4}
       \ncline {x}{y}  \nbput{2}
       \ncline {y}{z}  \naput{3}
     \end{array}
   \end{displaymath}
   \caption{\label{fig:G5}The standard dual equivalence graphs
     $\G_{(4,1)}, \G_{(3,2)}$ and $\G_{(3,1,1)}$.}
 \end{center}
\end{figure}

Define the generating function associated to $\G_{\lambda}$ by
\begin{equation}
 \sum_{v \in V(\G_{\lambda})} Q_{\sigma(v)}(X) = s_{\lambda}(X).
\label{eqn:glamschur}
\end{equation}
In particular, the generating function of any vertex-signed graph
whose connected components are all isomorphic to some $\G_{\lambda}$ is
automatically Schur positive.

\subsection{Dual equivalence graphs and \D graphs}
\label{sec:degs-graphs}

Given any collection of objects with an associated signature function,
the goal is to build a graph on the given objects that mimics the
structure of these $\G_{\lambda}$. To facilitate this, we recall the
local characterization of dual equivalence graphs presented in
\cite{Assaf2007-2}. First, we need a bit of terminology.

A {\em signed, colored graph of degree $m$} consists of the following
data: a vertex set $V$; a signature function $\sigma : V \rightarrow
\{\pm 1\}^{m-1}$; and for each $1 < i < m$, a collection $E_i$ of
unordered pairs of vertices of $V$ that represents the edges colored
$i$. We denote such a graph by $\G = (V,\sigma,E_{2} \cup \cdots \cup
E_{m-1})$ or simply $(V,\sigma,E)$.

We say that two signed, colored graphs are {\em isomorphic} if there
is a bijection between vertex sets that respects signatures and
color-adjacency. Definition~\ref{defn:deg} gives criteria for when a
signed, colored graph is isomorphic to $\G_{\lambda}$ by Theorem~\ref{thm:deg}.  

\begin{definition}
 A signed, colored graph $\G = (V,\sigma,E)$ of degree $m$ is a {\em
   dual equivalence graph} if the following hold:
 \begin{itemize}

 \item[{\em (ax$1$)}] For $w \in V$ and $1<i<m$, $\sigma(w)_{i-1} =
   -\sigma(w)_{i}$ if and only if there exists $x \in V$ such that
   $\{w,x\} \in E_{i}$.  Moreover, $x$ is unique when it exists.\\

 \item[{\em (ax$2$)}] \begin{tabbing} Whenever $\{w,x\} \in E_{i}$, \=
     $\sigma(w)_i = -\sigma(x)_i$  and \\
     \> $\sigma(w)_h = \hspace{1ex}\sigma(x)_h$ if $h <
     i-2$ or $h > i+1$.\\
   \end{tabbing}

 \item[{\em (ax$3$)}] \begin{tabbing} Whenever $\{w,x\} \in E_{i}$, \=
     if $\sigma(w)_{i-2} = -\sigma(x)_{i-2}$, then
     $\sigma(w)_{i-2} = -\sigma(w)_{i-1}$, and \\
     \> if $\sigma(w)_{i+1} = -\sigma(x)_{i+1}$, then $\sigma(w)_{i+1}
     = -\sigma(w)_{i}$.\\
   \end{tabbing}

 \item[{\em (ax$4$)}] For all $3<i<m$, every connected component of
   $(V,\sigma,E_{i-2} \cup E_{i-1} \cup E_{i})$ is either an isolated vertex or it is isomorphic to a graph in 
   Figure~\ref{fig:lambda5} after the signature function is restricted to positions $[i-2,i+1]$. 
	If $m=4$, every connected component of 
   $(V,\sigma, E_{2} \cup E_{3})$ is either an isolated vertex or it
   is isomorphic to a connected component in an induced subgraph of a graph
   in
   Figure~\ref{fig:lambda5} using only 2-edges and 3-edges and 
 restricting the  signature function to positions $[i-1,i+1]$. 
\\

 \item[{\em (ax$5$)}] Whenever $|i-j| \geq 3$, $\{w,x\} \in E_i$ and
   $\{x,y\} \in E_j$, there exists $v \in V$ such that $\{w,v\} \in
   E_j$ and $\{v,y\} \in E_i$.\\

 \item[{\em (ax$6$)}] Between any two vertices of a connected
component of $(V,\sigma,E_2 \cup \cdots \cup E_i)$, there exists a
path containing at most one edge in $E_i$.

 \end{itemize}
\label{defn:deg}
\end{definition}

\medskip


\begin{figure}[ht]
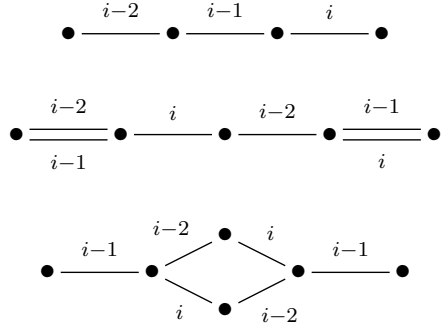

 \begin{displaymath}
   \begin{array}{c}
       \begin{array}{cccc}
         \rnode{h}{\bullet} \hspace{1em} & \hspace{1em}
         \rnode{i}{\bullet} \hspace{1em} & \hspace{1em}
         \rnode{j}{\bullet} \hspace{1em} & \hspace{1em}
         \rnode{k}{\bullet}
       \end{array} \\[2\cellsize]
       \begin{array}{ccccc}
         \rnode{a}{\bullet} \hspace{1em} & \hspace{1em}
         \rnode{b}{\bullet} \hspace{1em} & \hspace{1em}
         \rnode{c}{\bullet} \hspace{1em} & \hspace{1em}
         \rnode{d}{\bullet} \hspace{1em} & \hspace{1em}
         \rnode{e}{\bullet}
       \end{array} \\[2\cellsize]
       \begin{array}{ccccc}
         & & \rnode{w}{\bullet} & & \\
         \rnode{u}{\bullet} \hspace{1em} & \hspace{1em}
         \rnode{v}{\bullet} \hspace{1em} & & \hspace{1em}
         \rnode{y}{\bullet} \hspace{1em} & \hspace{1em}
         \rnode{z}{\bullet} \\
         & & \rnode{x}{\bullet} & &
       \end{array}
       \psset{nodesep=2pt,linewidth=.1ex}
       \everypsbox{\scriptstyle}
       \ncline  {h}{i} \naput{i-2}
       \ncline  {i}{j} \naput{i-1}
       \ncline  {j}{k} \naput{i}
       \ncline[offset=2pt] {a}{b} \naput{i-2}
       \ncline[offset=2pt] {b}{a} \naput{i-1}
       \ncline             {b}{c} \naput{i}
       \ncline             {c}{d} \naput{i-2}
       \ncline[offset=2pt] {d}{e} \naput{i-1}
       \ncline[offset=2pt] {e}{d} \naput{i}
       \ncline {u}{v}  \naput{i-1}
       \ncline {v}{w}  \naput{i-2}
       \ncline {v}{x}  \nbput{i}
       \ncline {w}{y}  \naput{i}
       \ncline {x}{y}  \nbput{i-2}
       \ncline {y}{z}  \naput{i-1}
     \end{array}
   \end{displaymath}
 \caption{\label{fig:lambda5} Possible $3$-color connected
   components of a dual equivalence graph with at least two vertices. Isolated vertices are also possible.}
\end{figure}

Comparing Figure~\ref{fig:G5} with Figure~\ref{fig:lambda5}, the largest possible connected
components of $(V,\sigma,E_{i-2}\cup E_{i-1} \cup E_i)$ are exactly
the graphs for $\G_{\lambda}$ when $\lambda$ is a partition of
$5$. Taking this comparison to its ultimate conclusion yields the
following result.

\begin{theorem}\cite{Assaf2007-2}
 For $\lambda$ a partition of $m$, $\G_{\lambda}$ is a dual
 equivalence graph of degree $m$. Moreover, every connected component
 of a dual equivalence graph of degree $m$ is isomorphic to
 $\G_{\lambda}$ for a unique partition $\lambda$ of $m$.
\label{thm:deg}
\end{theorem}

In practice, Axioms $1, 2$ and $5$ are trivially verified if $E_{i}$
is the set of pairs $\{(w, \phi_{i}(w)): w \neq \phi_{i}(w)\}$ determined by a family of
involutions $\phi_{i}: V \longrightarrow V$ such that for all $w \in
V$:
\begin{enumerate}
\item If $\sigma(w)_{i-1,i} = +-$, then $\sigma(\phi_{i}(w))_{i-1,i} =
  -+$, and vice versa.
\item Fixed points of $\phi_i$ are precisely those $w$ such that
  $\sigma(w)_{i-1,i} = ++$ or $--$.
\item The signatures $\sigma(w)$ and $\sigma(\phi_{i}(w))$ agree
outside the range of indices $i-2\leq j\leq i+2$.
\item The involutions $\phi_{i}$ and $\phi_{j}$ commute whenever
$|j-i|\geq 3$.
\end{enumerate}
Axiom 3 is typically verified by keeping track of a witness in each
case. The real difficulty lies in Axioms $4$ and $6$.

In \cite{Assaf2007-2}, the first author extended the notion of dual
equivalence in order to apply it to the
LLT and Macdonald polynomials.  For the extension, Axiom $6$ is no
longer required and Axiom $4$ is replaced by a weaker axiom.  Using
this generalized notion of dual equivalence, it was shown that
connected components have Schur positive generating functions but not
necessarily a single Schur function.  We use the same technique here
to prove that the terms in a $k$-Schur function can be partitioned in
connected components of a graph which are locally Schur positive.
Therefore, we review the necessary material from \cite{Assaf2007-2}
below.

\begin{definition}\label{def:graph.gf}
Let $\G = (V,\sigma,E)$ be a signed, colored graph.   Define the generating function associated to $\G$ to be 
\[
F_{\G}(X) = \sum_{v\in V} Q_{\sigma (v)}(X).
\]
\end{definition}

\begin{definition}\label{defn:lsp}
  A signed, colored graph $\G = (V,\sigma,E)$ of degree $m$ is a {\em
\D graph} if Axioms 1, 2, 3 and 5 (from Definition~\ref{defn:deg})
hold for $\G$.  A \D graph is said to be {\em locally Schur positive
on $h$-colored edges}, denoted $\lsp_{h}$, provided for all $2\leq
h<i<m$:

  \begin{itemize} 
\item[{ ($\lsp_{h}$)}] Every connected component of
$(V,\sigma,E_{i-h+1}\cup \ldots \cup E_{i})$ using  $h$ consecutive edge sets with signatures
restricted to positions $[i-h,i+1]$ has a symmetric and Schur positive
generating function.




\end{itemize} \label{defn:D-graph}
\end{definition}

For example, all of the graphs on Page~\pageref{fig:dgraph531} are
locally Schur positive on 2-colored edges and 3-colored edges.  Notice that any \D graph
satisfying Axioms $4$ and $6$ necessarily implies the graph is
$\lsp_h$ for all $h$ by Theorem~\ref{thm:deg}.

Observe that the signature function of a \D graph can be recovered
from the edges plus a single sign in any one signature on any one
vertex via the axioms.  Thus each graph in Figure~\ref{fig:lambda5}
can be assigned signature functions in exactly 2 ways which make them
into a \D graph.  The third graph can only be signed in one way up to
isomorphism.

\section{Poset on $n$-cores}
\label{sec:poset}

In order to define an analog of dual equivalence for starred strong
tableaux, we must first understand saturated chains in the $n$-core
poset. In this section, we do this by exploiting the connection
between $n$-cores and $\aSn$ using the abacus model for
partitions.

\subsection{Covering relations}
\label{sec:poset-cover}

We can describe the $n$-core poset more directly using the abacus
model for cores from \cite{JaKe1981}. Consider the diagram of a
partition $\lambda$, not necessarily an $n$-core, lying in the
$\mathbb{N} \times \mathbb{N}$ plane with infinite positive axes. Walk
in unit steps along the boundary of $\lambda$ placing a bead
($\bullet$) on each vertical step and a spacer ($\circ$) on each
horizontal step. Then straighten the boundary to get a doubly infinite
rod with the main diagonal marked by a vertical line. This gives the
binary string uniquely representing $\lambda$ when beads are replaced
by $1$'s and spacers by $0$'s. For example, we construct the string
for $(4,2)$ as follows.
\label{example.balanced}
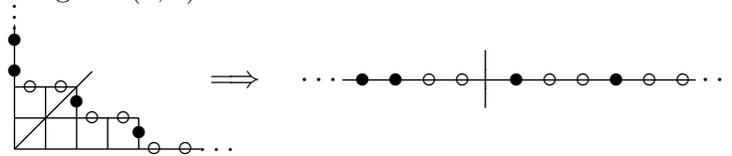
\begin{displaymath} 
  \raisebox{-3em}{%
  \psset{unit=1em,linewidth=.1ex}
  \pspicture(0,-1)(6,4)
  \psline(0,2)(2,2)
  \psline(0,1)(4,1)
  \psline(0,0)(6,0)
  \psline(0,0)(0,4)
  \psline(1,0)(1,2)
  \psline(2,0)(2,2)
  \psline(3,0)(3,1)
  \psline(4,0)(4,1)
  \psline(0,0)(2.5,2.5)
  \rput(0,4.5){$\vdots$}
  \rput(0,3.5){$\bullet$}
  \rput(0,2.5){$\bullet$}
  \rput(2,1.5){$\bullet$}
  \rput(4,0.5){$\bullet$}
  \rput(6.5,0){$\ldots$}
  \rput(5.5,0){$\circ$}
  \rput(4.5,0){$\circ$}
  \rput(3.5,1){$\circ$}
  \rput(2.5,1){$\circ$}
  \rput(1.5,2){$\circ$}
  \rput(0.5,2){$\circ$}
  \endpspicture}
  \Longrightarrow \hspace{1em}
  \rnode{l}{\cdots}  \ \bullet \ \bullet \ \circ   \ \circ \ \Big| \
  \bullet \ \circ   \ \circ  \ \bullet \ \circ   \ \circ   \
  \rnode{r}{\cdots}
  \psset{nodesep=2pt,linewidth=.1ex}
  \ncline {l}{r}
\end{displaymath}
Define the \textit{content} of a bead or spacer to be the content of the
diagonal immediately southeast. Indexing each bead or spacer by its
content gives an injective map from partitions to binary strings. The
\emph{abacus associated to $\lambda$} is the binary string of
$\lambda$ with beads and spacers indexed by their content. 

\begin{remark}
  Given any doubly infinite binary string $s$ such that $s_i$ is a
  bead for all $i<l$ and $s_i$ is a spacer for all $i>r$ for some
  $l,r$, there is a unique re-indexing of $s$ making it an abacus
  associated to a (unique) partition.
\label{rmk:bipart}
\end{remark}

Interchanging a bead on the abacus of $\mu$ with a spacer $m$ places
to its right corresponds to adding a ribbon of length $m$ to $\mu$,
and similarly interchanging a bead with a spacer $m$ positions to its
left removes an $m$-ribbon from $\mu$. In particular, if the moving
bead lands in position $s$, then the head of the
added ribbon will have content $s-1$.

Divide the abacus into $n$ rods, each containing all beads and spacers
of the same residue. Removing an $n$-ribbon from the boundary of
$\lambda$ precisely corresponds to moving a bead left along its
rod. Therefore $\lambda$ is an $n$-core precisely when each rod is an
infinite string of beads followed by an infinite string of spacers.
Define the \emph{content of a rod} to be the content of the bead or
spacer immediately to the right of the vertical line marking the main
diagonal.  We will identify a rod by its content throughout the paper.
Continuing with the previous example, taking $n=3$ gives the following
abacus decomposition of $(4,2)$, showing rods $1,2$ and $3$.
 \begin{displaymath}
  \raisebox{-3em}{%
  \psset{unit=1em}
  
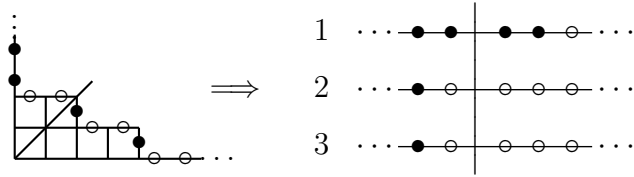
\pspicture(0,-1)(6,4)
  \psline(0,2)(2,2)
  \psline(0,1)(4,1)
  \psline(0,0)(6,0)
  \psline(0,0)(0,4)
  \psline(1,0)(1,2)
  \psline(2,0)(2,2)
  \psline(3,0)(3,1)
  \psline(4,0)(4,1)
  \psline(0,0)(2.5,2.5)
  \rput(0,4.5){$\vdots$}
  \rput(0,3.5){$\bullet$}
  \rput(0,2.5){$\bullet$}
  \rput(2,1.5){$\bullet$}
  \rput(4,0.5){$\bullet$}
  \rput(6.5,0){$\ldots$}
  \rput(5.5,0){$\circ$}
  \rput(4.5,0){$\circ$}
  \rput(3.5,1){$\circ$}
  \rput(2.5,1){$\circ$}
  \rput(1.5,2){$\circ$}
  \rput(0.5,2){$\circ$}
  \endpspicture}
  \Longrightarrow \hspace{1em}
  \begin{array}{lr}
    1 &  \rnode{l1}{\cdots}  \ \bullet \ \bullet \
    \Big| \ \bullet \ \bullet \ \circ \ \rnode{r1}{\cdots} \\
    2 &  \rnode{l2}{\cdots}  \ \bullet \ \circ \ \Big|
    \ \circ \ \circ \ \circ \ \rnode{r2}{\cdots} \\ 
    3 &  \rnode{l3}{\cdots}  \ \bullet \ \circ \ \Big|
    \ \circ \ \circ \ \circ \ \rnode{r3}{\cdots} \\ 
  \end{array}
  \psset{nodesep=2pt,linewidth=.1ex}
  \ncline {l1}{r1}    
  \ncline {l2}{r2}    
  \ncline {l3}{r3}    
\end{displaymath}

\begin{remark}
  Rotating the bottom row of the $n$-rod abacus for $\mu$ to the top
and shifting all beads in that row one column to the right will again
represent the abacus for $\mu$, but now shifted so that the rods have
contents $0, \ldots, n-1$ from top to bottom. Similarly, rotating the
top row down to the bottom and shifting all beads left on that row gives the
$n$-rod abacus for $\mu$ with contents $2, \ldots, n+1$.  Thus, the
abacus can be represented by $n$ rods of contents $k,k+1,k+2,\dots
,k+n-1$ for any integer $k$ by scrolling up or down.
\label{rmk:rotate} \end{remark}

Define the \emph{length} of each rod of the $n$-rod abacus as
follows. For $i=1,2,\ldots,n$, define the length of the rod with
content $i$ to be the number of beads on the rod with positive content
minus the number of spacers on the rod with nonpositive content (at
most one of these numbers is nonzero). For example, the lengths of
rods $1,2,3$ for the $3$-core $(4,2)$ are $2,-1,-1$.  In line with
Remark~\ref{rmk:rotate}, define the length of the remaining rods by
setting the length of rod $i-n$ equal to one plus the length of rod
$i$. It is sometimes convenient to rescale the lengths of the rods so
that the rods $1,2,\ldots,n$ have nonnegative length with at least one
having length $0$. For now, we are concerned only with the relative
lengths of the rods.

Affine permutations act on $n$-core partitions as discussed in Section
\ref{sec:defs-perms}. This action can be stated in terms of abaci as
well.  Recall we can represent a partition by an infinite binary
string.  Since affine permutations are bijections from $\mathbb{Z}$ to
$\mathbb{Z}$, we can apply such a bijection to any binary string.  If the
binary string represents an $n$-core then any affine transposition
applied to the binary string will also represent an $n$-core.  We
leave it to the reader to verify this action is consistent with the
action of simple affine transpositions acting on $n$-cores described
earlier.  In particular, the action of an affine transposition on an
$n$-core can be thought of pictorially as exchanging two rods of its
abacus and modifying all $n$-translates of these two rods
accordingly. The following observations, also noted in \cite{LLMS},
follow easily from the abacus model.

\begin{proposition}\label{p:rod.cover}  
The following statements hold for an $n$-core $\mu$ and $t_{r,s} \in
\aSn$ with $r<s$, $r \not \equiv s$:
\begin{enumerate}
\item The abacus for $t_{r,s} \mu$ is obtained from the abacus for
$\mu$ by swapping the lengths of the two rods with contents $r$ and
$s$.  All rods with content distinct from $r,s \mod n$ have the same
length in $\mu$ and $t_{r,s}\mu$.
\item In the $n$-core poset, $t_{r,s}\mu > \mu$ if and only if the rod
of content $r$ has larger length than the rod of content $s$ in $\mu$.
\item An $n$-core $\lambda$ covers $\mu$ if and only if $\lambda
= t_{p,q}\mu $ for some pair $p<q$, $p \not \equiv q$ such that in the
abacus for $\mu$ there is a bead at position $p$, a spacer at position
$q$, and no rod between $p$ and $q$ has length weakly between the
length of rod $p$ and the length of rod $q$.  Furthermore, the head
and tail of one ribbon in $\lambda /\mu $ have contents $q-1$ and $p$
respectively.
\end{enumerate}
\end{proposition}

\begin{proof} The first statement follows form the action of an affine
permutation on infinite binary strings.  The
  second statement is immediate since moving beads right adds ribbons
  and moving beads left removes ribbons. The third statement also
  follows from this interpretation.
\end{proof}

Proposition~\ref{p:rod.cover} is enough to describe precisely what
$\lambda/\mu$ may look like when $\lambda$ covers $\mu$ in the
$n$-core poset. The condition on the lengths of the rods that lie
between the interchanging rods of the abacus implies that the
connected components of $\lambda/\mu$ are identical ribbons. By
Remark~\ref{rmk:rotate}, the two rods being exchanged must have
distinct residues and no rod between them may have the same residue as
either of them. The contents across which the beads move determine the
contents contained in the ribbons, and the fact that both rods are
beads followed by spacers ensures that the ribbons lie on consecutive
residues. These observations reprove the following statement due to
Lam, Lapointe, Morse and Shimozono.

\begin{corollary}\cite[Prop. 9.5]{LLMS}
  Let $\mu$ be an $n$-core and $t_{r,s}$ an affine transposition such
that $t_{r,s} \mu$ covers $\mu$ in the $n$-core poset.  Then $0<s-r
<n$ and the connected components of $t_{r,s}\mu/\mu$ are identical
shape ribbons with cell residues from $r \mod n$ to $s-1 \mod
n$. Moreover, if rod $r$ has $k>0$ more beads than rod $s$, then
$t_{r,s}\mu/\mu$ has exactly $k$ identical ribbons.  If the head of
the first ribbon lies in a cell with content $c$, then the head of the
other ribbons have content $c+n, c+2n,\dots , c+(k-1)n$.  
\label{cor:iribbons}
\end{corollary}

By Corollary~\ref{cor:iribbons}, for a strong tableau $S$ of shape
$\lambda$, call the connected components of $\lambda_i /
\lambda_{i-1}$ the \emph{$i$-ribbons of $S$}.  Recall from
Section~\ref{sub:starred.strong.tableaux} that a starred strong
tableau consists of a strong tableau plus a choice of $i$-ribbon for
each $i$ present in $S$.  We use the next definition and corollary to
relate the starred strong tableaux to saturated chains labeled by
certain sequences of transpositions.   

\begin{definition}\label{defn:transposition.sequences}
  Let $\mu \subset \lambda$ be $n$-cores, and let
  $\Trans{\lambda/\mu,n}$ be the set of all \emph{transposition
    sequences}
  \[
  \left(t_{r_{1}s_{1}}\rightarrow t_{r_{2}s_{2}}\rightarrow \dotsb
    \rightarrow t_{r_{m}s_{m}} \right)
  \]
  such that
  \begin{enumerate}
  \item the product $t_{r_{m}s_{m}} \dotsb t_{r_{2}s_{2}} t_{r_{1}s_{1}}
    \mu= \lambda$ as elements of $\aSn/\Sn$;
  \item for each $1\leq i\leq m$, we have $0<s_{i}-r_{i}<n$;
  \item for each $0\leq i< m$, the abacus for $\mu^{(i)}=
    t_{r_{i}s_{i}} \dotsb t_{r_{2}s_{2}} t_{r_{1}s_{1}} \mu $ contains a
    bead at position $r_{i+1}$, a spacer at position $s_{i+1}$, and every rod	
    with content    between  $r_{i+1}$ and $s_{i+1}$ has length strictly 
     smaller than both the length of rod $r_{i+1}$ and the length of rod $s_{i+1}$ or strictly larger than both.
\end{enumerate}
\end{definition}

By Proposition~\ref{p:rod.cover}, condition (3) above implies $\mu
=\mu^{(0)}<\mu^{(1)}<\dotsb < \mu^{(m)}=\lambda$ forms a saturated
chain in the $n$-core poset.  The following is a consequence of
Proposition~\ref{p:rod.cover} and Corollary~\ref{cor:iribbons}.

\begin{corollary}\label{cor:transposition.seq.bijection}
  Let $\mu \subset \lambda$ be $n$-cores.  There exists a bijection
  from skew starred strong tableaux $S^{*} \in \SST(\lambda/\mu, n)$
  to $\Trans{\lambda/\mu, n}$ given by mapping
  \[
  S^{*} \mapsto
  \left(t_{r_{1}s_{1}}\rightarrow t_{r_{2}s_{2}}\rightarrow \dotsb
    \rightarrow t_{r_{m}s_{m}} \right) 
  \]
  where $s_{i}-1$ and $r_{i}$ are the contents of the head and tail of
  the $i$-ribbon containing $i^{*}$ in $S^{*}$.
\end{corollary}

For example, this bijection maps   
\begin{displaymath}
  \raisebox{\cellsize}{\tableau{4^{*} \\ 3 \\ 1^{*} & 2^{*} & 3^{*}}} \mapsto
  \left(t_{0,1}\rightarrow t_{1,2}\rightarrow t_{2,3}\rightarrow
  t_{-2,-1} \right).
\end{displaymath}

\subsection{Intervals of length two}
\label{sec:poset-intervals}

As motivation, recall that an elementary dual equivalence on standard
tableaux may be defined in terms of interval exchanges in Young's
lattice. Though the induced poset on $n$-cores in not as nice as
Young's lattice, Bj\"{o}rner and Brenti \cite{BjBr1996} showed that
any interval of length two is either a chain or isomorphic to $B_{2}$.

\begin{definition}
  Let $S = (\emptyset = \mu^0 \subset \mu^1 \subset \cdots \subset \mu^{m})$ be a
  saturated chain in the $n$-core poset such that the interval
  $[\mu^{i-1},\mu^{i+1}]$ is not a chain for $0 < i < m$. The
  \emph{$i$-interval swap} on $S$, denoted
  $\swap_{i,i+1}(S)=\swap_{i+1,i}(S)$, replaces $\mu^i$ with the
  unique other $n$-core at rank $i$ in $[\mu^{i-1},\mu^{i+1}]$.
  \label{defn:swap}
\end{definition}

For example, from Figure~\ref{fig:poset} we see that a $2$-interval
swap on the chain
\begin{displaymath}
  \emptyset \hspace{.5\cellsize} \subset \hspace{.5\cellsize} 
  \tableau{\e} \hspace{.5\cellsize} \subset \hspace{.5\cellsize} 
  \tableau{\e & \e} \hspace{.5\cellsize} \subset \hspace{.5\cellsize}
  \raisebox{\cellsize}{\tableau{\e \\ \e & \e & \e}} 
  \hspace{.5\cellsize} \subset \hspace{.5\cellsize} 
  \raisebox{2\cellsize}{\tableau{\e \\ \e \\ \e & \e & \e}}
\end{displaymath}
results in the chain
\begin{displaymath}
  \emptyset \hspace{.5\cellsize} \subset \hspace{.5\cellsize} 
  \tableau{\e} \hspace{.5\cellsize} \subset \hspace{.5\cellsize} 
  \raisebox{\cellsize}{\tableau{\e \\ \e}} 
  \hspace{.5\cellsize} \subset \hspace{.5\cellsize}
  \raisebox{\cellsize}{\tableau{\e \\ \e & \e & \e}} 
  \hspace{.5\cellsize} \subset \hspace{.5\cellsize} 
  \raisebox{2\cellsize}{\tableau{\e \\ \e \\ \e & \e & \e}}.
\end{displaymath}
In terms of the strong tableaux, the same 2-interval swap gives
\begin{equation}\label{e:swap.2.3}
  \raisebox{\cellsize}{\tableau{4 \\ 3 \\ 1 & 2 & 3}}
  \hspace{\cellsize} 
  \stackrel{\swap_{2,3}}{\longleftrightarrow}
  \hspace{\cellsize} 
  \raisebox{\cellsize}{\tableau{4 \\ 2 \\ 1 & 3 & 3}}. 
\end{equation}

By Definition~\ref{defn:transposition.sequences}, the same two
saturated chains can be represented by the following transposition
sequence
\[
\swap_{2,3}\left(t_{0,1}\rightarrow t_{1,2}\rightarrow
  t_{2,3}\rightarrow t_{-2,-1} \right)
=t_{0,1}\rightarrow t_{-1,0}\rightarrow t_{1,3} \rightarrow t_{-2,-1}, 
\]
where only the two transpositions in the middle are modified.  In
general, the map $\swap_{i,i+1}(S)$ always modifies the $i$th and the
$i+1$st transpositions in any transposition sequence representing the
saturated chain $S$ and leaves all other transpositions in the
sequence fixed.  The two new transpositions are not unique however
since we have not yet described how the stars will move.  This
extension will be called a $\basic$ and introduced in
Section~\ref{sec:equivalence}.  First, we need a complete
understanding of how the $i$-ribbons and the $i+1$-ribbons can appear
in a strong tableau and how they change under an $i$-interval swap.

Using the abacus model for cores and Proposition~\ref{p:rod.cover}, we
can explicitly describe the result of an $i$-interval swap on a strong
tableau $S=(\emptyset \subset \mu^1 \subset \cdots \subset \mu^{m})$
in terms of the two rod exchanges corresponding to the covering
relations in the interval $[\mu^{i-1},\mu^{i+1}]$. In order to analyze
two consecutive rod exchanges, we extend the $n$-rod abacus picture to
include extra rods above as necessary so that the four rods to be
exchanged all appear as rows of the picture with the longer row above
the shorter row.  Ignoring the rods which are untouched by the
exchange and choosing representatives of the exchanging rods as close
together as possible, there are four natural cases to consider, each
depicted in Figure~\ref{fig:abaci}: disjoint, interleaving, nested and
abutting. There are three possible ways for the rod exchanges to be
abutting; the two depicted and also the reverse of the right hand
side. For the first three cases in Figure~\ref{fig:abaci}, the
corresponding transpositions will have four distinct residues whereas
for the abutting case, they will have only two or three distinct
residues.

\begin{figure}[ht]
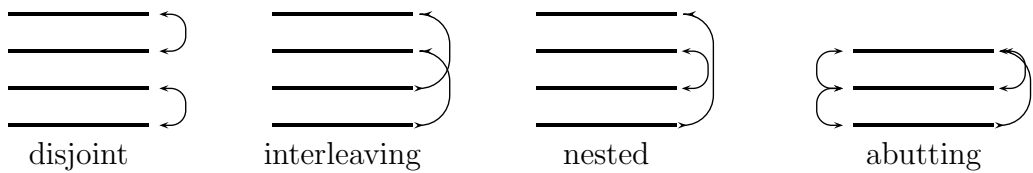

  \begin{center}
    \begin{displaymath}
      \begin{array}{cccc}
        \begin{array}{c}
          \rnode{da}{\rule{4.5em}{1pt}} \\
          \rnode{db}{\rule{4.5em}{1pt}} \\
          \rnode{dc}{\rule{4.5em}{1pt}} \\
          \rnode{dd}{\rule{4.5em}{1pt}} \\
          \mbox{disjoint}
        \end{array} \hspace{2em} &
        \begin{array}{c}
          \rnode{oa}{\rule{4.5em}{1pt}} \\
          \rnode{ob}{\rule{4.5em}{1pt}} \\
          \rnode{oc}{\rule{4.5em}{1pt}} \\
          \rnode{od}{\rule{4.5em}{1pt}} \\
          \mbox{interleaving}
        \end{array} \hspace{2em} &
        \begin{array}{c}
          \rnode{na}{\rule{4.5em}{1pt}} \\
          \rnode{nb}{\rule{4.5em}{1pt}} \\
          \rnode{nc}{\rule{4.5em}{1pt}} \\
          \rnode{nd}{\rule{4.5em}{1pt}} \\
          \mbox{nested}
        \end{array} \hspace{2em} & \hspace{2em} 
        \begin{array}{c}
          \\
          \rnode{ba}{\rule{4.5em}{1pt}} \\
          \rnode{bb}{\rule{4.5em}{1pt}} \\
          \rnode{bc}{\rule{4.5em}{1pt}} \\
          \mbox{abutting}
        \end{array}
      \end{array}
      \psset{nodesep=4pt,linewidth=.1ex}
      \ncdiag[angle=0,linearc=.2]{<->}{da}{db}
      \ncdiag[angle=0,linearc=.2]{<->}{dc}{dd}
      \ncdiag[angle=0,linearc=.4]{<->}{oa}{oc}
      \ncdiag[angle=0,linearc=.4]{<->}{ob}{od}
      \ncdiag[angle=0,linearc=.4]{<->}{na}{nd}
      \ncdiag[angle=0,linearc=.2,nodesep=2pt]{<->}{nb}{nc}
      \ncdiag[angle=0,linearc=.4]{<->}{ba}{bc}
      \ncdiag[angle=0,linearc=.2,nodesep=2pt]{<->}{ba}{bb}
      \ncdiag[angle=180,linearc=.2]{<->}{ba}{bb}
      \ncdiag[angle=180,linearc=.2]{<->}{bb}{bc}
    \end{displaymath}
    \caption{\label{fig:abaci}Possible rod exchanges for a length two interval.}
  \end{center}
\end{figure}

The easiest case to consider is a \emph{disjoint exchange}.  Here we
assume all of the residues of the rods to be exchanged are distinct,
lest we actually have an abutting exchange. Further, we can assume the
exchanging rods have contents $a < b < c < d$ with $n > d - a > 0$
since the rods are as close together as possible. The two exchanges in
this case clearly commute, and taking either first raises the rank by
exactly one by Corollary~\ref{cor:iribbons}. In the strong tableau,
such an $i$-interval swap will happen precisely when the cells of the
$i$-ribbons and $(i+1)$-ribbons have no residues in common, and the
effect of the swap will be to exchange all $i$'s for $i+1$'s and
conversely.

The case of an \emph{interleaving exchange} is only slightly more
interesting, though the conclusion of this case is
noteworthy. Labeling the residues of the exchanging rods $a < b < c <
d$ from top to bottom, again we assume all four residues to be
distinct lest we be pulled into the abutting case.  The assumption
that these two exchanges each increase the rank in the poset forces
rod $a$ longer than rod $c$ and similarly rod $b$ longer than rod $d$
by Proposition~\ref{p:rod.cover}.  Suppose $\mu^{i-1} \subset
t_{a,c}\mu^{i-1}=\mu^{i} \subset t_{b,d}t_{a,c} \mu^{i-1}=\mu^{i+1}$;
the other case is similarly resolved. By
Proposition~\ref{p:rod.cover}, this means the length of rod $b$ does
not lie between the lengths of rods $a$ and $c$ and that the length of
rod $a$ does not lie between the lengths of rods $b$ and $d$. Recall,
we chose a picture for the abacus so that the length of rod $b$ is
larger than the length of rod $d$ and the length of rod $a$ is longer
than the length of rod $c$.  These statements together imply that the
lengths of rods $a$ and $c$ do not interleave the lengths of rods $b$
and $d$, and so the transpositions taken in the other order each raise
the rank by exactly one, thus $\mu^{i-1} \subset t_{b,d}\mu^{i-1} \subset
t_{a,c}t_{b,d} \mu^{i-1}=\mu^{i+1}$ is a valid strong tableau.  In this
new strong tableau, the contents of the $i$-ribbons and $i+1$-ribbons
will not overlap, though the residues will. It is also important to
note that the $i$-ribbons and $i+1$-ribbons will not have the same
residues for their heads or tails. In this case, the $i$-interval swap
again simply exchanges all $i$'s for $i+1$'s and conversely. We
summarize the key observation in this case as follows.

\begin{proposition}\label{prop:nested}
  An $i$-ribbon and an $i+1$-ribbon in a strong tableau have
overlapping contents if and only if the contents of one ribbon are
strictly contained in the contents of the other.  Furthermore, the
contents of the head and tail of the longer ribbon do not occur among
the contents of the shorter ribbon.  \label{prop:overlap}
\end{proposition}

More generally, we say two ribbons are \emph{nested} if the second
condition of Proposition~\ref{prop:nested} holds. We also say two
ribbons $R_{1}$ and $R_{2}$ are \emph{independent} if $R_{1} \cup
R_{2}$ has two connected components as a skew shape.

In the case of a \emph{nested exchange}, again label the rod
contents $a < b < c < d$ from top to bottom.  We can assume $d-a<n$ by
Corollary~\ref{cor:iribbons}.  Here the two corresponding
transpositions commute, and each will raise the rank by exactly
one. The interesting feature of this case lies in the strong tableaux.

By Proposition~\ref{p:rod.cover}, neither the length of rod $b$ nor
the length of rod $c$ may lie between the lengths of rods $a$ and $d$.
If both rod $b$ and rod $c$ are longer than rod $a$ or both shorter
than rod $d$ (necessarily rod $a$ is longer than rod $d$), then the
$i$-ribbons and $i+1$-ribbons will have no contents in common, though
the residues of one ribbon will be strictly contained within the
residues of the other.  Furthermore, both the heads and tails of the
$i$-ribbons and $i+1$-ribbons have distinct residues.

On the other hand, if rod $b$ is longer than rod $a$ and rod $c$ is
shorter than rod $d$ (necessarily rod $b$ is longer than rod $c$),
then the content of every instance of the longer ribbon (corresponding
to $t_{a,d}$) overlaps the content of a  shorter ribbon
(corresponding to $t_{b,c}$) and there must be an instance of the
shorter ribbon containing a cell of content $b-1$ which occurs
independently from all of the longer ribbons.  An $i$-interval swap
changes all entries in all of the shorter ribbons that appear
independently of the longer ribbons and all entries of the longer
ribbons that are not on the same content as a shorter ribbon.  For
example, below is the skew strong tableau corresponding to a nested
exchange and the result of the interval swap

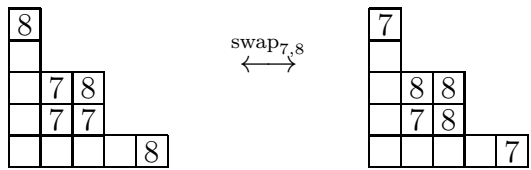
\[
 \tableau{  
8 & \\
{}\\
{} & 7 & 8 & \\
{} & 7 & 7 & \\
{} & {} & {} & {} & 8 & \\
}
        \hspace{\cellsize}
   \raisebox{-\cellsize}{$\stackrel{\swap_{7,8}}{\longleftrightarrow}$}
        \hspace{2\cellsize}
 \tableau{
7 & \\
{}\\
{} & 8 & 8 & \\
{} & 7 & 8 & \\
{} & {} & {} & {} & 7 & \\
}
\]
\bigskip

\noindent This discussion proves the following lemma.

\begin{lemma}\label{lem:nested.ribbons}
If an $i$-ribbon and an $i+1$-ribbon are nested, then 
\begin{enumerate}
\item At least two copies of the shorter ribbon occur independently
from the longer ribbon, with at least one on either side of the
consecutive sequence of copies of the longer ribbon.
\item Every copy of the longer ribbon nests a copy of the shorter ribbon.
\item Both the heads and tails of the $i$-ribbons and $i+1$-ribbons have
distinct residues.
\item An $i$-interval swap is possible.
\end{enumerate}
\end{lemma}

The final case of an \emph{abutting exchange} will involve exactly
three distinct indices on the transpositions, though possibly only two
distinct residues. Label the contents of the rods $a < b < c$ from top
to bottom.  Suppose that the three residues are all distinct.  This is
necessarily the case for the right hand side of
Figure~\ref{fig:abaci}.  Say the two exchanges correspond with the
transposition sequence $(t_{a,c} \rightarrow t_{a,b})$, then by
Proposition~\ref{p:rod.cover}, we know rod $a$ is strictly longer than
rod $c$ which is strictly longer than rod $b$, and taking the inner
exchange first forces rod $a$ longer than rod $b$ longer than rod
$c$. Therefore we note that $t_{b,c} t_{a,b} = t_{a,b} t_{a,c}$ so
this equation along with the total order on the lengths of the rods
ensures that an interval swap is possible.  The new transposition
sequence after applying this interval swap would be $(t_{a,b}
\rightarrow t_{b,c})$ which corresponds with the left hand side of the
abutting exchange pictured in Figure~\ref{fig:abaci}.  If the two
exchanges correspond with the transposition sequence $(t_{a,b}
\rightarrow t_{a,c})$, examining the required rod length inequalities
again we see that $(t_{b,c} \rightarrow t_{a,b})$ is a valid
transposition sequence on the same rank 2 interval.  This again
corresponds with the left hand side of the abutting exchange pictured
in Figure~\ref{fig:abaci}.  If the right hand side of
Figure~\ref{fig:abaci} is turned upside down, a similar analysis
holds.  Furthermore, the interval swaps form an involution on the two
chains in any interval isomorphic to $B_{2}$ so we have covered all
possible cases of an abutting exchange in the form of the left hand
side of the abutting picture of Figure~\ref{fig:abaci} as well.  Hence
in all cases of an abutting exchange with three distinct residues,
there exists an interval swap determined above.

Assuming that $[\mu^{i-1},\mu^{i+1}]$ is isomorphic to $B_{2}$, one
way to recognize if an abutting exchange is required for
$\swap_{i,i+1}(S)$ is that an $i$-ribbon and an $i+1$-ribbon together
form a ribbon shape.  In this case, we will say these two ribbons
\emph{abut} each other.  From the transpositions pictured in the
abutting case of Figure~\ref{fig:abaci} and
Corollary~\ref{cor:iribbons}, we observe that the sum of the lengths
of an $i$-ribbon and an $i+1$-ribbon is necessarily less than $n$ and
exactly one of the two ribbon types occurs without abutting a copy of
the other.  In this instance, the $i$-interval swap will change all
entries of the non-abutting ribbons and all entries in their
$n$-translates.  For example, the 2-ribbon abuts a 3-ribbon in the
strong tableau on the left in \eqref{e:swap.2.3}.

The other way to recognize if an abutting exchange is required for
$\swap_{i,i+1}(S)$ is that, among the $i,i+1$-ribbons, one ribbon is
strictly longer than the other and the longer ribbon contains an
$n$-translate of the shorter and the heads or tails of the two ribbons
have the same residue depending on if the shared rod is $a$ or
$c$. See for example, the 2-ribbon and 3-ribbon in the strong tableau
on the right in \eqref{e:swap.2.3}.  Here an interval swap will change
all entries of the shorter ribbons and all entries of the longer
ribbons that are not part of an $n$-translate of the shorter.  This
case is also recovered from the left hand side of
Figure~\ref{fig:abaci} when the lengths of the three rods are all
distinct; we omit details as the case is completely parallel.

Following the details of the abutting exchange case carefully, we have
the following.

\begin{proposition}
  Suppose $\swap_{i,i+1}(S)$ is obtained from $S$ by an abutting
exchange.  Assume the corresponding transpositions are indexed by 3
distinct residues mod $n$.  Then either

\begin{itemize}
\item No $i$-ribbon abuts any $i+1$-ribbon, but one of these ribbons
  strictly contains an $n$-translate of the other with a shared head
  or tail occurring on a consecutive residue.
  \\
\item OR, all instances of one ribbon type abut the other while the
  other will also have at least one components which is non-abutting
  and the sum of the length of an $i$-ribbon and an $i+1$-ribbon is at
  most $n-1$. In this case, if an $i+1$-ribbon abuts an $i$-ribbon
  from the north, then the non-abutting ribbons lie always southeast of
  the abutting ribbons, and if an $i$-ribbon abuts an $i+1$-ribbon
  from the west, then the non-abutting ribbons lie always northwest of
  the abutting ribbons.
\end{itemize}
  \label{prop:stack}
\end{proposition}

Finally, consider an abutting exchange as in the left hand side of the
abutting case in Figure~\ref{fig:abaci}.  If the three rod lengths are
distinct and the three residues are distinct, then the exchange is
covered by Prop~\ref{prop:stack}.  In each of the remaining cases, we
claim the interval $[\mu^{i-1},\mu^{i+1}]$ is a chain so an
$i$-interval swap is not possible.

\begin{proposition}
  Let $S = (\emptyset = \mu^0 \subset \mu^1 \subset \cdots \subset
\mu^{m})$ be a saturated chain in the $n$-core poset.  Then the
interval between $\mu^{i-1}$ and $\mu^{i+1}$ in the $n$-core poset is
a chain if and only if each $i$-ribbon abuts an $i+1$-ribbon and each
$i+1$-ribbon abuts an $i$-ribbon.  Moreover, the length of an
$i$-ribbon plus the length of a $i+1$-ribbon is less than or equal to
$n$, with equality if and only if $\mu^{i+1}/\mu^{i-1}$ is a single connected
ribbon shaped component starting and ending with $i+1$-ribbons.
\label{prop:chains}
\end{proposition}

\begin{proof}
Assume $[\mu^{i-1}, \mu^{i+1}]$ is a chain.  Then by the previous
analysis of rod exchange cases the chain corresponds with a
transposition sequence of the form $t_{b,c} \rightarrow t_{a,b} $ with
$a<b<c$ or $a>b>c$.  Assume $a$ and $c$ have different residues (both
necessarily have distinct residues from $b$).  In this case, $t_{a,b}
t_{b,c} = t_{a,c}$ and by Proposition~\ref{p:rod.cover} we can assume
$0<c-a<n$.  Thus the skew shape $\mu^{i+1}/\mu^{i-1}$ is the union of
a positive number of $n$-translates of a single ribbon of length less
than $n$ and none of these ribbons overlap in content.  More
precisely, $i$-ribbons and $i+1$-ribbons always occur in pairs and the
sum of their lengths is strictly less than $n$.

If, on the other hand, $a$ and $c$ have the same residue, then we can
assume $c = a + n$ by choosing to label the exchanging rods as close
together as possible. Hence, the length of the ribbons corresponding
to $t_{a,b}$ and those corresponding to $t_{b,a+n}$ necessarily add to
$n$ so $\mu^{i+1}/\mu^{i-1}$ is a single connected ribbon shaped
component.  Furthermore, recall that rod $c=a+n$ is one shorter than
the length of rod $a$ by Remark~\ref{rmk:rotate}.  If rod $b$ is
shorter than rod $a$, then the chain corresponds with the
transposition sequence $t_{a,b} \rightarrow t_{b,c}$, otherwise the
transpositions happen in the reverse order.  In either case, by
considering how ribbons are created using the abacus model and
Proposition~\ref{p:rod.cover} we observe that the ribbon
$\mu^{i+1}/\mu^{i-1}$ is tiled by an alternating sequence of
$i$-ribbons and $i+1$-ribbons and it begins and ends with an
$i+1$-ribbon.  

To prove the reverse direction, assume each $i$-ribbon abuts an
$i+1$-ribbon and conversely.  Then by Corollary~\ref{cor:iribbons} we
can infer that the chain $\mu^{i-1} \subset \mu^{i} \subset \mu^{i+1}$
corresponds to an abutting exchange.  If all three contents of the
exchanging rods have distinct residues, then either $[\mu^{i-1},
\mu^{i+1}]$ is a chain or we would find a contradiction to the second
case of Proposition~\ref{prop:stack}.

If there are only two distinct indices among the exchanging rods then
the relative lengths of these rods determine the only possible
exchange sequence taking $\mu^{i-1}$ to $\mu^{i+1}$ by
Proposition~\ref{p:rod.cover}.  Thus, $[\mu^{i-1}, \mu^{i+1}]$ is
again a chain.   
\end{proof}

\begin{corollary}\label{cor:ribbons}
If a strong tableau $S = (\mu^0 \subset \mu^1 \subset \mu^{2})$ is the
result of an abutting exchange, then $\mu^{2}/\mu^{0}$ is the union of
ribbons with nonoverlapping content.  If every ribbon in 
$\mu^{2}/\mu^{0}$ is an identical $n$-translate of the first, then the interval $[\mu^{0},\mu^{2}]$ is a chain.
\end{corollary}

\begin{proof}
This follows from the characterization of all abutting exchanges in
this subsection, Proposition~\ref{prop:stack} and
Proposition~\ref{prop:chains}.
\end{proof}

Table~\ref{table:length.two}  summarizes the discussion above characterizing all
possible length two intervals determined by two consecutive rod
exchanges.  Assume the initial $n$-core is $\mu$.  First apply
$t_{a,b}$ then $t_{c,d}$ assuming $0<b-a<n$, $0<d-c<n$, rod $a$ longer
than rod $b$, rod $c$ longer than rod $d$, and all 4 indices appear in
the smallest possible interval of $\mathbb{Z}$ which satisfies these
conditions.  Let $\#res$ be the number of distinct residues
among $a,b,c,d$ mod $n$.  Let $\#dis$ be the number of
distinct rod lengths among rods $a$, $b$, $c$, $d$ in $\mu$.  The
interval $[\mu , t_{cd} t_{ab} \mu ]$ is either isomorphic to $B_{2}$
or the chain $C_{3}$ with 3 elements.  The two interval types are distinguished
by considering  $\#res$ and $\#dis$ or equivalently by considering the skew shape as partitions of  $t_{cd} t_{ab} \mu/ \mu $.  

\begin{table}
\begin{small}
\begin{tabular}{|c|c|c|c|l|}
\hline
$\#res$ & $\#dis$ &   Type & Exchange  & Skew Shape\\
\hline
2 &	2  &	$C_{3}$  & abutting &	One long ribbon alternating $i$'s and $i+1$'s,\\
&&&&                                    starting and ending with $i+1$.\\ [2ex]
\hline
3 &	2     & $C_{3}$  & abutting &   Every component is an identical ribbon \\
&	&	&	&	        composed of one $i$-ribbon abutting one $i+1$ ribbon. \\ [2ex]
\hline
3 &	3     & $B_{2}$  & abutting &	Long ribbons contain $n$-translate of shorter ribbons. \\
&	&	&	&	        One short ribbon occurs independently.\\ [2ex]
\hline
4 &	2,3,4 &  $B_{2}$ & disjoint  &  All ribbons are non-overlapping, \\
&	&	&	&	        $i$-ribbons don't abut $i+1$-ribbons or vice versa. \\
&	&	&	&	        No $n$-translate of one ribbon type overlaps the other.\\ [2ex]
\hline
4 &	4    &$B_{2}$ & interleaving &  All ribbons are non-overlapping,  \\
&	&	&	&	        $i$-ribbons don't abut $i+1$-ribbons or vice versa.  \\
&	&	&	&	        Some $n$-translate of an $i$-ribbon overlaps an  \\
&	&	&	&	        $i+1$-ribbon with distinct heads and tails.\\ [2ex]
\hline
4 & 4         & $B_{2}$  & nested    & Either all ribbons are non-overlapping, or \\
&	&	&	&	       all longer ribbons overlap shorter ribbons and\\
&	&	&	&	       at least one short ribbon occurs independently  \\
&	&	&	&	      NW (SE) of each long ribbon ribbons.  These two   \\
&	&	&	&	      cases distinguished by comparing rod lengths.\\
\hline
\end{tabular}
\end{small}
\caption{Summary of length two interval types.}
\label{table:length.two}
\end{table}

\section{Affine dual equivalence}
\label{sec:equivalence}

We now have all the ingredients to construct an analog of dual
equivalence for starred strong tableaux, which we call \emph{affine
dual equivalence}. Though our equivalence relation will not share all
of the properties of dual equivalence on tableaux, we will go on in
Section~\ref{sec:graph} to construct a signed colored graph from
our elementary equivalence relations that we show to be a \D graph.   

While the elementary equivalence relations will have a somewhat
complicated description, there are essentially only two cases: one
that precisely mirrors dual equivalence, and another that is a close
approximation when the former is not applicable. Remarkably, the
relations also preserve the spin statistic on starred strong tableaux.

\subsection{Elementary equivalences}
\label{sec:graph-sst}

In this subsection, we describe a family of involutions $\vp_{i}$ on
all starred strong tableaux of a given shape that will define the
\emph{elementary affine dual equivalence} on $i-1,i,i+1$. Recall
that a starred strong tableau $S^{*}$ of shape $\lambda$ can be
represented by a strong tableau $S = ( \emptyset \subset \lambda^{(1)}
\subset \lambda^{(2)} \subset \cdots \subset \lambda^{(m)})$ with
$\lambda^{(m)}=\lambda$ and a vector $c^{*}=(c_{1},c_{2},\dotsc,
c_{m})$ where $c_{i}$ is the content of the cell of $S^{*}$ containing
$i^{*}$.  In this case, we will say the \emph{rank} of $S^{*}$ is
$m$.

\begin{definition}\label{def:witness}
  Let $S^*= (S,c^{*})$ be a starred strong tableau of rank $m$.  Fix
  $1<i<m$.  Consider the locations of $(i-1)^{*}, i^{*},(i+1)^{*}$ in
  $S^{*}$.  The $i$-\emph{witness}, or simply the \emph{witness}
  when $i$ is fixed, is chosen among $\{i-1,i,i+1 \}$ as follows.

  \begin{enumerate}
  \item If $c_{i-1} \neq c_{i+1}$, then $c_{i-1},c_{i},c_{i+1}$ are
    all distinct since $p$-ribbons and $p+1$-ribbons cannot have 
    head or tails of the same content by the analysis in Section~\ref{sec:poset-intervals}. 
    In this case, the \emph{witness} is the index of
    the median of the set $\{c_{i-1},c_{i},c_{i+1} \}$.
  \item If $c_{i-1}=c_{i+1}$, then we have three cases to consider.
    \begin{enumerate}
    \item If the $(i-1)$-ribbons and $(i+1)$-ribbons have the same
      length, then $i+1$ is the \emph{witness}.
    \item If the $(i-1)$-ribbons and $(i+1)$-ribbons have different lengths
      and $c_{i-1}>c_{i}$, then the \emph{witness} is the letter
      indexing the longer ribbons among the $(i-1)$-ribbons and the
      $(i+1)$-ribbons.
    \item If the $(i-1)$-ribbons and $(i+1)$-ribbons have different
      lengths and $c_{i-1}<c_{i}$, then the \emph{witness} is the
      letter indexing the shorter ribbons among the $(i-1)$-ribbons
      and the $(i+1)$-ribbons.
    \end{enumerate}
  \end{enumerate}
  \label{defn:witness}
\end{definition}


Note that when $S^{*}$ is a Young tableau, the contents of the unique
cells containing $i-1, i$ and $i+1$ must all be distinct, ensuring
that the witness is always the index of the median of the set
$\{c_{i-1},c_{i},c_{i+1}\}$.

Next we define the involution $\vp_{i}$ on starred strong tableaux
that will serve as a model for dual equivalence. Intuitively, if $i$
and $j$ are witnessed by $h$ in $S^{*}$, then an elementary dual
equivalence move should be based on the map $\swap_{i,j}$ where
$\{h,j\}=\{i-1,i+1\}$. This would be straightforward but for the
difficulty of defining how the stars should behave under such an
action.  We obtained these rules experimentally guided by the
principle that the stars should move as little as possible while
preserving the spin statistic, always remaining in the same connected
component of the union of cells in $S^{*}$ containing $i-1,i,i+1$ but
necessarily switching which letter they adorn.  This will be the
action of $\vp_i$ whenever such a move is possible without changing
the witness. However, if the interval is a chain and the starred
letters both lie in the same connected component, then neither an
interval swap nor a star swap is possible. We overcome this challenge
by exchanging saturated chains of length three.

\begin{definition}  \label{defn:phi}
  Fix a starred strong tableau $S^{*}=(S,c^{*})$ of rank $m$ with
$1<i<m$.  Let $h$ be the $i$-witness for $S^{*}$.  If $h \neq i$, then
let $j$ be defined by $\{i-1,i+1\} = \{j,h\}$.  Let $S_{q}$ be the
union of all $q$-ribbons and let $S_{q^*}$ be the connected component
of $S_q$ containing $q^*$ for $1\leq q \leq m$. We will say $S_{q}$
\emph{nests} $S_{p^{*}}$ if the content of every cell of $S_{p^{*}}$
is also the content of a cell in $S_{q}$ but no head or tail of a
ribbon in $S_{q}$ has the same content as the head or tail of
$S_{p^{*}}$.  Similarly, a connected skew shape $A$ \textit{nests}
another connected skew shape $B$ provided the content of every cell of
$B$ is the content of some cell of $A$, but the largest and smallest
contents of cells in $A$ are not the contents of any cells in $B$.
Let $b_{q}$ be the content of the ribbon tail for $S_{q^{*}}$.  We
will say $S_{i}$ and $S_{j}$ are \emph{not abutting} if
$b_{i},b_{j},(c_{i}+1),(c_{j}+1)$ have distinct residues, otherwise
$S_{i}$ and $S_{j}$ are \textit{abutting}.  Let $B_{i}$ and $B_{j}$ be
the connected components of $S_i \cup S_j$ containing $i^*$ and $j^*$,
respectively.

Then $ \vp_i(S^*)$ is
defined by the first case that applies below

  \begin{equation}    \label{eqn:phi}
    \vp_i(S^*) = \left\{ \begin{array}{rl}
        S^{*}	 
        & \mbox{if $i=h$, } \\[2ex]
        \basic_{i,j} (S^*) 
        & \mbox{if $S_{i}$ and $S_{j}$ are not abutting, or if $B_{i}$ and  $B_{j}$ have} \\
& \mbox{different shapes and neither nests $S_{h^{*}}$,}\\[2ex]
        \snake^{h}_{i,j} (S^*) 
        & \text{if }b_{h} \equiv b_{j} \text{ and } c_{h} \equiv c_{j}, \\[2ex]
        \basic_{i,j} \basic_{i,h}(S^{*}) 
        & \mbox{if  $S_{i}$ or $S_{j}$ nests $S_{h^*}$,}\\[2ex]
        \double^{h}_{i,j}(S^{*})  
        & \mbox{if  $B_{i}$ or $B_{j}$ nests $S_{h^*}$,}\\[2ex]
        \starswap_{i,j} (S^*) 
        & \mbox{if $B_{i} \neq B_{j}$ but they have the same shape.} 
      \end{array} \right.
  \end{equation}
\end{definition}

Here the map $\vp_{i}$ depends on four types of ribbon swaps: basic
swap, snake swap, double swap, and star swap. Each of the ribbon swaps
will only be well-defined under certain circumstances. As we prove in
Theorem~\ref{thm:phi}, the circumstances where a ribbon swap is
applied in \eqref{eqn:phi} will be precisely the circumstances when
the ribbon swap is well-defined.  The fact that these are all possible
cases can be observed from the fact that $h^{*}$ lies weakly between
$i^{*}$ and $j^{*}$ and the notation at the beginning of
Definition~\ref{defn:phi}.

The \emph{basic swap}, denoted $\basic_{i,j} (S^{*})$, is the result
of an interval swap on $S$ and interchanging the blocks containing
$i^{*}$ and $j^{*}$
\[
\basic_{i,j} (S^{*})=(\swap_{i,j}(S), c^{*} (B_{i} \leftrightarrow B_{j}) ).
\]
For example, if $n=4$ and $i=4$ then $\vp_{4} = \basic_{4,5}$ interchanges 
\begin{equation}\label{e:bswap.example}
\rnode{left}{%
  \tableau{
    5^{*} & \\
    4 & \\
    1^{*} & 2^{*} & 3^{*} & 4^{*} & \\
  }}
\hspace{ 5\cellsize}
\rnode{right}{%
  \tableau{
    4 & \\
    4^{*} & \\
    1^{*} & 2^{*} & 3^{*} & 5^{*} & \\
  }}
\psset{nodesep=8pt,linewidth=.1ex,offset=4pt}
\everypsbox{\scriptstyle}
\ncline[nodesepB=17pt]{<->} {left}{right} \naput{\basic_{4,5}}
\end{equation}
In the left tableau, $B_{4}$ is the cell of content 3 filled by
$4^{*}$ and $B_{5}$ is the set of cells with contents $\{-1,-2\}$
filled by $4,5^{*}$.  In the right tableau $B_{5}$ is the cell of
content 3 and $B_{4}$ is the set of cells with contents $\{-1,-2\}$.
Note, the star in the $\{-1,-2 \}$ block must move when applying the
map in either direction so as to return a valid starred strong tableau
with a star at the head of an $i$-ribbon and a $j$-ribbon.

A description of the operation $c^{*} (B_{i} \leftrightarrow B_{j})$
is given specifically as follows.  Let $d_{p}=c_{p}+1$ for each $p$ so
that the $p$-ribbons in $S^{*}$ correspond with applying the
transposition $t_{b_{p},d_{p}}$.  Let $r_{p}=d_{p}-b_{p}$ be the
length of a $p$-ribbon in $S^{*}$.  Let $\varepsilon_{p}$ be the unit
vector with a 1 in the $p$-th position.  Assume $p<q$, then define

\newcommand{\1}{p} 
\newcommand{\2}{q}
\newcommand{\cflop}{\mathrm{flop}}

\begin{equation}\label{e:flop}
\cflop_{\2,\1} (c^{*})=
\cflop_{\1,\2} (c^{*})=
\begin{cases}
t_{\1,\2}(c^{*}) - r_{\1} \cdot\varepsilon_{\1} & \text{if $d_{\1} \equiv
d_{\2}$ and $|B_{\1}|< |B_{\2}|$},
\\
t_{\1,\2}(c^{*}) - r_{\2} \cdot\varepsilon_{\2} & \text{if $d_{\1} \equiv
d_{\2}$ and $|B_{\1}|> |B_{\2}|$},
\\
t_{\1,\2}(c^{*})  + r_{\2} \cdot\varepsilon_{\2} & \text{if $b_{\2} \equiv
d_{\1}$ and $|B_{\1}| > |B_{\2}|$},
\\
t_{\1,\2}(c^{*}) + r_{\1} \cdot\varepsilon_{\1} & \text{if $b_{\1} \equiv
d_{\2}$ and $|B_{\1}|<|B_{\2}|$},
\\
 t_{\1,\2}(c^{*}) & \text{otherwise.}
\end{cases}
\end{equation}
Therefore, formally we define
\[
\basic_{i,j} (S^{*})=(\swap_{i,j}(S), \cflop_{i,j}(c^{*})).
\]
We prove $\basic_{i,j} (S^{*})$ is always a valid starred strong
tableau in Theorem~\ref{thm:phi}.

\begin{remark}
  Note that when $S^*$ is a Young tableau, it is impossible for the
cell containing $i$ to abut the cell containing $j$ when $h\neq i$ is
the witness. Therefore the required ribbon swap in this case will
always be
$\vp_{i}(S^{*})=\basic_{i,j}(S^{*})=(\swap_{i,j}(S),t_{i,j}(c^{*}))$. Hence
$\vp_i$ reduces to the usual elementary dual equivalence relation on
Young tableaux.
\end{remark}

The \emph{snake swap}, denoted $\snake^{h}_{i,j} (S^{*})$, is the
result of moving the stars on all three ribbons $i-1,i,i+1$ while
keeping the underlying strong tableau fixed.  If $i-1$ is the witness,
the moves are based on the permutation $231=t_{12}t_{23}$; if $i+1$ is
the witness, the moves are based on the permutation
$312=t_{23}t_{12}$.  Either way, $j$ will become the $i$-witness of
$\snake^{h}_{i,j} (S^{*})$.   Assuming $h$ is the witness, then

\begin{equation}\label{e:snake}
\snake^{h}_{i,j} (S^{*})=
\begin{cases}
(S,\ t_{i,j} t_{i,h}(c^{*})     - r_{j} \cdot\varepsilon_{i}  + r_{h}  \cdot\varepsilon_{h})  &	\text{if } (c_{j}<c_{i}) \text{ xor } (i<j),  \\
(S,\  t_{i,j} t_{i,h}(c^{*})  + r_{i} \cdot\varepsilon_{i} - r_{i}
\cdot\varepsilon_{h}) & \text{otherwise.}
\end{cases}
\end{equation}

We will show in the proof of Theorem~\ref{thm:phi} that
$\snake^{h}_{i,j}$ is only applied when $S_i \cup S_j$ and $S_i \cup
S_h$ are both single connected ribbons so
$[\lambda^{(i-2)},\lambda^{(i+1)}]$ is a chain by
Proposition~\ref{prop:chains}.  When $h=i+1$, the stars move away from
the diagonal of content $c_{h}$ along these ribbons and when $h=i-1$
the stars move in toward the diagonal of content $c_{h}$ along these
ribbons.  The star on the witness toggles between $h$ and $j$ by
sliding along the diagonal with content $c_{h}$.  For example, if
$n=2$ and $i=3$, then $\vp_{3} = \snake^{4}_{3,2}$ maps
\[
\rnode{left}{%
\tableau{
4 & \\
3^{*} & 4^{*} & \\
2 & 3 & 4 & \\
1^{*} & 2^{*} & 3 & 4 & \\
}}
\hspace{5\cellsize}
\rnode{right}{%
\tableau{
4^{*} & \\
3 & 4 & \\
2^{*} & 3 & 4 & \\
1^{*} & 2 & 3^{*} & 4 & \\
}}
\psset{nodesep=10pt,linewidth=.1ex,offset=4pt}
\everypsbox{\scriptstyle}
\ncline[nodesepB=20pt]{->} {left}{right} \naput{\snake^{4}_{3,2}}
\ncline[nodesepA=20pt]{->} {right}{left} \naput{\snake^{2}_{3,4}}
\]
The inverse map is given by $\snake^{2}_{3,4}$ applied to the tableau
on the right.

The \emph{double swap}, denoted $\double^{h}_{i,j}(S^{*})$, is the
result of two interval swaps on $S$ and another ``almost permutation''
of the three relevant indices in the content vector. Precisely,
\[
\double^{h}_{i,j}(S^{*})= 
\begin{cases}
\left(\swap_{i,j} \swap_{i,h}(S),\ t_{i, j} t_{i,h} (c^*)   + r_{h} \cdot\varepsilon_{h} \right) & \text{if } b_{h} \equiv b_{j}, \\
\left(\swap_{i,j} \swap_{i,h}(S),\ t_{i, j} t_{i,h} (c^*)   - r_{h} \cdot\varepsilon_{i} \right) & \text{if } c_{h} \equiv c_{j}.
\end{cases}
\] 
Since $\double^{h}_{i,j}$ is only applied when $B_{i}$ or $B_{j}$
nests $S_{h^{*}}$ but neither $S_{i}$ or $S_{j}$ nests $S_{h^{*}}$, we
can conclude that the nesting block is a ribbon and that $S_{j}$
contains a cell with the same content as either the head or tail of
$S_{h^{*}}$ by considering all possible rank 2 abutting rod exchanges.
Thus, when its applied either $b_{h} \equiv b_{j} \text{ or } c_{h}
\equiv c_{j}$.  For example, if $n=3$ and $i=4$, then $\vp_{4}$
interchanges the following tableaux via double swaps:
\[ 
\rnode{left}{%
\tableau{
5 & \\
5 & \\
4^{*} & 5 & \\
3^{*} & 5^{*} & \\
1^{*} & 2^{*} & 3 & \\
}}
\hspace{5\cellsize}
\rnode{right}{%
\tableau{
5 & \\
4 & \\
3 & 5^{*} & \\
3^{*} & 4^{*} & \\
1^{*} & 2^{*} & 5 & \\
}}
\psset{nodesep=10pt,linewidth=.1ex,offset=4pt}
\everypsbox{\scriptstyle}
\ncline[nodesepB=20pt]{->} {left}{right} \naput{\double^{3}_{4,5}}
\ncline[nodesepA=20pt]{->} {right}{left} \naput{\double^{5}_{4,3}}
\]
\bigskip

The \emph{star swap}, denoted $\starswap_{i,j} (S^{*})$, is the result
of moving the star on $i^{*}$ to the adjacent $j$-ribbon and vice
versa while keeping the underlying strong tableau fixed.  To be
precise, if $B_{i}$ and $B_{j}$ are distinct and both $B_{i}$ and
$B_{j}$ contain both an $i$ and $j$-ribbon, then both blocks have the
same shape by Proposition~\ref{prop:stack} and
Proposition~\ref{prop:chains}.  Say $f$ is the offset of the contents
of $B_{j}$ from $B_{i}$, so $c_{i}+f$ is the content of the head of
the $i$-ribbon in $B_{j}$ and $c_{j}-f$ is the content of the head of
the $j$-ribbon in $B_{i}$.  Then
\[
\starswap_{i,j} (S^{*})=(S, c^{*} + f\cdot\varepsilon_{i} - f\cdot\varepsilon_{j} ).  
\]
For example, if $n=4$ and $i=6$, then $\vp_{6} = \starswap_{6,7}$ interchanges 
\[
\rnode{left}{%
\tableau{
4^{*} & \\
3^{*} & 6^{*} & 7 & 7 & \\
1^{*} & 2^{*} & 4 & 5^{*} & 6 & 7 & 7^{*} & \\
}}
\hspace{5\cellsize}
\rnode{right}{%
\tableau{
4^{*} & \\
3^{*} & 6 & 7 & 7^{*} & \\
1^{*} & 2^{*} & 4 & 5^{*} & 6^{*} & 7 & 7 & \\
}}
\psset{nodesep=10pt,linewidth=.1ex,offset=4pt}
\everypsbox{\scriptstyle}
\ncline[nodesepB=20pt]{<->} {left}{right} \naput{\starswap_{6,7}}
\]

\subsection{A well-defined involution}
\label{sec:well.defined} 

Given the complicated definition of the affine dual equivalence
relations, it is not obvious that $\vp_i$ is well-defined, much less
that it is an involution. Our next task is to establish these two
facts. In the course of doing so, we provide many more examples of the
action of $\vp_i$, though in the interest of space only the relevant
cells in the strong tableaux are shown.

\begin{theorem}
  For each $1<i<m$, the map $\vp_i$ is a well-defined involution on
  all starred strong tableaux of a fixed $n$-core $\lambda$ of rank $m$.
\label{thm:phi}
\end{theorem}

\begin{proof}
Let $S^*$ be a starred strong tableau of shape $\lambda$.  We can
assume $i \neq h$ throughout the proof, the contrary case being
trivial. Suppose first that $S_i$ and $S_j$ are not abutting.  In this
case, a $\swap_{i,j}(S)$ is well defined and $b_{i},b_{j},d_{i},d_{j}$
are all distinct $\mathrm{mod}\ n$ by the classification of rod
exchanges for rank 2 intervals in Section~\ref{sec:poset-intervals}.
Unless $S_i$ and $S_j$ come from an interleaving rod exchange with some
$i$-ribbon nested in an $i+1$-ribbon or vice versa, the interval swap
will simultaneously change all $i$'s to $j$'s and
conversely. Therefore
$\vp_{i}(S^{*})=\basic_{i,j}(S^{*})=(\swap_{i,j}(S),t_{i,j}(c^{*}))$
is a well-defined starred strong tableau with stars in the original
cells in $S^{*}$, though now adorning the opposite letter among
$\{i,j\}$ from before. When $S_i$ and $S_j$ come from an interleaving
rod exchange with some $i$-ribbon nested in an $i+1$-ribbon or vice
versa, then the interval swap will change all entries in the shorter
ribbon appearing independently as well as entries in the longer ribbon
not on the same content as a shorter ribbon. In particular, the shape
of the blocks $B_i$ and $B_j$ remains unchanged. Therefore
$\basic_{i,j}(S^{*})$ is again a valid starred strong tableau.  In
this case, the star adorning the longer ribbon remains in place, and
the star adorning the shorter ribbon remains if the shorter ribbon is
not nested in a longer, otherwise it slides one position along the
diagonal; see Figure~\ref{fig:P-nested} for an example.

  Consequently, in order to show $\vp_{i}$ is an involution in this
case, it remains only to show that $h$ remains the witness after
applying $\basic_{i,j}$. Since the effect on the content vector is
merely to interchange $c_i$ and $c_j$, the result follows provided
$c_h \neq c_j$. However, the contrary case forces an $i$-ribbon to
abut both the $i-1$-ribbon and $i+1$-ribbon with heads on content
$c_{i-1} = c_{i+1}$. This ensures that $\basic_{i,j}$ is an involution
in this case.  

  \begin{figure}[ht]
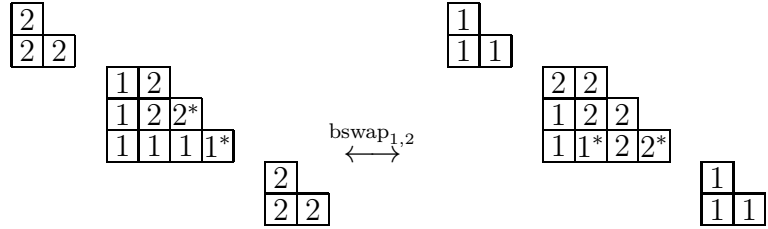

    \begin{center}
      \begin{displaymath}
    \tableau{2 \\
             2 & 2 \\ 
               &   & & 1 & 2 \\
               &   & & 1 & 2 & 2^* \\
               &   & & 1 & 1 & 1 & 1^* \\ 
               &   & &   &   &   &    &  & 2 \\
               &   & &   &   &   &    &  & 2 & 2 }
    \raisebox{-4\cellsize}{%
      $\displaystyle{\stackrel{\basic_{1,2}}{\longleftrightarrow}}$}
    \hspace{\cellsize}
    \tableau{1 \\
             1 & 1 \\
               &   &   & 2 & 2 \\
               &   &   & 1 & 2 & 2 \\
               &   &   & 1 & 1^* & 2 & 2^* \\
               &   &   &   &     &   &   & & 1 \\
               &   &   &   &     &   &   & & 1 & 1 }
      \end{displaymath}
      \caption{\label{fig:P-nested}The action of $\vp_i$ on $S^*$ when
        $S_i$ and $S_j$ are nested}
    \end{center}  
  \end{figure}

Henceforth, we will assume that $S_{i}$ and $ S_{j}$ are abutting, and
thus both $B_{i}$ and $B_{j}$ must have ribbon shape by
Corollary~\ref{cor:ribbons}.

  \begin{figure}[ht]
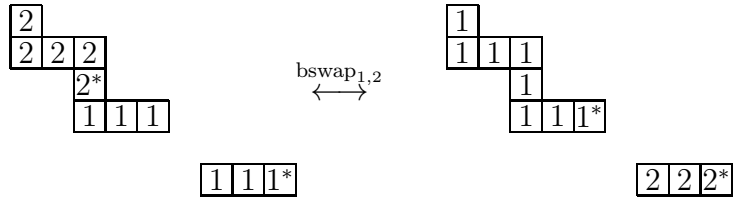

    \begin{center}
      \begin{displaymath}
        \tableau{2 \\
          2 & 2 & 2 \\
          &   & 2^* \\
          &   & 1   & 1 & 1 \\ \\
          &   &     &   &   & & 1 & 1 & 1^*}
        \raisebox{-2\cellsize}{%
          $\displaystyle{\stackrel{\basic_{1,2}}{\longleftrightarrow}}$}
        \hspace{2\cellsize}
        \tableau{1 \\
          1 & 1 & 1 \\
          &   & 1 \\
          &   & 1 & 1 & 1^* \\ \\
          &   &   &   &     & & 2 & 2 & 2^*}
      \end{displaymath}
      \caption{\label{fig:P-witness}The action of $\vp_i$ when $S_i
        \cup S_j$ is abutting and $S_h$ is not nested.}
    \end{center}
  \end{figure}

Assume that $B_{i}$ and $B_{j}$ have different (ribbon) shapes but
neither nests $S_{h^{*}}$.  Since $B_{i}$ and $B_{j}$ have different
shapes, the map $\swap_{i,j}$ will toggle between each block
containing only one letter and exactly one of these blocks containing
both letters as shown in Figure~\ref{fig:P-witness}.  By
Proposition~\ref{prop:stack} one can deduce how the stars move in the
blocks $B_{i}$ and $B_{j}$ in order to adorn the other letter.  These
moves are summarized in the function $\cflop_{i,j}(c^{*})$, hence
$T^{*}=\vp_{i}(S^{*})=\basic_{i,j}(S^{*})$ is a well defined starred
strong tableau.  Furthermore, by inspection we have that
$\basic_{i,j}(T^{*})=S^{*}$.  Thus, $\vp_{i}(S^{*})$ is an involution
provided $h$ is also the $i$-witness of $T^{*}$.

Observe that the only way for the witness to change is if $h^*$ lies
on a diagonal within a block containing both $i$'s and $j$'s, and
$h^*$ lies weakly between their respective heads.  Let $I$ and $J$ be
the abutting $i$-ribbon and $j$-ribbon in the block overlapping
$h^{*}$.  By Proposition~\ref{prop:overlap}, consecutive ribbons may
not have partially overlapping contents. Therefore if an $h$-ribbon
has content overlapping an $i$-ribbon, one of the two must be
nested. By assumption, $S_{h^{*}}$ is not nested inside either $B_{i}$
or $B_{j}$, hence is not nested in $I$.  On the other hand, if
$h$-ribbons nest $i$-ribbons, then they must also nest $j$-ribbons,
otherwise a $\swap_{i,h}(S)$ is possible and will leave $i$-ribbons
and $j$-ribbons with partially overlapping contents, again
contradicting Proposition~\ref{prop:overlap}.  Therefore $h$-ribbons
and $i$-ribbons may not have overlapping contents, so the cell
containing $h^{*}$ must overlap $J$ in content.

  If $J$ lies southeast of $I$, this forces the $h^{*}$-ribbon to
overlap $I$ or be nested in $J$, neither of which is possible. Thus
$J$ must lie northwest of $I$, hence the head of the $h$-ribbon is
forced to have the same content as the head of $J$.  Furthermore, the
$h^{*}$-ribbon must be longer than $J$ since by assumption $h$-ribbons
are not nested in $I \cup J$. Therefore when $h^*$ and $j^*$ lie on
the same diagonal, $h$ remains the $i$-witness for $\vp_i(S^*)$ using
part (2) of Definition~\ref{def:witness}.  
 
  Next consider the case where $c_{i-1} \equiv c_{i+1}$ and $b_{i-1}
  \equiv b_{i+1}$.  Then, $t_{b_{i-1},d_{i-1}}= t_{b_{i+1},d_{i+1}}$
  as affine permutations, where recall $d_{p}=c_{p}+1$.  By
  Corollary~\ref{cor:transposition.seq.bijection}, $S^{*}$ is
  associated to a transposition sequence.  In order for
  $t_{b_{i-1}d_{i-1}}\rightarrow t_{b_{i}d_{i}}\rightarrow
  t_{b_{i+1}d_{i+1}}$ to be a valid triple in the transposition
  sequence, $t_{b_{i}d_{i}}$ must not commute with the other two.
  Hence at least one $i+1$-ribbon completely overlaps some
  $i-1$-ribbon, sharing both a head and tail, and $i$-ribbons
  must abut each such pair from both sides. By
  Proposition~\ref{prop:chains}, this means $S_i \cup S_j$ and $S_i
  \cup S_h$ are both ribbons, hence $\snake_{i,j}^{h}$ is
  well-defined on $S^{*}$.  By inspecting \eqref{e:snake}, we see that
  $T^{*}=\snake_{i,j}^{h}(S^{*})$ is a starred strong tableau on the
  same underlying strong tableau $S$ with $j$ as its $i$-witness and
  $\snake_{i,h}^{j}(T^{*})=S^{*}$, so $\vp_i$ is an involution in this
  case as well.  For example, see Figure~\ref{fig:doublesnakes}.

  \begin{figure}[ht]
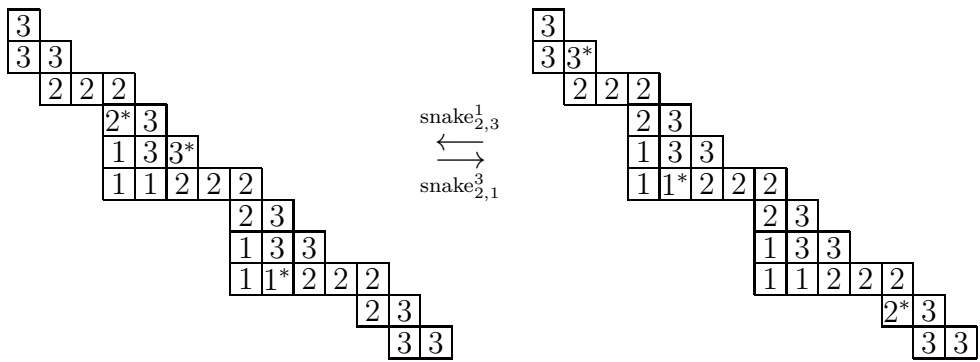

    \begin{center}
      \begin{displaymath}
        \tableau{%
          3 \\ 
          3 & 3 \\ 
          & 2 & 2 & 2\\
          & & & 2^* & 3 \\           
          & & & 1 & 3 & 3^*\\        
          & & & 1 & 1 & 2 & 2 & 2\\
          & & & & & & & 2 & 3 \\           
          & & & & & & & 1 & 3 & 3\\        
          & & & & & & & 1 & 1^* & 2 & 2 & 2\\
          & & & & & & & & & & & 2 & 3\\          
          & & & & & & & & & & & & 3 & 3}
        \hspace{-\cellsize}
        \raisebox{-4\cellsize}{%
     $\displaystyle\mathop{\stackrel{\displaystyle\longleftarrow}{\longrightarrow}}_{\snake_{2,1}^{3}}^{\snake_{2,3}^{1}}$}
        \hspace{\cellsize}
        \tableau{
          3 \\ 
          3 & 3^* \\ 
          & 2 & 2 & 2\\
          & & & 2 & 3 \\           
          & & & 1 & 3 & 3\\        
          & & & 1 & 1^* & 2 & 2 & 2\\
          & & & & & & & 2 & 3 \\           
          & & & & & & & 1 & 3 & 3\\        
          & & & & & & & 1 & 1 & 2 & 2 & 2\\
          & & & & & & & & & & & 2^* & 3\\          
          & & & & & & & & & & & & 3 & 3}
      \end{displaymath}
      \caption{\label{fig:doublesnakes}The action of $\vp_i$ when
        $i-1$-ribbons and $i+1$-ribbons share a head and tail.}
    \end{center}
  \end{figure}    

  Henceforth, we will assume that $S_{i}$ and $ S_{j}$ are abutting
  and either $c_{i-1} \not \equiv c_{i+1}$ or $b_{i-1} \not \equiv
  b_{i+1}$.   The next case to consider is when
  $B_{i}$ or $B_{j}$ nests $S_{h^*}$, including the possibility that
  $S_{i}$ or $S_{j}$ nests $S_{h^*}$.  Note that if $B_{i} = B_{j}$,
  then $B_i$ necessarily nests $S_{h^*}$ in order for $h$ to be the
  witness.

  We claim that in all these cases some connected component of $S_h
  \cup S_i \cup S_j$ is a single $h$-ribbon.  If some $h$-ribbon is
  nested in an $i$-ribbon or a $j$-ribbon, then the claim follows
  immediately from Lemma~\ref{lem:nested.ribbons}.  So, assume that
  some connected component of $S_i \cup S_j$ contains both $i$'s and
  $j$'s and nests an $h$-ribbon.  By Proposition~\ref{prop:nested}, we
  may further assume that the nested $h$-ribbon shares a head or tail
  with the $j$-ribbon. Necessarily the $h$-ribbon and $j$-ribbon must
  both abut an $i$-ribbon at their shared content in the strong
  tableau.  Therefore we have one of the following scenarios for the
  three transpositions corresponding to $i-1,i,i+1$ on the abacus,
  where $a < b < c < d \leq a+n$.
  \begin{equation}
    \label{e:triple.rod.exchange}
    \begin{array}{cc} \begin{array}{rl}
        \scriptstyle{a} & \rnode{ae}{\rule{4em}{1pt}} \\
        \scriptstyle{b} & \rnode{aa}{\rule{4em}{1pt}} \\
        \scriptstyle{c} & \rnode{ab}{\rule{4em}{1pt}} \\
        \scriptstyle{d} & \rnode{ac}{\rule{4em}{1pt}}
      \end{array} \hspace{2em} & \hspace{2em}
      \begin{array}{rl}
        \scriptstyle{a} & \rnode{ee}{\rule{4em}{1pt}} \\
        \scriptstyle{b} & \rnode{ea}{\rule{4em}{1pt}} \\
        \scriptstyle{c} & \rnode{eb}{\rule{4em}{1pt}} \\
        \scriptstyle{d} & \rnode{ec}{\rule{4em}{1pt}}
      \end{array}
    \end{array}
    \everypsbox{\scriptstyle}
    \psset{nodesep=5pt,linewidth=.1ex}
    \ncdiag[angle=0,linearc=.2]{<->}{ae}{aa} 
    \ncdiag[angle=0,linearc=.4]{<->}{aa}{ac} 
    \ncdiag[angle=0,linearc=.2,nodesep=2pt]{<->}{aa}{ab} 
    \ncdiag[angle=0,linearc=.2,nodesep=2pt]{<->}{ea}{eb} 
    \ncdiag[angle=0,linearc=.4]{<->}{ee}{eb} 
    \ncdiag[angle=0,linearc=.2]{<->}{eb}{ec} 
  \end{equation}


  To ease notation we assume $h=i-1$ and $j=i+1$ and we are doing the
  rod exchange on the left in \eqref{e:triple.rod.exchange}, noting
  that the other cases are completely analogous. In this case,
  analyzing the transposition triple $t_{b,c}\rightarrow
  t_{a,b}\rightarrow t_{b,d}$ shows that initially the length of
  $\mathrm{rod}(b)$ cannot be weakly between the lengths of
  $\mathrm{rod}(a)$ and $\mathrm{rod}(d)$ or else the transpositions
  don't each increase the rank by exactly one at each step.
  Furthermore, the length of $\mathrm{rod}(b)$ cannot be less than the
  length of $\mathrm{rod}(d)$ because otherwise there would be no
  $i+1$-ribbon with content overlapping any $i-1$-ribbon by the
  definition of $T(\lambda,n)$ and
  Corollary~\ref{cor:transposition.seq.bijection} contradicting the
  assumption that some connected component of $S_i \cup S_j$ contains
  both $i$'s and $j$'s and nests an $h$-ribbon.  Therefore, the length
  of $\mathrm{rod}(b)$ is strictly greater than the length of
  $\mathrm{rod}(a)$ so, by
  Corollary~\ref{cor:transposition.seq.bijection} again, there must an
  $i-1$-ribbon occurring independently from all $i,i+1$-ribbons.

  With the claim proved, we conclude by Proposition~\ref{prop:chains}
that $T=\swap_{i,h}(S)$ is a well defined, valid strong tableau. After
such, some $i$-ribbon appears independently of all $j$-ribbons in $T$,
making $U=\swap_{i,j}(T)$ well defined.

  In the case $B_{i}$ or $B_{j}$ nests $S_{h^*}$ but neither $S_{i}$
or $S_{j}$ nests $S_{h^*}$, the final result of $\swap_{i,j}
(\swap_{i,h}(S))$ is much like a single interval swap in the case of
nested $i,i+1$-ribbons: all independently occurring $h$'s change to
$j$'s and all letters of $S_{i} \cup S_{j}$ not on the same diagonal
as an $h$ will change with $i$'s becoming $h$'s and $j$'s becoming
$i$'s; for example, see Figure~\ref{fig:snakes}. The shape and
contents of the nested ribbon remains the same, but now these are
$j$-ribbons. Therefore, $U^{*}=\double_{i,j}^{h}(S)$ is a well-defined
starred strong tableau with a star placed at the head of some
$p$-ribbon for each $p$.  The effect of $\vp_{i}$ on the content
vector for $S^{*}$ is an involution by inspection.  Since $j$-ribbons
are now nested in $U^{*}$, we only need to show $j$ becomes the
$i$-witness in $U^{*}$ in order to prove $\vp_{i}$ is an involution on
such an $S^{*}$. This will clearly be the case so as long as $c_{h}
\neq c_{j}$ both before and after applying $\double_{i,j}^{h}$.
Assuming $c_{h} = c_{j}$, an $i$-ribbon will be forced to lie
southeast of $S_{h^{*}}$ and $S_{j^{*}}$. However, after applying
$\vp_i(S^{*})$, $h$-ribbons and $j$-ribbons will share a tail instead,
so the witness indeed changes as desired.

  \begin{figure}[ht]
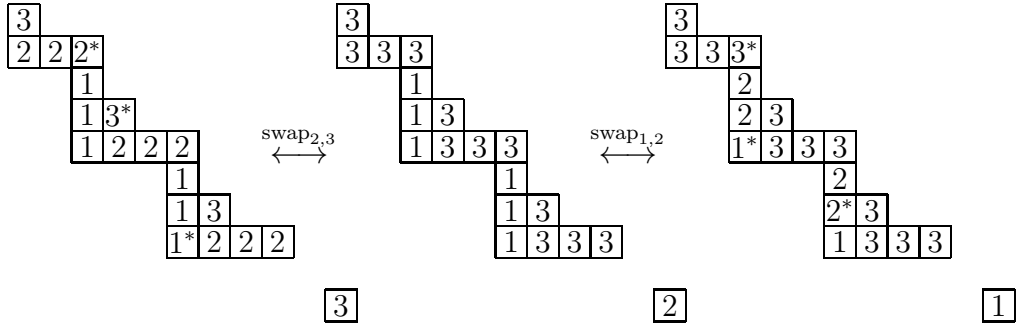

    \begin{center}
      \begin{displaymath}
        \tableau{3 \\
          2 & 2 & 2^*\\
          &   & 1 \\
          &   & 1 & 3^*\\
          &   & 1 & 2 & 2 & 2\\
          &   &   &   &   & 1\\
          &   &   &   &   & 1 & 3\\
          &   &   &   &   & 1^* & 2 & 2 & 2\\ \\
          &   &   &   &   &   &   &   &  & & 3}
        \hspace{-3\cellsize}
        \raisebox{-4\cellsize}{$\stackrel{\swap_{2,3}}{\longleftrightarrow}$}
        \tableau{3 \\
          3 & 3 & 3\\
          &   & 1 \\
          &   & 1 & 3\\
          &   & 1 & 3 & 3 & 3\\
          &   &   &   &   & 1\\
          &   &   &   &   & 1 & 3\\
          &   &   &   &   & 1 & 3 & 3 & 3\\ \\
          &   &   &   &   &   &   &   &  & & 2}
        \hspace{-3\cellsize}
        \raisebox{-4\cellsize}{$\stackrel{\swap_{1,2}}{\longleftrightarrow}$}
        \tableau{3 \\
          3 & 3 & 3^*\\
          &   & 2 \\
          &   & 2 & 3\\
          &   & 1^* & 3 & 3 & 3\\
          &   &   &   &   & 2\\
          &   &   &   &   & 2^* & 3\\
          &   &   &   &   & 1 & 3 & 3 & 3\\ \\
          & & & & & & & & & & 1} \end{displaymath}
      \caption{\label{fig:snakes}The action of $\vp_2=\double_{i,j}^{h}$
        when $S_{i^*} \cup S_{j^*}$ nests $S_{h^*}$.}  \end{center}
  \end{figure}
    
  Similarly, in the case $S_{i}$ or $S_{j}$ nests $S_{h^*}$, the image
  $U^{*}=\vp_{i}(S^{*})=\basic_{i,j} (\basic_{i,h}(S))$ is a
  well-defined starred strong tableau with a star placed at the head
  of some $p$-ribbon for each $p$ by the proof of a single basic swap
  above.  Either $U_{i}$ or $U_{h}$ nests $j^{*}$ and $j$ is the
  $i$-witness for $U^{*}$.  Hence $\vp_{i}^{2}(S^{*})=S^{*}$ for such
  an $S^{*}$.

  Finally, we will assume $B_{i}$ and $B_{j}$ have the same shape but
lie on distinct content diagonals, each have ribbon shape, and neither
nests $S_{h^{*}}$.  Then $\vp_{i}(S^{*})=\starswap_{i,j}(S^{*})$ is a
well defined starred strong tableau where the same hypotheses hold.
See Figure~\ref{fig:P-starswap} for an example.

  \begin{figure}[ht]
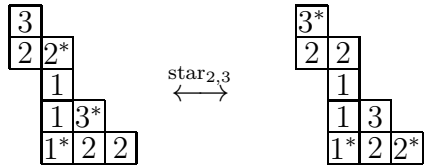

    \begin{center}
      \begin{displaymath}
    \tableau{3 \\
             2 & 2^* \\
               & 1   \\
               & 1   & 3^*\\
               & 1^* & 2 & 2}
    \hspace{\cellsize}
    \raisebox{-2\cellsize}{%
      $\displaystyle{\stackrel{\starswap_{2,3}}{\longleftrightarrow}}$}
    \hspace{2\cellsize}
    \tableau{3^* \\
             2 & 2   \\
               & 1   \\
               & 1   & 3\\
               & 1^* & 2 & 2^*}
      \end{displaymath}
      \caption{\label{fig:P-starswap}The action of $\vp_i$ when $h^*$
        overlaps $S_{i^*} \cup S_{j^*}$ without being nested.}
    \end{center}
  \end{figure}

In this case, both blocks $B_{i}$ and $B_{j}$ contain both $i$'s and
$j$'s.  The only way for the witness to change is if $h^*$ lies on a
diagonal within a block containing both $i$'s and $j$'s, and $h^*$
lies weakly between their respective heads.  The proof that $h$
remains the witness is the same as the argument above for the case
$B_{i}$ and $B_{j}$ have different shape and neither nests
$S_{h^{*}}$.
\end{proof}

\subsection{Preservation of spin}
\label{sec:graph-spin}

Next we show that the involution $\vp_i$ has the added feature of
preserving the $\spin$ statistic. Recall from \eqref{eqn:spin} that
$\spin$ is defined by
\begin{displaymath}
  \spin(S^*) = \sum_{i} n(i) \cdot (h(i) - 1) + d(i^*),
\end{displaymath}
where $n(i)$ is the number of $i$-ribbons, $h(i)$ is the height of an
$i$-ribbon and $d(i^*)$ is the depth of the starred $i$-ribbon. We
will show that $\vp_i$ preserves the spin by tracking the contribution
for $i-1, i$ and $i+1$.

\begin{proposition}
  For any starred strong tableau $S^*$, we have $\spin(\vp_i(S^*)) =
  \spin(S^*)$.
\label{prop:spin}
\end{proposition}

\begin{proof}
  Recall the notation from Definition~\ref{defn:phi}.  Assume $i\neq
h$.  If $S_i$ and $S_j$ are disjoint, interleaving or nested with
non-overlapping content, then $\vp_i$ acts by simultaneously replacing
all $i$'s with $j$'s and conversely. The contribution to $\spin$ for
ribbons other than $i,j$ is unchanged, and these two swap
contributions, thereby preserving the statistic. If any two ribbons of
$S_i$ and $S_j$ are nested with overlapping contents, then recall that
$\vp_i$ does not change the shape of the shorter ribbon nor the height
(nor width) of the longer ribbon, and the stars remain on the same
diagonals. This ensures that contributions to $\spin$ for $i$ and $j$
are exchanged, and all other contributions are unchanged.

  We may now assume that $S_i$ and $S_j$ are abutting for the remainder
  of the proof. If $\vp_i$ acts by $\starswap_{i,j}$, then this
  affects only the depths of $i^*$ and $j^*$.  We claim
  $d(i^{*})+d(j^{*})$ is preserved since every connected component of
  $S_{i} \cup S_{j}$ inclusively between $B_{i}$ and $B_{j}$ have the
  same shape when $\starswap_{i,j}$ is applied.  Hence, $\vp_i$ again
  preserves $\spin$.

  If $\vp_i$ acts by a $\basic_{i,j}$, then recall from the proof of
  Theorem~\ref{thm:phi} that each connected component of $S_{i} \cup
  S_{j}$ is a ribbon, in particular one of $B_{i}$ and $B_{j}$ is a
  longer ribbon containing an $n$-translate of the shorter.  Let $n_l$
  and $n_s$ denote the number of the longer ribbons and shorter
  ribbons in $S_{i} \cup S_{j}$, respectively, and let $h_l$ and $h_s$
  denote their respective heights. Let $d_l$ be the number of longer
  ribbons northwest of the starred long ribbon, and similarly let
  $d_s$ denote the number of shorter ribbons northwest of the starred
  short ribbon.

  Supposing that the connected components of $S_{i} \cup S_{j}$ each
  contain a unique letter, the contributions for $i$ and $j$ to
  $\spin$ are
  \begin{eqnarray} 
    \spin_{S^*}(i) & = & n_s (h_s - 1) + d_s, \label{e:distinct.letters.i} \\
    \spin_{S^*}(j) & = & n_l (h_l - 1) + d_l. \label{e:distinct.letters.j}
  \end{eqnarray}
  On the other hand, letting $T^* = \basic_{i,j}(S^*)$, some connected
  component of $T_{i} \cup T_{j}$ contains both $i$'s and $j$'s. By
  Proposition~\ref{prop:stack}, this implies that in every component
  containing both letters, the smaller entries are south of the larger
  entries if and only if the shorter ribbons appear independently to
  the southeast. Armed with this observation, we compute that if the
  longer ribbon among $B_{i}$ and $B_{j}$ contains both $i$'s and
  $j$'s, then the contribution to $\spin$ for $i$ and $j$ in $T^*$ is
  \begin{eqnarray} 
    \spin_{T^*}(i) & = & (n_s + n_l)(h_s -1) + (d_s + \varepsilon n_l),
    \label{e:nondistinct.letters.i} \\
    \spin_{T^*}(j) & = & n_l(h_l - h_s + (1-\varepsilon) - 1) + d_l,
    \label{e:nondistinct.letters.j}
  \end{eqnarray}
  where $\varepsilon$ is $1$ if the shorter ribbons appear
  independently southeast of the longer (equivalently, the larger
  entry abuts the shorter from the north), and $0$ if the shorter
  ribbons appear independently northwest of the longer (equivalently,
  the smaller entry abuts the larger from the west). Noting the
  equality between \eqref{e:distinct.letters.i} plus
  \eqref{e:distinct.letters.j} and \eqref{e:nondistinct.letters.i}
  plus \eqref{e:nondistinct.letters.j} shows $\spin$ is once again
  preserved. This also handles the case when $\vp_i$ acts by
  $\basic_{i,j}\basic_{i,h}$.

  Consider now the case when $\vp_i$ acts by $\double_{i,j}^{h}$. For
  this case, we may assume, from the analysis in
  Theorem~\ref{thm:phi}, that some $h$-ribbon and $j$-ribbon share a
  head or tail. Also from Theorem~\ref{thm:phi}, some $h$-ribbon must
  appear independently of all $i$-ribbons. Furthermore, by
  Proposition~\ref{prop:stack}, if $i$- or $j$-ribbons appear
  independently of the other, then they do so on the opposite side of
  abutting $i$- and $j$-ribbons from $h$-ribbons. Supposing first that
  the combined lengths of an $i$-ribbon plus a $j$-ribbon is less than
  $n$, reading from northwest to southeast or from southeast to
  northwest one sees isolated $h$-ribbons followed by abutting $i$-
  and $j$-ribbons nesting $h$-ribbons. There are then three options
  for what follows: isolated $j$-ribbons; isolated abutting $i$- and
  $h$-ribbons; or no further $i$-, $j$- or $h$-ribbons.  For example,
  see Figure~\ref{fig:spin.ex}. Note that, in particular, $S^*$ has
  isolated $j$-ribbons if and only if $\double_{i,j}^{h}(S^*)$ has
  isolated abutting $i$- and $h$-ribbons.

  \begin{figure}[ht]
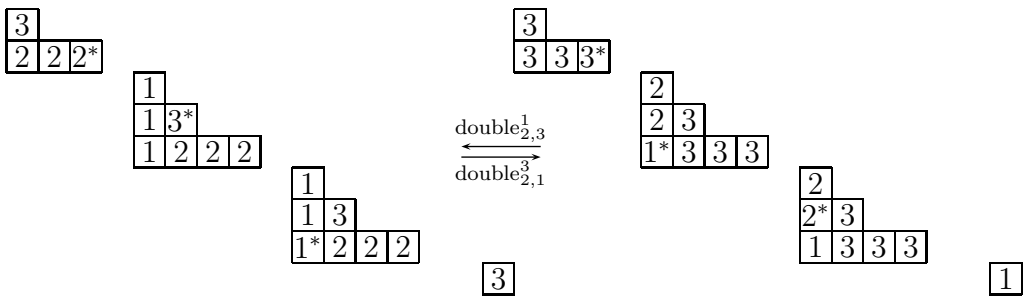

    \begin{center}
      \begin{displaymath}
        \rnode{left}{%
        \tableau{3 \\
          2 & 2 & 2^*\\
          & & &   & 1 \\
          & & &   & 1 & 3^*\\
          & & &   & 1 & 2 & 2 & 2\\
          & & &   &   &   &   & & & 1\\
          & & &   &   &   &   & & & 1 & 3\\
          & & &   &   &   &   & & & 1^* & 2 & 2 & 2\\
          & & &   &   &   &   & & &   &   &   &  & & & 3}}
        \rnode{right}{%
        \tableau{3 \\
          3 & 3 & 3^*\\
          & & &   & 2 \\
          & & &   & 2 & 3\\
          & & &   & 1^* & 3 & 3 & 3\\ 
          & & &   &   &   &   & & & 2\\
          & & &   &   &   &   & & & 2^* & 3\\
          & & &   &   &   &   & & & 1 & 3 & 3 & 3\\ 
          & & & & & & & & & & & & & & & 1}}
        \psset{nodesep=10pt,linewidth=.1ex,offset=4pt}
        \everypsbox{\scriptstyle}
        \ncline[nodesepA=10pt,nodesepB=20pt,offset=2pt]{->} {left}{right}%
        \naput{\double_{2,1}^{3}}
        \ncline[nodesepA=20pt,nodesepB=10pt,offset=2pt]{->} {right}{left}%
        \naput{\double_{2,3}^{1}}
      \end{displaymath}
      \caption{\label{fig:spin.ex}An example of $\vp_i$ acting via
        $\vp_i=\double_{i,j}^{h}$.}  \end{center}
  \end{figure}

  To assess the contributions to $\spin$, assume that $S^*$ has no
  isolated abutting $i$- and $h$-ribbons, as in the right hand side of
  Figure~\ref{fig:spin.ex}. Let $h_w$ and $d_w$ denote the height and
  depth of the starred $h$-ribbon, respectively, and let $n_w$ denote
  the number of isolated witness ribbons. For the example, we have
  $n_w = 1$, $h_w = 1$, $d_w = 0$. Let $n_l$ be the number of
  $i$-ribbons, each with height $h_l$ and the starred one with depth
  $d_l$. Let $n_s$ be the number of isolated $j$-ribbons, and let
  $h_s$ and $d_s$ denote the height and depth, respectively, of
  $j$-ribbons. For the example, we have $n_l = 2$, $h_l = 2$, $d_l =
  1$ and $n_s=1$, $h_s = 2$, $d_s=0$.

  The contribution to $\spin$ from $i-1,i,i+1$ in $S^*$, where $S^*$
  has no isolated abutting $i$- and $h$-ribbons, is given by
  \begin{eqnarray} 
    \spin_{S^*}(h) & = & (n_l + n_w) (h_w - 1) + d_w , \\
    \spin_{S^*}(i) & = & n_l (h_l - 1) + d_l , \\ 
    \spin_{S^*}(j) & = & (n_l + n_s) (h_s - 1) + d_s .
  \end{eqnarray}
  Following the description of how $\double_{i,j}^{h}$ acts on these
  ribbons, we may similarly compute the contributions of $i-1,i,i+1$
  to the $\spin$ of $T^{*}=\double_{i,j}^{h}(S^*)$. With $h$ and $j$
  defined relative to $T^*$, we have
  \begin{eqnarray} 
    \spin_{T^*}(h) & = & (n_l + n_w + n_s)(h_w - 1) + (d_w + \varepsilon n_s),\\
    \spin_{T^*}(i) & = & (n_l + n_s)(h_s - h_w + (1-\varepsilon) - 1) + d_s,\\
    \spin_{T^*}(j) & = & n_l (h_l + h_w - (1-\varepsilon) - 1) + d_l, 
  \end{eqnarray}
  where, similar to before, $\varepsilon$ is $0$ if the witness
  originally existed only to the left of abutting $i$- and $j$-ribbons
  and $1$ otherwise. Adding the contributions in either case
  miraculously yields the same result, thereby showing that the spin
  statistic is preserved.

  If $i$-ribbons and $j$-ribbons have lengths adding to $n$, we regard
  the abutting $i$- and $j$-ribbons which together nest an $h$-ribbon
  as abutting pairs, and the leftover $\max(i,j)$-ribbon as
  isolated. For example, in Figure~\ref{fig:snakes}, we regard the
  left side as having $1$-ribbons abutting $2$-ribbons from the west
  with an isolated abutting $2$-ribbon and $3$-ribbon to the
  northwest, and the right side we regard as having $2$-ribbons
  abutting $3$-ribbons from the west with an isolated $3$-ribbon to
  the northwest. That is to say, Figure~\ref{fig:snakes} is the same
  as Figure~\ref{fig:spin.ex} for the purposes of calculating spin. In
  this case, note that $S^*$ has no isolated abutting $i$- and
  $h$-ribbons precisely when  $h=i-1$. Moreover, in this case we
  always have $n_s = 1$. With this alteration, the analysis of spin is
  precisely as before, again showing that spin is preserved.

  Finally, if $\vp_i$ acts by $\snake_{i,j}^{h}$, then the difference
  $\spin(S^{*})-\spin(\vp_{i}(S^{*}))$ only depends on the change in
  depth for $h^{*},i^{*},j^{*}$ since both $S^{*}$ and $\vp_{i}(S^{*})$
  have the same underlying strong tableau.  Furthermore, we have that
  $i-1$-ribbons and $i+1$-ribbons both have length $n$ minus the
  length of an $i$-ribbon. In this case, there is one more
  $i+1$-ribbon than $i$-ribbon and one more $i$-ribbon than
  $i-1$-ribbon.  Using the intuitive definition of $\snake_{i,j}^{h}$
  following \eqref{e:snake} we see that moving the witness from $i-1$
  to $i+1$ increases the depth of the witness by one, and similarly
  moving from $i+1$ to $i-1$ decreases the depth by one. As the stars
  on $i$ and $j$ move in or out along their respective ribbons, one
  star necessarily moves to an abutting ribbon joined on an east/west
  edge and the other star moves to abutting ribbon joined on a
  north/south edge.  Moving a star across a north/south edge will not
  change the depth of the star, but moving a star across an east or
  west edge will increase or decrease the depth by one, respectively,
  canceling the contribution from moving the witness.  Therefore the
  total contribution to $\spin$ from $i-1,i,i+1$ remains the same
  after applying $\snake_{i,j}^{h}$. All cases are now covered.
\end{proof}

The results in Theorem~\ref{thm:phi} and Proposition~\ref{prop:spin}
naturally extend to skew partitions as well since the proofs only
involve intervals of rank 3 in the $n$-core poset.

\begin{corollary}\label{cor:phi.skew.cores}
  Let $\mu \subset \nu $ be $n$-cores of lengths $\ell(\mu)=p$ and
  $\ell(\nu)=q$. Then, for $p<i<q$, the map $\vp_i$ is a well-defined,
  spin preserving, involution on all starred strong skew tableaux for
  $\nu /\mu$. In particular, $\spin$ is constant on \emph{affine
    dual equivalence classes}.
\end{corollary}

\section{A graph on starred strong tableaux}
\label{sec:graph}

In this section, we construct a vertex-signed, edge-colored graph from
our elementary affine dual equivalence map $\vp_i$. The main goal of
this section is to show that this graph is, in fact, an $\lsp_{2}$
graph by Definition~\ref{defn:lsp}.  In order to establish this, we
introduce two operations on starred strong tableaux which together
show that there are only finitely many isomorphism types for 2-colored
connected components.  The reduction to finitely many isomorphism
types can be viewed as an (incomplete) analog of the jeu da taquin
algorithm for starred strong tableaux.  This analogy is summarized in
Remark~\ref{r:jdt}.

\begin{definition}\label{def:Gnu}
  For an $n$-core $\nu$, the \textit{affine dual equivalence graph}
$\G_{\nu}^{(n)}$ is the signed, colored graph with vertex set $V_{\nu}$
given by the set of all starred strong tableaux $S^{*}$ of shape
$\nu$, with signature function $\sigma(S^{*})$ obtained from the
reading word on the starred letters in $S^{*}$, and for each
$1<i<\ell(\nu)$, the set of $i$-colored edges, $E_i$, is the set of
all pairs $\{S^{*},\vp_i(S^{*})\}$ such that $S^{*} \neq
\vp_i(S^{*})$.  This definition also extends to skew shapes $\nu /\mu$
in the $n$-core poset.    For $S^{*} \in \SST(\nu/\mu, n)$, let $[S^{*}]$ denote the
\textit{connected component of the affine dual equivalence graph}
$\G_{\nu /\mu}^{(n)}$ \textit{containing} $S^{*}$.
\end{definition}

For example, for $n=3$ and $\mu =(5,3,1)$ the affine dual equivalence
graph is shown on page~\pageref{fig:interval531}.

Recall that $\vp_i$ is an involution which preserves the spin
statistic by Corollary~\ref{cor:phi.skew.cores}. In order to justify
our terminology of affine dual equivalence.  We want to prove that the
graph induced by these involutions satisfies Axioms 1,2,3,5 from
Definition~\ref{defn:deg} and local Schur positivity on all two
adjacent colored connected components.  Thus each affine dual
equivalence graph is $\lsp_{2}$.  The key
will be reducing local Schur positivity to a finite verification.  The
reduction is achieved with the help of flattening rows and squashing
and/or cloning columns.

\subsection{The flattening map}
\label{sec:flattening.map} 

Here we define an iterative procedure to flatten an $n$-core
partition down to an $m$-core partition for any $1\leq m<n$. We will
extend this procedure to starred strong tableaux in a way that
commutes with the affine dual equivalence involutions.

\begin{definition}\label{defn:flattening.cores}
  For any $m+1$-core $\lambda$ and any $1 \leq d \leq m+1$, define
  $\lambda^{(d)}$ to be the unique partition associated to the binary
  string obtained by removing all beads and spacers with content
  congruent to $d$ modulo $m+1$ from the abacus of $\lambda$. In
  particular, $\lambda^{(d)}$ is an $m$-core.
\end{definition}

We note that the above definition makes sense in light of
Remark~\ref{rmk:bipart} and the characterization of $n$-cores in terms
of the $n$-rod abacus. For example, regarding $(7,4,4,2,2)$ as a
$4$-core, $(7,4,4,2,2)^{(2)}$ is the $3$-core $(6,4,2)$.
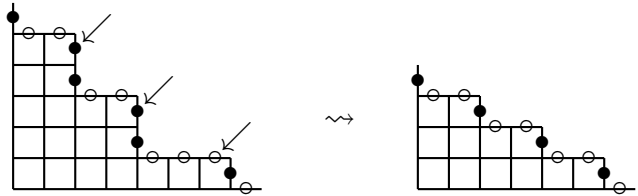
\begin{displaymath}
  \psset{unit=1em}
  \pspicture(0,0)(8,6)
  \psline(0,0)(0,6)
  \psline(1,0)(1,5)
  \psline(2,0)(2,5)
  \psline(3,0)(3,3)
  \psline(4,0)(4,3)
  \psline(5,0)(5,1)
  \psline(6,0)(6,1)
  \psline(7,0)(7,1)
  \psline(0,0)(8,0)
  \psline(0,1)(7,1)
  \psline(0,2)(4,2)
  \psline(0,3)(4,3)
  \psline(0,4)(2,4)
  \psline(0,5)(2,5)
  \rput(0,5.5){$\bullet$}
  \rput(2,4.5){$\bullet$}
  \rput(2,3.5){$\bullet$}
  \rput(4,2.5){$\bullet$}
  \rput(4,1.5){$\bullet$}
  \rput(7,0.5){$\bullet$}
  \rput(0.5,5){$\circ$}
  \rput(1.5,5){$\circ$}
  \rput(2.5,3){$\circ$}
  \rput(3.5,3){$\circ$}
  \rput(4.5,1){$\circ$}
  \rput(5.5,1){$\circ$}
  \rput(6.5,1){$\circ$}
  \rput(7.5,0){$\circ$}
  \rput(2.7,5.2){$\swarrow$}
  \rput(4.7,3.2){$\swarrow$}
  \rput(7.2,1.7){$\swarrow$}
  \endpspicture
  \hspace{2em} 
  \raisebox{2em}{$\leadsto$}
  \hspace{2em}
  \psset{unit=1em}
  \pspicture(0,0)(7,4)
  \psline(0,0)(0,4)
  \psline(1,0)(1,3)
  \psline(2,0)(2,3)
  \psline(3,0)(3,2)
  \psline(4,0)(4,2)
  \psline(5,0)(5,1)
  \psline(6,0)(6,1)
  \psline(0,0)(7,0)
  \psline(0,1)(6,1)
  \psline(0,2)(4,2)
  \psline(0,3)(2,3)
  \rput(0,3.5){$\bullet$}
  \rput(2,2.5){$\bullet$}
  \rput(4,1.5){$\bullet$}
  \rput(6,0.5){$\bullet$}
  \rput(0.5,3){$\circ$}
  \rput(1.5,3){$\circ$}
  \rput(2.5,2){$\circ$}
  \rput(3.5,2){$\circ$}
  \rput(4.5,1){$\circ$}
  \rput(5.5,1){$\circ$}
  \rput(6.5,0){$\circ$}
  \endpspicture
\end{displaymath}

\begin{remark}
  For  $n$-cores $\mu \subset \nu$, if some transposition sequence
  from $\mu$ to $\nu$ touches rod $d$ then every transposition
  sequence from $\mu$ to $\nu$ touches rod $d$. This follows from the
  observation that any saturated chain from $\mu$ to $\nu$ can be
  obtained from any other by some sequence of interval exchanges, none
  of which may change which rods are touched.
\label{rmk:omit.d}
\end{remark}

\begin{proposition}\label{prop:interval.reduction}
  Let $\mu \subset \nu$ be $m+1$-cores such that some (equivalently,
  every) transposition sequence from $\mu $ to $\nu $ does not touch
  rod $d$. Then the interval $[\mu , \nu ] $ in the $m+1$-core poset
  is isomorphic to $[\mu^{(d)} , \nu^{(d)} ]$ in the $m$-core poset.
  This isomorphism extends to a bijection on skew strong tableaux
  which preserves the number of $i$-ribbons for each $i$.\end{proposition}

\begin{remark}
  Proposition~\ref{prop:interval.reduction} can be used in reverse:
  given $\mu <\nu $ in the $m$-core poset, we can lift the interval
  $[\mu, \nu ]$ to an isomorphic interval in the $m+1$-core posets
  with the same nice implications on strong tableaux.  This map is
  implemented by using the inverse procedure of adding in an extra rod
  between any two existing rods. This can be done precisely when the
  length of the inserted rod never has length weakly between the
  length of two interchanging rods, for instance, we may always take
  the rod to be longer than all other rods or shorter than all other
  rods.
\end{remark}

\begin{proof}
  Recall that exchanging rods in the $n$-rod abacus preserves the fact
  that the corresponding binary strings are balanced. Since the length
  of rod $d$ for each $m+1$-core $\lambda$ in the interval $[\mu,\nu]$
  is constant, the re-indexing for each $\lambda^{(d)}$ is the
  same. Further, since the covering relations in the $m+1$-core poset
  depend on rod $d$ only in the sense that it must not have 
  length weakly between that of the two exchanging rods, covering relations
  in mapping $[\mu,\nu]$ down to the $m$-core poset are preserved.
  Conversely, given any $m$-core $\gamma \in [\mu^{(d)},\nu^{(d)}]$,
  we can lift it to an $m+1$-core by reversing the procedure.  The
  reverse procedure also is injective and preserves containment order.
  Hence the intervals are isomorphic.


  The bijection on skew strong tableaux is obtained in the obvious
  way, by mapping the saturated chain
  \[
  S=(\mu =\mu_{0} \subset \mu_1 \subset \mu_2 \subset \cdots
  \subset \mu_k = \nu)
  \]
  to the chain 
  \[
  S^{(d)}=(\mu^{(d)} =\mu_{0}^{(d)} \subset \mu_1^{(d)} \subset
  \mu_2^{(d)} \subset \cdots \subset \mu_k^{(d)} = \nu^{(d)}).
  \]
  To see that this bijection preserves the number of $i$-ribbons,
  recall from Corollary~\ref{cor:iribbons} that the number of
  $i$-ribbons of a strong tableau is equal to the difference in length
  of the interchanging rods taking $\lambda_{i-1}$ to
  $\lambda_{i}$. Since the map from $m+1$-cores to $m$-cores preserves
  the relative lengths of all rods other than rod $d$, this number is
  clearly preserved.
\end{proof}

By Proposition~\ref{prop:interval.reduction}, the following map is
well defined.

\begin{definition}\label{def:flattening.ssts}
  Let $\mu \subset \nu$ be $m+1$-cores such that some (equivalently,
  every) transposition sequence from $\mu $ to $\nu $ does not touch
  rod $d$.  Define the \emph{flattening map}
  $$\fl_{d}:\SST(\nu/\mu,m+1) \longrightarrow \SST(\nu^{(d)}/\mu^{(d)},m)$$
  sending $S^{*} \in \SST(\nu/\mu,m+1)$ to the underlying strong
  tableau $S^{(d)}$ with the stars placed on each $i$-ribbon in such a
  way as to preserve the depth.
\end{definition}


Note that the flattening map does not, in general, preserve the spin
statistic because it can shorten the height of ribbons. 

\begin{proposition}\label{prop:flattening}
  Let $\mu \subset \nu$ be $m+1$-cores such that some transposition
  sequence from $\mu $ to $\nu $ does not touch rod $d$.  The
  flattening map $\fl_{d} : \SST(\nu/\mu,m+1) \longrightarrow
  \SST(\nu^{(d)}/\mu^{(d)},m)$ is a bijection preserving the signature
  of a starred strong tableau and it commutes with the involutions
  $\vp_{i}$ for all $1<i<\ell(\nu)-\ell(\mu)$.
\end{proposition}

\begin{proof}
  To see $\fl_{d}$ preserves the signature $\sigma (S^{*})$, recall
  from Definition~\ref{defn:transposition.sequences} and
  Corollary~\ref{cor:transposition.seq.bijection} that the content of
  $i^{*}$ is determined by an excess bead on the longer rod in the
  $i$th exchange on the $n$-rod abacus.  Since the relative order
  among the beads on the abacus is unchanged by the procedure in
  Definition~\ref{defn:flattening.cores}, the contents of
  $i^{*},(i+1)^{*}$ will form a decent in $\sigma(S^{*})$ if and only
  if there is a corresponding descent in
  $\sigma(\fl_{d}(S^{*}))$. This proves
  $\sigma(S^{*})=\sigma(\fl_{d}(S^{*}))$.

  To show $\fl_{d}(\vp_{i}(S^{*})) = \vp_{i}(\fl_{d}(S^{*}))$, simply
  note that the cases in the definition of $\vp_{i}$ depend only on the
  types of rod exchanges in the corresponding 3-interval of the $m+1$
  or $m$-core poset respectively.  But, the relative order among the
  endpoints of the exchanging rods and the isomorphism type of the
  interval are persevered by the flattening map.  Hence the flattening
  map and the involution commute.
\end{proof}

\begin{corollary}\label{cor:flat}
  Let $\mu \subset \nu$ be $n$-cores with $\nu$ lying $r$ ranks above
  $\mu$. Then for $m = 2r$, there exists $m$-cores $\widehat{\mu}
  \subset \widehat{\nu}$ such that
  there exists a bijection from $\SST(\nu/\mu,n)$ to
  $\SST(\widehat{\nu}/\widehat{\mu}, m)$ that preserves the signature
  and commutes with the involutions $\vp_{i}$ for all
  $1<i<\ell(\nu)-\ell(\mu)$.  Thus, the affine dual equivalence graphs of 
$\widehat{\nu}/\widehat{\mu}$ and $\nu/\mu$ are isomorphic as signed colored graphs. 

\end{corollary}

\subsection{The cloning map}
\label{sec:graph-cloning} 

Whereas flattening removes rows of the abacus, cloning adds
columns. Analogous to flattening, we will define cloning on starred
strong tableaux so that it preserves the signatures.  In some cases,
cloning commutes with the affine dual equivalence operators $\vp_i$.

\begin{definition}\label{defn:cloning.cores}
  For any $n$-core $\mu$, define $\mu_{(j)}$ to be the unique
partition associated to the abacus obtained by \emph{cloning} the
column of the $n$-rod abacus of $\mu$ containing positions $j,
j+1, \ldots, j+n-1$.  Specifically, let $\beta$ be the binary string
encoding the abacus for $\mu$.  Then $\mu_{(j)}$ is the abacus
associated to the string obtained from $\beta$ by inserting a copy of
the substring $\beta_{j},\beta _{j+1}, \ldots, \beta_{j+n-1}$ into the
abacus for $\mu$ between positions $j-1$ and $j$. 
\end{definition}

Cloning a column has the effect of extending some of the rods in the
$n$-rod abacus, hence $\mu_{(j)}$ is also an $n$-core.  To see the
effect of cloning on partitions, consider taking $(5,2,2)$ regarded as
a $4$-core and cloning the column beginning with content $0$.  This gives
$(5,2,2)_{(0)} = (7,4,4,2,2)$, as depicted below.

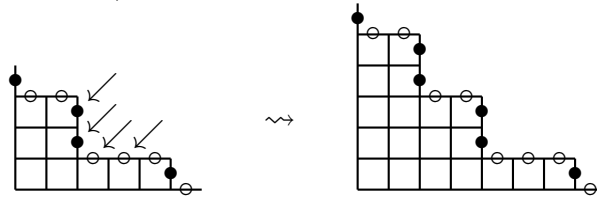
\begin{displaymath}
  \psset{unit=1em}
  \pspicture(0,0)(6,4)
  \psline(0,0)(0,4)
  \psline(1,0)(1,3)
  \psline(2,0)(2,3)
  \psline(3,0)(3,1)
  \psline(4,0)(4,1)
  \psline(5,0)(5,1)
  \psline(0,0)(6,0)
  \psline(0,1)(5,1)
  \psline(0,2)(2,2)
  \psline(0,3)(2,3)
  \rput(0,3.5){$\bullet$}
  \rput(2,2.5){$\bullet$}
  \rput(2,1.5){$\bullet$}
  \rput(5,0.5){$\bullet$}
  \rput(0.5,3){$\circ$}
  \rput(1.5,3){$\circ$}
  \rput(2.5,1){$\circ$}
  \rput(3.5,1){$\circ$}
  \rput(4.5,1){$\circ$}
  \rput(5.5,0){$\circ$}
  \rput(2.8,3.3){$\swarrow$}
  \rput(2.8,2.3){$\swarrow$}
  \rput(3.3,1.8){$\swarrow$}
  \rput(4.3,1.8){$\swarrow$}
  \endpspicture
  \hspace{2em} 
  \raisebox{2em}{$\leadsto$}
  \hspace{2em}
  \psset{unit=1em}
  \pspicture(0,0)(7,5)
  \psline(0,0)(0,6)
  \psline(1,0)(1,5)
  \psline(2,0)(2,5)
  \psline(3,0)(3,3)
  \psline(4,0)(4,3)
  \psline(5,0)(5,1)
  \psline(6,0)(6,1)
  \psline(7,0)(7,1)
  \psline(0,0)(8,0)
  \psline(0,1)(7,1)
  \psline(0,2)(4,2)
  \psline(0,3)(4,3)
  \psline(0,4)(2,4)
  \psline(0,5)(2,5)
  \rput(0,5.5){$\bullet$}
  \rput(2,4.5){$\bullet$}
  \rput(2,3.5){$\bullet$}
  \rput(4,2.5){$\bullet$}
  \rput(4,1.5){$\bullet$}
  \rput(7,0.5){$\bullet$}
  \rput(0.5,5){$\circ$}
  \rput(1.5,5){$\circ$}
  \rput(2.5,3){$\circ$}
  \rput(3.5,3){$\circ$}
  \rput(4.5,1){$\circ$}
  \rput(5.5,1){$\circ$}
  \rput(6.5,1){$\circ$}
  \rput(7.5,0){$\circ$}
  \endpspicture
\end{displaymath}
Observe that for fixed $\mu$, different values for $j$ can lead to the
same $n$-core $\mu_{(j)}$. For instance, taking any $j \in \{-4,\ldots,0
\}$ results in $(5,2,2)_{(j)} = (7,4,4,2,2)$.

In order for flattening to preserve a covering relation in the
$n$-core poset, the transposition sequence simply needs to avoid the
rod being removed. The situation for cloning is more subtle. Covering
relations are not always preserved even when the replicated column is
disjoint from the indexing transposition. It is immediate from
Proposition~\ref{p:rod.cover} that if $t_{r,s} \mu > \mu$ is a
covering relation in the $n$-core poset, then $(t_{r,s} \mu)_{(j)}$
covers $\mu_{(j)}$ in the $n$-core poset if and and only if for every
$r < h \leq s$ the relative order of the lengths of rods $h,r$ and $s$
is the same in both $\mu$ and $\mu_{(j)}$.  We call such a $j$ a
\emph{cloneable index} for $\mu \subset t_{r,s} \mu$.  More generally,
$j$ is a \emph{cloneable index} for the interval $[\mu ,\nu ]$
provided cloning the column beginning at $j$ of every core partition
in the interval results in another isomorphic interval in the $n$-core
poset.  This happens if and only if no rod in the n-rod abacus
representing any element in the interval has a rightmost bead of
content $j,j+1,\dotsc , j+n-1$.  Similarly, we say $j$ is a
\emph{cloneable index} for $S^{*} \in \SST(\nu /\mu,n) $ provided $j$
is a cloneable index for $[\mu ,\nu ]$.  The clone of $S^{*}$, denoted
$\mathrm{cl}_{j}(S^{*})$, is defined to be the saturated chain
obtained from $S$ by cloning, the column beginning with $j$ in each
$n$-rod abacus in the chain and leaving all the stars with content
less than $j$ at the same depth and increasing the depth by 1 for 
all stars with content at least $j$.  Note, all $i$-ribbons will have
the same shape in $S^{*}$ and $\mathrm{cl}_{j}(S^{*})$ since the
relative order of the rod lengths is unchanged by cloning a column.
See Figure~\ref{fig:squash} for example.

\begin{figure}[ht]
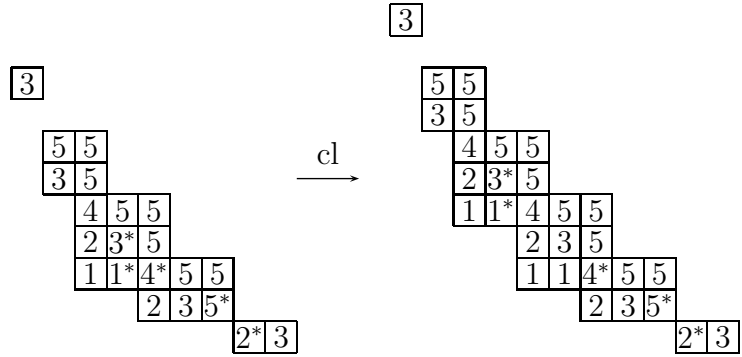

\begin{displaymath}
  \rnode{dolly}{%
      \tableau{%
        \\
        \\
        3 \\
        & \\
        & 5 & 5 \\
        & 3 & 5 \\
        &   & 4 & 5 & 5 \\
        &   & 2 & 3^* & 5 \\
        &   & 1 & 1^* & 4^* & 5 & 5 \\
        &   &   &   & 2 & 3 & 5^* \\
        &   &   &   &   &   &   & 2^* & 3}
  }
  \hspace{3em}
  \rnode{clone}{%
    \tableau{%
      3 \\
      & \\
      & 5 & 5 \\
      & 3 & 5 \\
      &   & 4 & 5 & 5 \\
      &   & 2 & 3^* & 5 \\
      &   & 1 & 1^* & 4 & 5 & 5 \\
      &   &   &   & 2 & 3 & 5 \\
      &   &   &   & 1 & 1 & 4^* & 5 & 5 \\
      &   &   &   &   &   & 2 & 3 & 5^* \\
      &   &   &   &   &   &   &   &   & 2^* & 3}
  }
  \psset{nodesep=0pt,linewidth=.1ex}
  \ncline[nodesepA=0em,nodesepB=1em]{->} {dolly}{clone} \naput{\displaystyle\cl}
\end{displaymath}
\caption{\label{fig:squash} An example of the cloning map on a starred
  strong tableau.}
\end{figure}

Observe that if $S^{*} \in \SST(\nu/\mu, n)$ and $j$ is a cloneable
index for $S^{*}$, then $j$ is a cloneable index for every other
starred strong tableau in $\SST(\nu/\mu, n)$ as well since the
definition of a cloneable index only depends on the interval $[\mu , \nu ]$.  

\begin{definition}\label{defn:squashing}
  Assume that $S^{*}$ has a cloneable index at $j$ and that
$T^{*}=\mathrm{cl}_{j}(S^{*}) \in \SST(\beta /\alpha ,n)$.  Define the \emph{cloning map} on
components
$$\mathrm{cl}_{j}:[S^{*}] \longrightarrow  \SST(\beta /\alpha ,n) 
$$ 
by cloning each starred strong tableaux in $[S^{*}]$ at the column beginning with $j$.  The
inverse map to cloning, when its defined, will be denoted by 
$$\mathrm{sq}_{j}:[T^{*}] \longrightarrow [S^{*}]
$$
and we call it the \textit{squashing map}.  
\end{definition}

As with the flattening map, the cloning map does not, in general,
preserve the spin statistic since it may alter the number of
$i$-ribbons and/or it may alter the depth of the starred
ribbons. Nonetheless, once the cloning map commutes with the $\vp_i$'s
on a connected component of an affine dual equivalence graph then we
can clone the same column any number of times and get an isomorphic
component.   

The following proposition is the analog of Proposition~\ref{prop:flattening}. 

\begin{proposition}\label{prop:cloning}
  Assume that $S^{*}\in \SST(\nu/\mu, n)$ has a cloneable index at $j$
and that $T^{*}=\mathrm{cl}_{j}(S^{*})$.  Further assume that cloning
the column beginning at $j$ commutes with the involutions $\vp_{i}$
for all $1<i<\ell(\nu)-\ell(\mu)$ on the component $[S^{*}]$.  Then
$j$ is a cloneable index for every starred strong tableaux in
$[T^{*}]$.  Moreover, if $U^{*}=\cl_{j}(T^{*})$, then $\cl_{j} :
[T^{*}] \longrightarrow [U^{*}]$ is a bijection preserving the
signature of each starred strong tableau and it commutes with the
involutions $\vp_{i}$ for all $1<i<\ell(\nu)-\ell(\mu)$.  Thus,
$[S^{*}] \approx [T^{*}] \approx [U^{*}]$ as signed, colored graphs.
\end{proposition}

\begin{proof}
The fact that $j$ is again a cloneable index for $T^{*}$ follows
directly from the characterization of $j$ being a cloneable index for
the interval containing $S^{*}$ in terms of rod lengths.

To see that $[T^{*}]$ is isomorphic to $[U^{*}]$ as signed colored
graphs, one must check that the affine dual equivalence maps
$\varphi_{i}$ commute with the cloning map from $[T^{*}]$ to
$[U^{*}]$.  This follows since the conditions for the affine dual
equivalence map on rank 3 intervals are unchanged at each step by
removing $n$ consecutive content diagonals that contains no head or
tail of a starred ribbon in any of the starred strong tableaux in
$[U^{*}]$ and that the cells in those $n$ diagonals must necessarily
be a copy of the next $n$-translate down.
\end{proof}

\begin{remark}\label{r:jdt} 
As a consequence of Proposition~\ref{prop:flattening} and
Proposition~\ref{prop:cloning}, we observe that the process of
flattening and squashing a component in an affine dual equivalence
graph as much as possible is similar to applying the necessary jeu da
taquin slides which bring together all of the connected components in
a skew tableaux by removing empty rows and columns.  Note, both
flattening and cloning/squashing can change the spin statistic even
when they commute with affine dual equivalence on a component.  Thus a
complete analog of jeu da taquin generalizing these moves would need
to keep track of powers of $t$ separately from the algorithm.
\end{remark}

\subsection{Local Schur positivity}\label{sub:lsp}

Our next goal is to show that there are only a small number of
isomorphism classes of connected components of rank $4$ affine dual
equivalence graphs.  Recall that a starred strong tableau $S^{*}$ on
an interval $[\mu ,\nu ]$ has ribbons labeled 1,2, \dots , $\ell(\nu
)-\ell(\mu )$.  We say $S^{*}$ has \textit{rank} $r$ provided
$r=\ell(\nu )-\ell(\mu )$.  The component $[S^{*}]$ of the affine dual
equivalence graph on $\nu /\mu$ has edges labeled 2,3,.., $r-1$ and
each vertex has a signature of length $r-1$.

\begin{lemma}\label{l:four.prime}
Let $S^{*} \in \SST_{n} (\nu /\mu )$ be a starred strong tableau of
rank $k=4$.  Then $[S^{*}]$ has a Schur positive generating function.
In fact, each such $[S^{*}]$ is either an isolated vertex, or a path with
either 2 or 4 edges with alternating color labels.  See Figure~\ref{fig:lambda4}.  
\end{lemma}

\begin{figure}[ht]
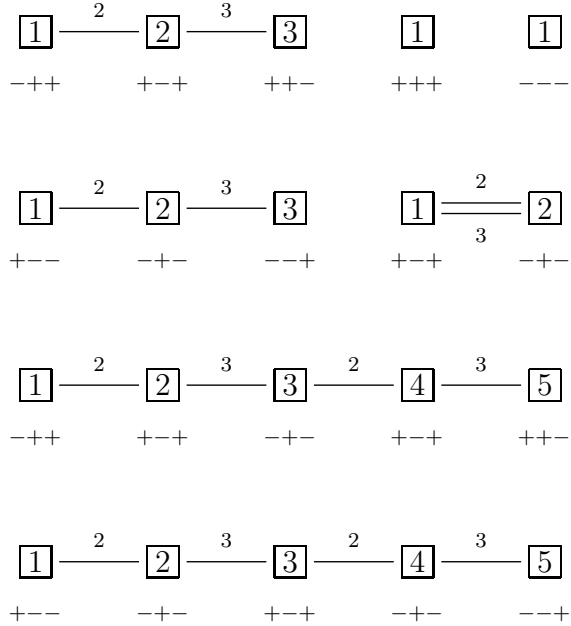

  \begin{displaymath}
       \begin{array}{ccccc}
         \stab{a}{1}{-++} \ & \
         \stab{b}{2}{+-+} \ & \
         \stab{c}{3}{++-} \ & \
         \stab{d}{1}{+++} \ & \
         \stab{e}{1}{---}
       \end{array} 
    \psset{nodesep=3pt,linewidth=.1ex}
    \everypsbox{\scriptstyle}
    \ncline            {a}{b} \naput{2}
    \ncline            {b}{c} \naput{3}
  \end{displaymath}\\[2\cellsize]
  \begin{displaymath}
       \begin{array}{ccccc}
         \stab{a}{1}{+--} \ & \
         \stab{b}{2}{-+-} \ & \
         \stab{c}{3}{--+} \ & \
         \stab{d}{1}{+-+} \ & \
         \stab{e}{2}{-+-}
       \end{array} 
    \psset{nodesep=3pt,linewidth=.1ex}
    \everypsbox{\scriptstyle}
    \ncline            {a}{b} \naput{2}
    \ncline            {b}{c} \naput{3}
    \ncline[offset=2pt]{d}{e} \naput{2}
    \ncline[offset=2pt]{e}{d} \naput{3}
  \end{displaymath}\\[2\cellsize]
  \begin{displaymath}
       \begin{array}{ccccc}
         \stab{a}{1}{-++} \ & \
         \stab{b}{2}{+-+} \ & \
         \stab{c}{3}{-+-} \ & \
         \stab{d}{4}{+-+} \ & \
         \stab{e}{5}{++-}
       \end{array} 
    \psset{nodesep=3pt,linewidth=.1ex}
    \everypsbox{\scriptstyle}
    \ncline            {a}{b} \naput{2}
    \ncline            {b}{c} \naput{3}
    \ncline            {c}{d} \naput{2}
    \ncline            {d}{e} \naput{3}
  \end{displaymath}\\[2\cellsize]
  \begin{displaymath}
       \begin{array}{ccccc}
         \stab{a}{1}{+--} \ & \
         \stab{b}{2}{-+-} \ & \
         \stab{c}{3}{+-+} \ & \
         \stab{d}{4}{-+-} \ & \
         \stab{e}{5}{--+}
       \end{array} 
    \psset{nodesep=3pt,linewidth=.1ex}
    \everypsbox{\scriptstyle}
    \ncline            {a}{b} \naput{2}
    \ncline            {b}{c} \naput{3}
    \ncline            {c}{d} \naput{2}
    \ncline            {d}{e} \naput{3}
  \end{displaymath}\\[2\cellsize]
  \caption{\label{fig:lambda4} All 7 possible isomorphism types of connected
    components of  affine dual equivalence graphs of rank $4$.}
\end{figure}

This lemma can be proved in two ways.  One approach is to do a
computer verification by identifying a set of dual equivalence classes
which contain all possible isomorphism types after flattening and
squashing as much as possible.  Details of this approach can be found
at \url{http://www.math.washington.edu/~billey/kschur/}.  The second
approach is based on the reading words of the starred strong tableaux,
see \cite{Assaf-proc}.

\begin{remark}
A computer exploration for all possible isomorphism types for affine
dual equivalence graphs of rank 5 is underway.  As of November of
2011, we have observed 326 distinct isomorphism types which can be
viewed in
\url{http://www.math.washington.edu/~billey/kschur/d-graphs-11-2011.pdf}.
Note for comparison, there are only 25 isomorphism types for rank 5 graphs for LLT
polynomials as defined in Section~\ref{sec:LLT}.
\end{remark}

\begin{theorem}
  For any pair of $n$-core partitions $\mu \subset \nu$, the affine
dual equivalence graph $\G_{\nu/\mu }^{(n)}$ is a \D graph which is locally
Schur positive for 2-colored edges and for which $\spin$ is constant on
connected components.  \label{thm:Dgraph}
\end{theorem}

\begin{proof}
  By Proposition \ref{prop:spin}, the involutions $\vp_{i}$ preserve
the spin statistic, hence $\spin$ is constant on connected components
of $\G_{\nu/\mu }^{(n)}$.

  To prove $\G_{\nu/\mu }^{(n)}$ is a \D graph, we must verify the axioms in
Definition~\ref{defn:D-graph}.  Axiom $1$ follows from
Theorem~\ref{thm:phi} where $\vp_i$ is shown to be an involution which
switches the sign appropriately.  Axioms $2$ and $5$ follow from the
fact that $\vp_i$ affects only $i-1,i$ and $i+1$-ribbons.  Axiom 3 and
the $\lsp_{2}$ property both follow from Lemma~\ref{l:four.prime}
since every connected component of $\G_{\nu/\mu }^{(n)}$ restricted to
$E_{i-1} \cup E_{i}$ is isomorphic to a component of a rank $4$ affine
dual equivalence graph replacing the edge labels 2,3 by $i-1,i$
respectively.
\end{proof}

Note that affine dual equivalence graphs need not satisfy Axiom 4 of
Definition~\ref{defn:deg}.  It is not known if affine dual equivalence
graphs satisfy Axiom 6.

\section{Connections with LLT and Macdonald polynomials}
\label{sec:LLT}

The primary interest in $k$-Schur functions originally was the
conjecture of Lapointe, Lascoux and Morse that these functions straddle
the gap between Macdonald polynomials and Schur functions. That is,
when $\mu$ is a $k$-bounded partition, they conjecture
\begin{equation}
  H_{\mu} (X; q,t) = \sum_{\nu \leq \mu} K^{(k)}_{\nu,\mu}
  (q,t) \ s^{(k)}_{\nu} (X; t),
  \label{eqn:H2kS}
\end{equation}
where $K_{\nu,\mu}^{(k)} (q,t) \in \mathbb{N}[q,t]$, and
\begin{equation}
  s^{(k)}_{\nu} (X; t) = \sum_{\lambda \leq \nu}
  C^{(k)}_{\lambda,\nu}(t) \ s_{\lambda}(X),
  \label{eqn:kS2s}
\end{equation}
where $C^{(k)}_{\lambda,\nu}(t) \in \mathbb{N}[t]$.

Using the definition of $k$-Schur functions advocated for in this
paper, we show how our methods
shed light on equation~(\ref{eqn:H2kS}), and, more generally, the
problem of expanding LLT polynomials into $k$-Schur functions.

\subsection{Macdonald polynomials}
\label{sec:LLT-mac}

The transformed Macdonald polynomials $\widetilde{H}_{\mu}(X;q,t)$
form a basis for symmetric functions with two additional
parameters. Precisely, $\{\widetilde{H}_{\mu}(X;q,t)\}$ is a basis for
$\Lambda$ with coefficients in $\mathbb{Q}(q,t)$. Macdonald
\cite{Macdonald1988} originally defined the polynomials to be the
unique functions satisfying certain orthogonality and triangularity
conditions. Haglund's monomial (quasisymmetric) expansion for
Macdonald polynomials \cite{Haglund2004,HHL2005} gives an explicit
combinatorial description of $\widetilde{H}_{\mu}(X;q,t)$ as the
$q,t$-generating function of permutations, regarded as standard
fillings of the diagram of $\mu$.

For a cell $x$ in the diagram of a partition $\mu$, let $l(x)$
(respectively $a(x)$) denote the number of cells directly north
(respectively east) of $x$. Given a permutation $w$ of $\{1,2,\ldots,
|\mu| \}$, fill the diagram of $\mu$ with $w$ written in one-line
notation so that $w$ becomes the row reading word of the resulting
filling. A {\em $\mu$-descent} of such a filling is a pair of cells
$(x,y)$ with $x$ immediately north of $y$ and the entry in $x$ is
greater than the entry in $y$. Denote by $\Des_{\mu}(w)$ the set of
all $\mu$-descents of $w$. Define the {\em major index with respect to
$\mu$} to be
\begin{equation}
  \maj_{\mu}(w) \; = \; \sum_{(x,y) \in \Des_{\mu}(w)} l(x) + 1. 
\label{eqn:Hmaj}
\end{equation}
Note that when $\mu$ is a single column, $\maj_{\mu}$ is the usual
major index on permutations.

An ordered pair of cells $(x,y)$ in the diagram of $\mu$ is called
{\em attacking} if $x$ and $y$ lie in the same row with $x$ strictly
west of $y$, or if $x$ is in the row immediately north of $y$ and $x$
lies strictly east of $y$.  Given a permutation filling of $\mu$, a
{\em $\mu$-inversion pair} is an attacking pair $(x,y)$ where the
entry of $x$ is greater than the entry of $y$. Denote by
$\Inv_{\mu}(w)$ the set of inversion pairs of $w$ filled into $\mu$;
this set is a subset of the usual inversion set for $w$. Define the
{\em inversion number with respect to $\mu$} to be
\begin{equation}
  \inv_{\mu}(w) \; = \; \left| \Inv_{\mu}(w) \right| - \sum_{(x,y) \in
    \Des_{\mu}(w)} a(x) .
\label{eqn:Hinv}
\end{equation}
Note that when $\mu$ is a single row, $\inv_{\mu}$ is the usual
inversion number on permutations.

For example, let $\mu$ be the partition $(5,4,4,1)$ and take 
$w =[5 \ 11 \ 14 \ 9 \ 2 \ 6 \ 3 \ 4 \ 10 \ 8 \ 1 \ 13 \ 7 \ 12]$ 
in $S_{14}$. Filling $w$ into $\mu$ gives
\begin{equation}
  \tableau{ 5 \\
    11 & 14 &  9 &  2 \\
    6 &  3 &  4 & 10 \\
    8 &  1 & 13 &  7 & 12}
\label{eqn:Mac-ex}
\end{equation}
Abusing notation, represent a cell of the filling by the entry which
it contains. The $\mu$-descent set of $w$ is
$$
\Des_{\mu}(w) = \left\{(11,6), \ (14,3), \ (3,1), \ (9,4), \ (10,7) \right\} ,
$$
and the $\mu$-inversion pairs of $w$ are given by
\begin{displaymath}
\Inv_{\mu}(w) = \left\{ \begin{array}{rrrrrr}
    (11,9), & (14,2), & (9,6), & ( 6,4), & (10,1), & (13,7), \\
    (11,2), & (14,6), & (9,3), & ( 4,1), & ( 8,1), & (13,12) \\
    (14,9), & ( 9,2), & (6,3), & (10,8), & ( 8,7), &
  \end{array} \right\} .
\end{displaymath}
Therefore the $\maj_{\mu}$ and $\inv_{\mu}$ statistics associated to
$w$ are
\begin{eqnarray*}
  \maj_{\mu}(w) & = &  2 + 1 + 2 + 1 + 2 = 8, \\
  \inv_{\mu}(w) & = & 17 - (3 + 2 + 2 + 1 + 0) = 9.
\end{eqnarray*}

\begin{remark}
  If $(x,y) \in \Des_{\mu}(w)$, then for every cell $z$ of the arm of
  $x$, the entry of $z$ is either bigger than the entry of $y$ or
  smaller than the entry of $x$ (or both). In the former case, $(z,y)$
  will form an inversion pair, and in the latter case, $(x,z)$ will
  form an inversion pair. Therefore $\inv_{\mu}(w)$ is a non-negative
  integer.
\label{rmk:Htriple}
\end{remark}

Define the signature function $\sigma: \Sn \rightarrow \{\pm
1\}^{n-1}$ on permutations by
\begin{equation}\label{eqn:sigma-w}
  \sigma_{i}(w) \; = \; \left\{ 
    \begin{array}{ll}
      +1 & \; \mbox{if $i$ lies left of $i+1$ in $w$} \\
      -1 & \; \mbox{if $i+1$ lies left of $i$ in $w$.}
    \end{array} \right\}
\end{equation}
For the permutation above, $\sigma(w) = -++-++--+-+--$ if we abbreviate
$\{-1,+1 \}$ by $\{-,+ \}$.  Using $\sigma$ to associate a
quasisymmetric function to each permutation filling of $\mu$,
Haglund's formula for $\widetilde{H}_{\mu}(X;q,t)$ may be stated as
follows.

\begin{definition}\cite{Haglund2004}
  The transformed Macdonald polynomials are given by
  \begin{equation}
    \widetilde{H}_{\mu}(X;q,t) = \sum_{w \in \Sn} q^{\inv_{\mu}(w)}
    t^{\maj_{\mu}(w)} Q_{\sigma(w)}(X) . 
    \label{eqn:haglund}
  \end{equation}
\end{definition}

It is a theorem in \cite{HHL2005} that \eqref{eqn:haglund} satisfies
the conditions which uniquely characterize the transformed Macdonald
polynomials as originally defined in \cite{Macdonald1988}. The proof
is by an elegant and elementary combinatorial argument, so we take
Haglund's formula as the definition.

A combinatorial proof of Macdonald positivity is given in
\cite{Assaf2007-3} by putting a \D graph structure on permutation
fillings of a partition diagram. In this case, the edges of the graph
are defined by simple involutions on the permutations. The
\emph{$i$-witness} is defined as usual to be the middle letter of
$i-1,i,i+1$ encountered when reading the permutation from left to
right. The main ingredients in the edges are usual dual equivalence on
permutations, denoted $d_i$, and a natural modification of dual
equivalence, denoted $\widetilde{d}_i$. Precisely, these involutions
are defined by
\begin{eqnarray}
  \cdots i \cdots i\pm 1 \cdots i\mp 1 \cdots
  &
  \stackrel{d_i}{\longleftrightarrow}
  &  \cdots i\mp 1\cdots i\pm 1\cdots i \cdots , \\
  \cdots i \cdots i\pm 1\cdots i\mp 1\cdots
  &
  \stackrel{\widetilde{d}_i}{\longleftrightarrow}
  & \cdots i\pm 1\cdots i\mp 1\cdots i \cdots .
\end{eqnarray}
Note that in the former case the $i$-witness always remains the same
while in the latter it always toggles between $i-1$ and $i+1$. The
edge-defining involutions of the \D graph for $\widetilde{H}_{\mu}(X;q,t)$ is
then given by
\begin{equation}
  \varphi^{\mu}_i(w) \; = \; \left\{
      \begin{array}{rl}
        w & \mbox{if $i$ is the $i$-witness},\\
        \widetilde{d}_i(w) & \mbox{if $i-1,i,i+1$ fit in} \hspace{8ex}
        \makebox[0pt]{%
          \psset{xunit=1em}
          \psset{yunit=1em}
          \pspicture(0,0)(6,2)
          \psline(0,0)(3.5,0) \psline(0,1)(6,1) \psline(2.5,2)(6,2)
          \psline(0,0)(0,1) \psline(1,0)(1,1)
          \psline(2.5,0)(2.5,2) \psline(3.5,0)(3.5,2)
          \psline(5,1)(5,2) \psline(6,1)(6,2)
          \rput(1.8,0.5){$\cdots$}
          \rput(4.3,1.5){$\cdots$}
          \endpspicture} \\[\cellsize]
        d_i(w) & \mbox{otherwise}.
     \end{array} \right.
   \label{eqn:pistol}
\end{equation}
A key observation in \cite{Assaf2007-3} is that $\varphi^{\mu}_i$
preserves Haglund's statistics $\maj_{\mu}$ and $\inv_{\mu}$. Much
like the case for starred strong tableaux, the proof that the
resulting graph satisfies the Axioms 1,2,3, and 5 is rather
straightforward.  Axiom 4' requires a simpler computer verification.

\subsection{LLT Polynomials}
\label{sec:LLT-llt}

We may also regard Haglund's formula for Macdonald polynomials as a
weighted sum over tableaux-like objects. In this paradigm, equation
\eqref{eqn:haglund} can be interpreted as giving a positive expansion
of $\widetilde{H}_{\mu}(X;q,t)$ in terms of certain \emph{LLT
  polynomials}. Lascoux, Leclerc and Thibon \cite{LLT1997} originally
defined $\LLT_{\boldsymbol{\lambda}} (X;q)$ to be the $q$-generating
function of $d$-ribbon tableaux of shape $\mu$ weighted by
cospin. Below we give another modified definition that is popular in
the literature as the $q$-generating function of $d$-tuples of
tableaux weighted by $d$-inversions, first presented in
\cite{HHLRU2005}. The equivalence of these definitions uses the abacus
model for taking $d$-cores and $d$-quotients of partitions
\cite{JaKe1981}; for further details on the correspondence in this
context, we refer the reader to \cite{Assaf2007-2,HHLRU2005}.

Let $\boldsymbol{\lambda}$ represent the $d$-tuple of (skew)
partitions $(\lambda^{(0)}, \ldots, \lambda^{(d-1)})$, each embedded
in a specific way in $\mathbb{N} \times \mathbb{N}$. For such a
$d$-tuple, define the \emph{shifted content} of a cell $x$ by
\begin{equation}
  \widetilde{c}(x) \; = \; d \cdot c(x) + i
\label{eqn:shifted-content}
\end{equation}
when $x$ is a cell of $\lambda^{(i)}$, where $c(x)$ is the usual
content of $x$ regarded as a cell of $\lambda^{(i)}$. Define the
\emph{bandwidth} of a $d$-tuple to be one plus the difference
between the largest and smallest \emph{unshifted} cell contents.

A \emph{standard $d$-tuple of shape $\boldsymbol{\lambda}$} is a
bijective filling of the cells of $\boldsymbol{\lambda}$ with the
letters $1$ to $m$ so that entries increase along rows and up columns.
For example, below is a standard $5$-tuple of shape $((3,2)/(1), \
(1,1,1), \ (2,1), \ (2,2)/(1), \ (1))$, say with the southeasternmost
cell of each partition embedded at content $0$.
\begin{equation}
  \tableau{  &    &   & & 14 & &   &    & &   &    & & \\
    5 & 11 &   & &  3 & & 9 &    & & 2 & 10 & & \\
    &  6 & 8 & &  1 & & 4 & 13 & &   &  7 & & 12}
\label{eqn:LLT-ex}
\end{equation}

Call a pair of cells $(x,y)$ in a $d$-tuple $\boldsymbol{\lambda}$
\emph{attacking} if $d>\widetilde{c}(y) - \widetilde{c}(x)>0$.  For
a standard $d$-tuple $\mathbf{T}$ of shape $\boldsymbol{\lambda}$,
define the \emph{number of $d$-inversions}, denoted
$\inv_d(\mathbf{T})$, to be the number of attacking pairs $(x,y)$ with
the entry of $x$ greater than the entry of $y$. For example, the
$5$-inversion pairs of the standard $5$-tuple above are
\begin{displaymath}
  \left\{ \begin{array}{rrrrrr}
    (11,9), & (14,2), & (9,6), & ( 6,4), & (10,1), & (13,7), \\
    (11,2), & (14,6), & (9,3), & ( 4,1), & ( 8,1), & (13,12), \\
    (14,9), & ( 9,2), & (6,3), & (10,8), & ( 8,7), &
  \end{array} \right\} .
\end{displaymath}
Note that if the $d$-tuple consists of $d$ single boxes,
i.e. $\boldsymbol{\lambda} = \left( (1), (1), \ldots, (1) \right)$,
each embedded to have content $0$, then $d$-inversions are simply the
usual inversions in the permutation obtained by reading the entries in
increasing order of shifted content.

For a $d$-tuple $\boldsymbol{\lambda}$, define the normalizing
constant $a_{\boldsymbol{\lambda}}$ to be the minimum number of
$d$-inversions of a standard $d$-tuple of shape
$\boldsymbol{\lambda}$. This normalization is an artifact of
$q$-counting by $d$-inversions rather than cospin; see
\cite{HHLRU2005}.

\begin{remark}
  Define an \emph{inversion triple} to be a triple of cells
  $(x,y,z)$ such that $x$ lies immediately north of $z$ and $y$ has
  shifted content between that of $x$ and $z$. Then both $(x,y)$ and
  $(y,z)$ are attacking pairs. Since $x$ is north of $z$, we must have
  $x>z$, and so at least one of $(x,y)$ and $(y,z)$ will be a
  $d$-inversion. Say that two inversions triples are overlapping if
  they have the form $(w,x,y)$ and $(x,y,z)$. Note that two
  overlapping inversion triples of this form may contribute only one
  $d$-inversion if $w<x$, $x>y$ and $y<z$. Therefore the normalizing
  constant $a_{\boldsymbol{\lambda}}$ is also equal to the maximum
  number of pairwise nonoverlapping inversion triples of
  $\boldsymbol{\lambda}$.
\label{rmk:triple}
\end{remark}

\begin{definition}
  The \emph{LLT polynomial} of shape $\boldsymbol{\lambda}$, denoted
  $\LLT_{\boldsymbol{\lambda}}$, is defined by
  \begin{equation}
    \LLT_{\boldsymbol{\lambda}}(X;q) \; = \; \sum_{\mathbf{T} \in
      \SYT_{d}(\boldsymbol{\lambda})} q^{\inv_d(\mathbf{T}) -
      a_{\boldsymbol{\lambda}}} Q_{\sigma(\mathbf{T})}(X) , 
    \label{eqn:llt}
  \end{equation}
  where the sum is over standard $d$-tuples of shape
  $\boldsymbol{\lambda}$, $a_{\boldsymbol{\lambda}}$ is the
  normalizing constant for $\boldsymbol{\lambda}$, and
  $\sigma(\mathbf{T})$ is defined analogously to equation
  \eqref{eqn:sigma} using shifted contents.
\end{definition}

The connection between Macdonald polynomials and LLT polynomials can
be seen by transforming the permutation fillings of the diagram of
$\mu$ with a given $\mu$-descent set into standard $\mu_1$-tuples of a
certain shape as follows. Let $D$ be a possible $\mu$-descent set. For
$i=1,\ldots,\mu_1$, let $\mu_D^{(i-1)}$ be the ribbon obtained from
the $i$th column of $\mu$ by putting the entry of cell $(i,j)$
immediately south of the entry of cell $(i,j+1)$ if $((i,j+1),(i,j))
\in D$ and immediately east otherwise. Embed each $\mu_D^{(i)}$ so
that the southeasternmost cell has content $0$ and, equivalently,
shifted content $i$. Then each permutation filling of shape $\mu$ with
$\Des_{\mu}=D$ may be regarded as a standard $\mu_1$-tuple of tableaux
of shape $\boldsymbol{\mu}_{D}$. For example, the filling of
$(5,4,4,1)$ in equation~\eqref{eqn:Mac-ex} corresponds to the standard
$5$-tuple given in equation~\eqref{eqn:LLT-ex}.

Since the major index statistic depends only on the $\mu$-descent set,
we may define the major index of a descent set $D$ of $\mu$ by
$\maj_{\mu}(D) = \maj_{\mu}(w)$ for any permutation $w$ with
$\Des_{\mu}(w) = D$. Similarly, define the arm of a descent set $D$ by
$a(D) = \sum_{(x,y) \in D} a(x)$. Comparing attacking pairs in both
paradigms leads to the following expansion of Macdonald polynomials in
terms of LLT polynomials.

\begin{theorem}\cite{HHL2005}
  Macdonald polynomials may be expressed in terms of LLT polynomials
  as
  \begin{equation}
    \widetilde{H}_{\mu}(X;q,t) = \sum_{D}
    t^{\maj_{\mu}(D)} \ q^{a_{\boldsymbol{\mu}_D}-a(D)} \
    \LLT_{\boldsymbol{\mu}_D}(X;q) ,
    \label{eqn:hag-llt}
  \end{equation}
  where the sum is over all possible $\mu$-descent sets $D$.
\end{theorem}

Note that $a(D)$ counts certain nonoverlapping inversion triples of
$\boldsymbol{\mu}_{D}$, hence by Remark~\ref{rmk:triple},
$a_{\boldsymbol{\mu}_D} \geq a(D)$. Therefore
equation~\eqref{eqn:hag-llt} gives a \emph{positive} expansion of
Macdonald polynomials in terms of LLT polynomials.

The theory of dual equivalence graphs is used in \cite{Assaf2007-2} to
establish LLT positivity, and the graph for Macdonald polynomials
presented in \cite{Assaf2007-3} appears as a special case. The graph
for LLT polynomials may be described in terms of the same elementary
operations, $d_i$ and $\widetilde{d}_i$, on standard $d$-tuples of
tableaux. Define the \emph{$i$-witness} of the dual equivalence for
$i-1,i,i+1$ to be whichever of $i-1,i,i+1$ has shifted content between
the other two. As it transpires, none of the three may have equal
shifted contents. The analog of dual equivalence for standard
$d$-tuples is given by 
\begin{equation}
  \varphi^{d}_i(w) \; = \; \left\{
      \begin{array}{rl}
        w & \mbox{if $i$ is the $i$-witness},\\[\smcellsize]
        \widetilde{d}_i(w) & \mbox{if $|\widetilde{c}(i) -
          \widetilde{c}(i\!-\!1)| \leq d$ and $|\widetilde{c}(i) -
          \widetilde{c}(i\!+\!1)| \leq d$}, \\[\smcellsize]
        d_i(w) & \mbox{otherwise}.
     \end{array} \right.
   \label{eqn:dist}
\end{equation}
It is shown in \cite{Assaf2007-2} that $\varphi^{d}_i$ preserves the
number of $d$-inversions and that the graph constructed from these
involutions is in fact a \D graph. A close inspection of equations
\eqref{eqn:pistol} and \eqref{eqn:dist} reveals that if $S$ is a
permutation filling of $\mu$ and $\mathbf{T}$ is the standard
$d$-tuple corresponding to $S$ via the bijection described above, then
$\varphi^{d}_{i}(\mathbf{T})$ is the standard $d$-tuple corresponding
to the permutation filling $\varphi^{\mu}_{i}(S)$ of $\mu$.

\subsection{Expansions into $k$-Schur functions}
\label{sec:LLT-equal}

Consider the case when the Macdonald polynomial
$\widetilde{H}_{\mu}(X;q,t)$ is equal to a single LLT polynomial
$\LLT_{\boldsymbol{\mu}_D}$. This happens when $\mu$ is a single row,
and so $\boldsymbol{\mu}_D = \left( (1),\ldots,(1) \right)$ each
embedded at content $0$. For this extreme case, an LLT polynomial will
have the most terms in the quasisymmetric expansion. Similarly, a
$k$-Schur function has the most terms in the quasisymmetric expansion
when $k$ is as small as possible, i.e. $n=k+1=2$. In both cases, the \D
graph will have no double edges and will always toggle the $i$-witness
across an $i$-edge.

More to the point, define a map $\theta$ from starred strong tableaux
on the $2$-core $(m,m-1,\ldots,1)$ to standard tableaux on the
$m$-tuple $((1),\ldots,(1))$ embedded so that each cell has content
$0$ as follows. Assuming relative positions for $1$ up to $i-1$ have
been chosen, reading from left to right, place $i$ in position
$d(i^*)$. Once all letters are placed, fill the permutation into the
$m$-tuple. For example,
\begin{displaymath}
  \raisebox{\cellsize}{\tableau{4 \\ 
    3^* & 4 \\   
    2 & 3 & 4^* \\
    1^* & 2^* & 3 & 4}}
  \stackrel{\displaystyle{\theta}}{\longrightarrow} \hspace{\cellsize}
  \tableau{3} \hspace{\cellsize} \tableau{1} \hspace{\cellsize} \tableau{4}
  \hspace{\cellsize} \tableau{2}.
\end{displaymath}

\begin{theorem}
  The map $\theta$ is a \D graph isomorphism between the graph of
  starred strong tableaux on the $2$-core $(m,m-1,\ldots,1)$ and the
  graph of standard filling of the $m$-tuple $((1), \ldots, (1))$ each
  embedded at content $0$. Furthermore, $q^{\binom{m}{2}} - \spin(S^*)
  = \inv(\theta(S^*))$.
\label{thm:staircase}
\end{theorem}

\begin{proof}
  Since $d(i^*)$ may be recovered for each $i$ from the permutation,
  $\theta$ is clearly a bijection on the underlying vertex
  sets. Moreover, $\theta$ will place $i$ to the left of $i-1$ if and
  only if $d(i^*) \leq d(i-1^*)$ which is the case if and only if
  $i^*$ lies on an earlier diagonal from $(i-1)^*$. Therefore $\theta$
  preserves the relative signatures. A similar analysis of $d(i-1^*)$
  and $d(i+1^*)$ reveals that the witness for the action on the
  staircase is precisely the witness for the action on the
  permutation. Since both actions toggle the witness, $\theta$
  commutes with the respective $i$-edges of the graphs, and hence is
  an isomorphism. Finally, adding $i$ at position $d(i^*)$ creates
  exactly $i-(d(i^*)+1)$ inversions.
\end{proof}

Motivated by Theorem~\ref{thm:staircase}, define the \emph{cospin}
of a starred strong tableau by
\begin{equation}
  \cospin(S^*) = \sum_{i} n(i) \cdot (w(i) - 1) + n(i) - \left(d(i^*)
    + 1 \right),
\label{eqn:cospin}
\end{equation}
where $w(i)$ is the \emph{width} of an $i$-ribbon. Define the
\emph{modified $k$-Schur functions}, denoted
$\widetilde{s}_{\lambda}^{(k)}(X;q)$, by
\begin{equation}
  \widetilde{s}_{\lambda}^{(k)}(X;q) = \sum_{S^* \in \SST(\rho(\lambda),n)}
  q^{\cospin(S^*)} Q_{\sigma(S^*)}(X) .
  \label{eqn:co-kschur}
\end{equation}
Here we have changed to the parameter $q$ in order to highlight
connections with LLT polynomials and Macdonald polynomials. Recall
from Section~\ref{sec:LLT-mac} that a Macdonald polynomial indexed by
a single row is precisely an LLT polynomial where each component is a
single cell. Therefore we may interpret this isomorphism of \D graphs
as the following symmetric function identity.

\begin{corollary}
  For $m \geq 1$, we have
  \begin{equation}
    \LLT_{(1),\ldots,(1)}(X;q) = \widetilde{H}_{(m)}(X;q,t) =
    \widetilde{s}^{(1)}_{(1^m)}(X;q).
    \label{eqn:LLT-Mac-kSchur}
  \end{equation}
  \label{cor:staircase}
\end{corollary}

Another illuminating case to consider is when an LLT polynomial is
equal to a single Schur function. It is easy to see from the
definition that this is the case exactly when the indexing tuple
consists of a single partition. Similarly, we have the following
characterization for $k$-Schur functions.

\begin{proposition}
  A $k$-Schur function is equal to a single Schur function if and only
  if the indexing partition has bandwidth at most $k$.
  \label{prop:1schur}
\end{proposition}

\begin{proof}
  To say $\lambda$ has bandwidth $k$ is to say that the rim of
  $\lambda$ consists of $k$ cells, and thus $\lambda$ is an $n$-core,
  where $n=k+1$. Moreover, in this case the $n$-core poset is
  isomorphic to Young's lattice. Therefore the strong tableaux of
  shape $\lambda$ are precisely the standard tableaux of shape
  $\lambda$, and the contribution to $\spin$ is nil. This argument is
  easily reversed.
\end{proof}

In both cases, the \D graphs will be the standard dual equivalence
graph for the indexing partition. On the level of the symmetric
functions, we have the following identity.

\begin{corollary}
  For $\lambda$ a partition with bandwidth at most $k$, we have
  \begin{displaymath}
    \LLT_{\lambda}(X;q) = s_{\lambda}(X) = \widetilde{s}^{(k)}_{\lambda}(X;q).
  \end{displaymath}
  \label{cor:1schur}
\end{corollary}

Corollaries~\ref{cor:staircase} and \ref{cor:1schur} support the
following, first conjectured by Mark Haiman \cite{HaimanPC}.

\begin{conjecture}
  Let $\boldsymbol{\lambda} = (\lambda^{(0)}, \ldots, \lambda^{(m-1)})$
  be an $m$-tuple of partitions of bandwidth at most $k$. Then
  \begin{equation}
    \LLT_{\boldsymbol{\lambda}}(X;q) \ = \
    \sum_{\mu} c^{(k)}_{\boldsymbol{\lambda},\mu}(q) \
    \widetilde{s}^{(k)}_{\mu}(X;q),  
  \end{equation}
  where $c^{(k)}_{\boldsymbol{\lambda},\mu}(q) \in \mathbb{N}[q]$.
  \label{conj:LLT-kSchur}
\end{conjecture}

Using the expansion of Macdonald polynomials into certain LLT
polynomials, this conjecture implies that a Macdonald polynomial
indexed by a partition with at most $k$ rows is $k$-Schur
positive. This statement can be reformulated to recover the original
conjecture of Lascoux, Lapointe and Morse \eqref{eqn:H2kS} by
interchanging $q$ and $t$ and conjugating the indexing partition.

\begin{corollary}
  Assuming Conjecture~\ref{conj:LLT-kSchur}, if $\mu$ is a partition
  with at most $k$ rows, then
  \begin{equation}
    \widetilde{H}_{\mu}(X;q,t) = \sum_{\lambda}
    \widetilde{K}^{(k)}_{\lambda,\mu}(q,t) \
    \widetilde{s}^{(k)}_{\lambda} (X;q),
  \end{equation}
  where $\widetilde{K}^{(k)}_{\lambda,\mu}(q,t) \in \mathbb{N}[q,t]$. 
  In particular, Macdonald polynomials are $k$-Schur positive.
\end{corollary}

%
\appendix
%

\section{Examples}
\label{app:examples}

In this appendix we give the quasisymmetric and Schur expansion for
the $k$-Schur function $s^{(2)}_{(2,2,1)}$. We compute this using the
interval $[\emptyset, (5,3,1)]$ of the $3$-core poset
(Figure~\ref{fig:interval531}) and the corresponding \D graph
on all starred strong tableaux of shape $(5,3,1)$ regarded as a
$3$-core (Figure~\ref{fig:dgraph531}).

\begin{figure}[ht]
\begin{displaymath}
    \rnode{0}{\emptyset} \hspace{2\cellsize}
    \rnode{1}{\smtableau{\e}} \hspace{3\cellsize}
    \begin{array}{c}
      \rnode{2}{\smtableau{\e & \e}} \\[4\cellsize]
      \rnode{11}{\smtableau{\e \\ \e}}
    \end{array} \hspace{4\cellsize}
    \begin{array}{c}
      \rnode{31}{\smtableau{\e \\ \e & \e & \e}} \\[4\cellsize]
      \rnode{211}{\smtableau{\e \\ \e \\ \e & \e}}
    \end{array} \hspace{4\cellsize}
    \begin{array}{c}
      \rnode{42}{\smtableau{\e & \e \\ \e & \e &\e & \e}} \\[4\cellsize]
      \rnode{311}{\smtableau{\e \\ \e \\ \e & \e & \e}}
    \end{array} \hspace{3\cellsize}
    \rnode{531}{\raisebox{\smcellsize}{\smtableau{\e \\ \e & \e & \e
          \\ \e & \e & \e &\e & \e}}} 
  \psset{nodesep=5pt,linewidth=.1ex}
  \ncline{531}{42}  \nbput{1,t,t^2}
  \ncline{531}{311} \naput{1,t}
  \ncline{42}{31}   \nbput{1,t}
  \ncline{311}{31}  \nbput{1}
  \ncline{311}{211} \naput{1}
  \ncline{31}{2}    \nbput{1,t}
  \ncline{31}{11}   \naput[labelsep=1pt]{1}
  \ncline{211}{2}   \nbput[labelsep=1pt]{t}
  \ncline{211}{11}  \naput{1,t}
  \ncline{2}{1}     \nbput{1}
  \ncline{11}{1}    \naput{1}
  \ncline{1}{0}     \nbput{1}
\end{displaymath}
\caption{\label{fig:interval531} The poset of $3$-cores lying below
  $(5,3,1)$, with edge weights giving the spin contributions of
  possible starrings.}
\end{figure}

\begin{eqnarray*}
  s^{(2)}_{(2,2,1)} & = &  
  Q_{-+--} 
  \ + \ Q_{--+-} 
  \ + \ t \ Q_{++--} 
  \ + \ (1 + t) \ Q_{+--+} 
  \ + \ (2t + t^2) \ Q_{-++-} \\ & &
  \ + \ (1 + 2t + t^2) \ Q_{+-+-} 
  \ + \ (1 + 2t + t^2) \ Q_{-+-+} 
  \ + \ t \ Q_{--++} \\ & &
  \ + \ (t^2 + t^3) \ Q_{+++-}
  \ + \ (t + 2t^2 + t^3) \ Q_{++-+} 
  \ + \ (t + 2t^2 + t^3) \ Q_{+-++} \\ & &
  \ + \ (t^2 + t^3) \ Q_{-+++}
  \ + \ t^4 Q_{++++} \\[2ex]
  & = & s_{(2,2,1)} + t s_{(3,1,1)} + (t + t^2)
  s_{(3,2)} + (t^2 + t^3) s_{(4,1)} + t^4 s_{(5)}
\end{eqnarray*}

\begin{figure}[ht]
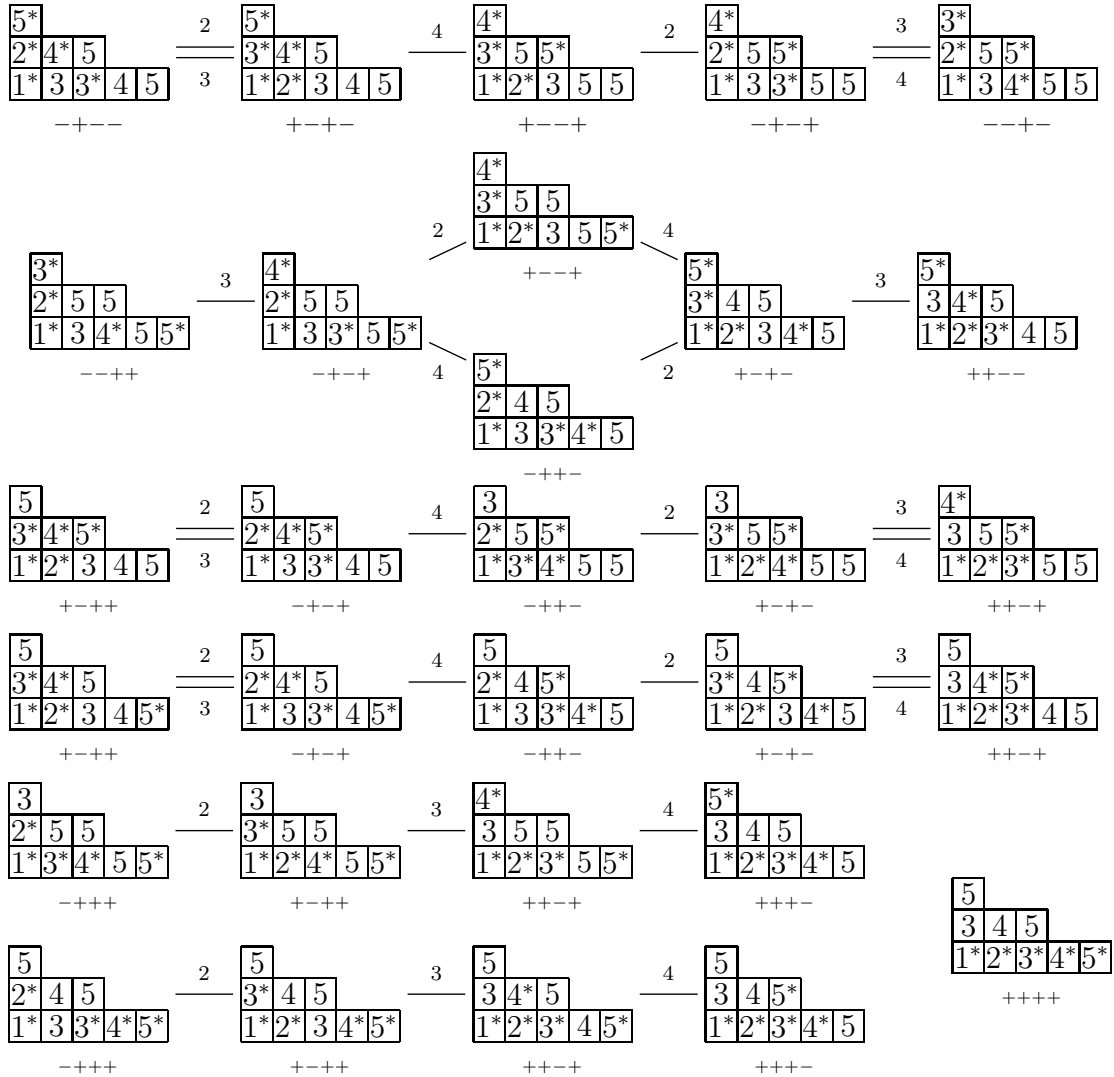

\begin{displaymath}
  \begin{array}{c}
    \begin{array}{ccccc}
      \stab{a}{5^* \\ 2^* & 4^* & 5 \\ 1^* & 3 & 3^* & 4 & 5}{-+--} \ & \
      \stab{b}{5^* \\ 3^* & 4^* & 5 \\ 1^* & 2^* & 3 & 4 & 5}{+-+-} \ & \
      \stab{c}{4^* \\ 3^* & 5 & 5^* \\ 1^* & 2^* & 3 & 5 & 5}{+--+} \ & \
      \stab{d}{4^* \\ 2^* & 5 & 5^* \\ 1^* & 3 & 3^* & 5 & 5}{-+-+} \ & \
      \stab{e}{3^* \\ 2^* & 5 & 5^* \\ 1^* & 3 & 4^* & 5 & 5}{--+-}
    \end{array} \\[2\cellsize]
    \psset{nodesep=3pt,linewidth=.1ex}
    \everypsbox{\scriptstyle}
    \ncline[offset=2pt] {a}{b} \naput{2}
    \ncline[offset=2pt] {b}{a} \naput{3}
    \ncline             {b}{c} \naput{4}
    \ncline             {c}{d} \naput{2}
    \ncline[offset=2pt] {d}{e} \naput{3}
    \ncline[offset=2pt] {e}{d} \naput{4}
    \begin{array}{ccccc}
      & & \stab{w}{4^* \\ 3^* & 5 & 5 \\ 1^* & 2^* & 3 & 5 & 5^*}{+--+} & & \\[-1\cellsize]
      \stab{u}{3^* \\ 2^* & 5 & 5 \\ 1^* & 3 & 4^* & 5 & 5^*}{--++} \ & \
      \stab{v}{4^* \\ 2^* & 5 & 5 \\ 1^* & 3 & 3^* & 5 & 5^*}{-+-+} & &
      \stab{y}{5^* \\ 3^* & 4 & 5 \\ 1^* & 2^* & 3 & 4^* & 5}{+-+-} \ & \
      \stab{z}{5^* \\ 3 & 4^* & 5 \\ 1^* & 2^* & 3^* & 4 & 5}{++--} \\[-1\cellsize]
      & & \stab{x}{5^* \\ 2^* & 4 & 5 \\ 1^* & 3 & 3^* & 4^* & 5}{-++-} & &
    \end{array} \\[2\cellsize]
    \psset{nodesep=3pt,linewidth=.1ex}
    \everypsbox{\scriptstyle}
    \ncline {u}{v}  \naput{3}
    \ncline {v}{w}  \naput{2}
    \ncline {v}{x}  \nbput{4}
    \ncline {w}{y}  \naput{4}
    \ncline {x}{y}  \nbput{2}
    \ncline {y}{z}  \naput{3}
    \begin{array}{ccccc}
      \stab{a2}{5 \\ 3^* & 4^* & 5^* \\ 1^* & 2^* & 3 & 4 & 5}{+-++} \ & \
      \stab{b2}{5 \\ 2^* & 4^* & 5^* \\ 1^* & 3 & 3^* & 4 & 5}{-+-+} \ & \
      \stab{c2}{3 \\ 2^* & 5 & 5^* \\ 1^* & 3^* & 4^* & 5 & 5}{-++-} \ & \
      \stab{d2}{3 \\ 3^* & 5 & 5^* \\ 1^* & 2^* & 4^* & 5 & 5}{+-+-} \ & \
      \stab{e2}{4^* \\ 3 & 5 & 5^* \\ 1^* & 2^* & 3^* & 5 & 5}{++-+}
    \end{array} \\[2\cellsize]
    \psset{nodesep=3pt,linewidth=.1ex}
    \everypsbox{\scriptstyle}
    \ncline[offset=2pt] {a2}{b2} \naput{2}
    \ncline[offset=2pt] {b2}{a2} \naput{3}
    \ncline             {b2}{c2} \naput{4}
    \ncline             {c2}{d2} \naput{2}
    \ncline[offset=2pt] {d2}{e2} \naput{3}
    \ncline[offset=2pt] {e2}{d2} \naput{4}
    \begin{array}{ccccc}
      \stab{a3}{5 \\ 3^* & 4^* & 5 \\ 1^* & 2^* & 3 & 4 & 5^*}{+-++} \ & \
      \stab{b3}{5 \\ 2^* & 4^* & 5 \\ 1^* & 3 & 3^* & 4 & 5^*}{-+-+} \ & \
      \stab{c3}{5 \\ 2^* & 4 & 5^* \\ 1^* & 3 & 3^* & 4^* & 5}{-++-} \ & \
      \stab{d3}{5 \\ 3^* & 4 & 5^* \\ 1^* & 2^* & 3 & 4^* & 5}{+-+-} \ & \
      \stab{e3}{5 \\ 3 & 4^* & 5^* \\ 1^* & 2^* & 3^* & 4 & 5}{++-+}
    \end{array} \\[2\cellsize]
    \psset{nodesep=3pt,linewidth=.1ex}
    \everypsbox{\scriptstyle}
    \ncline[offset=2pt] {a3}{b3} \naput{2}
    \ncline[offset=2pt] {b3}{a3} \naput{3}
    \ncline             {b3}{c3} \naput{4}
    \ncline             {c3}{d3} \naput{2}
    \ncline[offset=2pt] {d3}{e3} \naput{3}
    \ncline[offset=2pt] {e3}{d3} \naput{4}
    \begin{array}{lr}
    \begin{array}{cccc}
      \stab{h}{3 \\ 2^* & 5 & 5 \\ 1^* & 3^* & 4^* & 5 & 5^*}{-+++} \ & \
      \stab{i}{3 \\ 3^* & 5 & 5 \\ 1^* & 2^* & 4^* & 5 & 5^*}{+-++} \ & \
      \stab{j}{4^* \\ 3 & 5 & 5 \\ 1^* & 2^* & 3^* & 5 & 5^*}{++-+} \ & \
      \stab{k}{5^* \\ 3 & 4 & 5 \\ 1^* & 2^* & 3^* & 4^* & 5}{+++-}
    \end{array} 
    \psset{nodesep=3pt,linewidth=.1ex}
    \everypsbox{\scriptstyle}
    \ncline  {h}{i} \naput{2}
    \ncline  {i}{j} \naput{3}
    \ncline  {j}{k} \naput{4}
    & \hspace{.7\cellsize}
    \raisebox{-3\cellsize}{%
      $\stab{dot}{5 \\ 3 & 4 & 5 \\ 1^* & 2^* & 3^* & 4^* & 5^*}{++++}$}
    \\[-2\cellsize]
    \begin{array}{cccc}
      \stab{h2}{5 \\ 2^* & 4 & 5 \\ 1^* & 3 & 3^* & 4^* & 5^*}{-+++} \ & \
      \stab{i2}{5 \\ 3^* & 4 & 5 \\ 1^* & 2^* & 3 & 4^* & 5^*}{+-++} \ & \
      \stab{j2}{5 \\ 3 & 4^* & 5 \\ 1^* & 2^* & 3^* & 4 & 5^*}{++-+} \ & \
      \stab{k2}{5 \\ 3 & 4 & 5^* \\ 1^* & 2^* & 3^* & 4^* & 5}{+++-}
    \end{array} &
    \psset{nodesep=3pt,linewidth=.1ex}
    \everypsbox{\scriptstyle}
    \ncline  {h2}{i2} \naput{2}
    \ncline  {i2}{j2} \naput{3}
    \ncline  {j2}{k2} \naput{4}
  \end{array}
  \end{array}
\end{displaymath}
\caption{\label{fig:dgraph531}The \D graph on starred strong tableaux
  of shape $(5,3,1)$ regarded as a $3$-core.}
\end{figure}

%
%

\bibliographystyle{abbrv} 
\bibliography{references}

\end{document}